\documentclass[a4paper,12pt]{article}
\usepackage[dvipdfx,cmyk]{xcolor}
\definecolor{bookColor}{cmyk}{0 ,0 ,0 ,1} 
\color{bookColor} 
\usepackage{amsmath,amsthm,amssymb} 
\usepackage[hidelinks]{hyperref}
\usepackage{titlesec}
\titleformat{\part}
{\boldmath\huge\scshape\centering}{\thepart}{1em}{}
\titleformat{\section}
{\normalfont\Large\scshape\centering}{\thesection}{0.5em}{}

\titleformat{\subsection}
{\normalfont\large \scshape\centering}{\thesubsection}{0.3em}{}

\usepackage{xcolor}
\hypersetup{
	colorlinks,
	linkcolor={red},
	citecolor={blue}
}
\usepackage{graphicx}
\usepackage{caption}
\usepackage{float}
\usepackage{mathtools}
\usepackage{braket}
\usepackage{dsfont}

\usepackage{changepage}

\usepackage{marvosym}

\usepackage{tikz}
\usetikzlibrary{graphs,trees,snakes,automata,topaths,shapes,arrows}
\usetikzlibrary{positioning}

\makeatletter
\DeclareRobustCommand\widecheck[1]{{\mathpalette\@widecheck{#1}}}
\def\@widecheck#1#2{%
	\setbox\z@\hbox{\m@th$#1#2$}%
	\setbox\tw@\hbox{\m@th$#1%
		\widehat{%
			\vrule\@width\z@\@height\ht\z@
			\vrule\@height\z@\@width\wd\z@}$}%
	\dp\tw@-\ht\z@
	\@tempdima\ht\z@ \advance\@tempdima2\ht\tw@ \divide\@tempdima\thr@@
	\setbox\tw@\hbox{%
		\raise\@tempdima\hbox{\scalebox{1}[-1]{\lower\@tempdima\box
				\tw@}}}%
	{\ooalign{\box\tw@ \cr \box\z@}}}
\makeatother

\DeclareMathOperator{\Aut}{Aut}

\DeclareMathOperator{\Ind}{Ind}

\DeclareMathOperator{\Sym}{Sym}

\DeclareMathOperator{\proj}{proj}

\DeclareMathOperator{\Stab}{Stab}

\def\restriction#1#2{\mathchoice
	{\setbox1\hbox{${\displaystyle #1}_{\scriptstyle #2}$}
		\restrictionaux{#1}{#2}}
	{\setbox1\hbox{${\textstyle #1}_{\scriptstyle #2}$}
		\restrictionaux{#1}{#2}}
	{\setbox1\hbox{${\scriptstyle #1}_{\scriptscriptstyle #2}$}
		\restrictionaux{#1}{#2}}
	{\setbox1\hbox{${\scriptscriptstyle #1}_{\scriptscriptstyle #2}$}
		\restrictionaux{#1}{#2}}}
\def\restrictionaux#1#2{{#1\,\smash{\vrule height 1.2\ht1 depth 1.0\dp1}}_{\,#2}} 
\newcommand{\N}{\mathbb N}
\newcommand{\C}{\mathbb C}

\newcommand{\tg}[1]{\textbf{#1}}
\newcommand{\ub}[1]{\overline{#1}}

\newcommand{\es}{\varnothing}

\newcommand{\norm}[2]{\lVert #1 \lVert_{#2}}

\newcommand{\modu}[1]{\lvert#1\lvert}
\newcommand{\lb}{\lbrack}
\newcommand{\rb}{\rbrack}

\newcommand{\s}[2]{\sum\limits_{#1}^{#2}}
\newcommand{\li}[2]{\xrightarrow[#1\rightarrow#2]{}}

\newcommand{\prods}[2]{\langle #1,#2\rangle}

\newcommand{\restr}[2]{{% we make the whole thing an ordinary symbol
		\left.\kern-\nulldelimiterspace % automatically resize the bar with \right
		#1 % the function
		\vphantom{\big|} % pretend it's a little taller at normal size
		\right|_{#2} % this is the delimiter
	}}

\newcommand{\fct}[4]{\qq:\qq #1\qq\rightarrow\qq #2\qq :\qq #3\qq \mapsto\qq #4}

\newcommand{\q}{\quad}
\newcommand{\qq}{\mbox{ }}

\newcommand{\maxx}[1]{\underset{#1}{\mbox{max}}}

\newcommand{\supp}[1]{\underset{#1}{\mbox{sup}}}

\newcommand{\diff}{\text{\rm d}}

\newcommand{\Hr}[1]{\mathcal{H}_{#1}}

\newcommand{\Ch}{\text{\rm Ch}}

\newcommand{\Fix}{\text{\rm Fix}}

\newcommand*\circled[1]{\tikz[baseline=(char.base)]{
		\node[shape=circle,draw,inner sep=2pt] (char) {#1};}}

\theoremstyle{plain}
\newtheorem{theorem}{Theorem}
\newtheorem{proposition}[theorem]{Proposition}
\newtheorem{lemma}[theorem]{Lemma}

\newtheorem{corollary}[theorem]{Corollary}

\newtheorem{theoremletter}{Theorem}

\theoremstyle{definition}
\newtheorem{definition}[theorem]{Definition}
\newtheorem{remark}[theorem]{Remark}

\newtheorem{example}[theorem]{Example}

\theoremstyle{plain}
\newtheorem*{theorem*}{Theorem}
\newtheorem*{conjecture*}{Conjecture}
\newtheorem*{proposition*}{Proposition}
\newtheorem*{lemma*}{Lemma}
\newtheorem*{corollary*}{Corollary}

\theoremstyle{definition}
\newtheorem*{definition*}{Definition}
\newtheorem*{remark*}{Remark}
\newtheorem*{example*}{Example}

\numberwithin{theorem}{section}
\numberwithin{equation}{section}
\begin{document}

	\title{Unitary representations of totally disconnected locally compact groups satisfying Ol'shanskii's factorization}
\author{Lancelot Semal\footnote{
		F.R.S.-FNRS Research Fellow, email : lancelot.semal\MVAt uclouvain.be}\\
	\newline \\
	UCLouvain, 1348 Louvain-la-Neuve, Belgium\\}

\maketitle
	\begin{abstract}
		Inspired by Ol'shanskii's work, we provide an axiomatic framework to describe certain irreducible unitary representations of non-discrete unimodular totally disconnected locally compact groups. We then look at the applications to certain groups of automorphisms of locally finite trees and semi-regular right-angled buildings.
	\end{abstract}
		In this document topological groups are second-countable, locally compact groups are Hausdorff and the word ``representation" stands for strongly continuous unitary representation on a separable complex Hilbert space. 
	\newpage
	\tableofcontents
	\newpage
	\section{Introduction}\label{chaptitre intro}
	%$\mathcal{F}_\mathcal{S}$and \textbf{b.o.n.c.o.s.} for basis of neighbourhoods of the identity consisting of compact open subgroups.
	\subsection{Motivation}
	Despite the fact that the representation theory of locally compact groups has been an active domain of research since the 60's, this field of mathematics still contains vast uncharted territories. Similarly to finite groups the irreducible representations play a central role in the theory. However, by contrast to finite groups, the decomposition of a representation into irreducible representations is not always well behaved. Such a decomposition behaves `nicely' only for the so called type I groups (Bernstein and Kirillov used the terms `tame’ and `wild’ to qualify respectively type I and non-type I groups). Loosely speaking, type I groups are the locally compact groups for which the decomposition of any representation into a direct integral of irreducible representations is essentially unique. Moreover, the classification of the equivalence classes of the irreducible representations is known to be intractable unless the group is type I. Many statements address the problem to determine whether a group is type I or not. A result from Thoma \cite{Thoma1968} ensures that a discrete group is type I if and only if it is virtually abelian. For non-discrete groups, prominent examples of type I groups are provided by nilpotent connected Lie groups \cite{Dixmier1959}, reductive algebraic groups over non-Archimedean fields \cite{Harish1970}\cite{Bernshtein1974}, semisimple connected Lie groups \cite{Knapp1986} and reductive adelic groups \cite{Clozel2002}. Concerning groups acting on trees, which play a major role in the research on totally disconnected locally compact groups, the main source of examples comes from groups of automorphisms satisfying the Tits independence property (Definition \ref{equation prop ind de tits}) which are known to be type I if they act transitively on the boundary of the tree \cite{Ciobotaru2015}. This notable achievement was originally enlightened by Ol'shanskii who classified the irreducible representations of the full group of automorphisms of a $d$-regular tree with $d\geq 3$ \cite{Ol'shanskii1977} and concluded that this group is CCR \cite{Ol'shanskii1980}(hence type I). This classification is the starting point of our paper. We now recall the main ideas and refer to \cite{FigaNebbia1991} for a detailed expository.
	
	Let $T$ be a $d$-regular tree with $d\geq 3$. Let $V(T)$ and $E(T)$ denote respectively its set of vertices and edges and let $G\leq \Aut(T)$ be a closed subgroup of $\Aut(T)$. An irreducible representation $\pi$ of $G$ is called:
	\begin{itemize}\label{definition de spheric special cuspidal}
		\item \tg{spherical} if there exists a vertex $v\in V(T)$ such that $\pi$ admits a non-zero $\Fix_G(v)$-invariant vector where $\Fix_G(v)=\{g\in G\lvert gv=v\}$.
		\item \tg{special} if it is not spherical and there exists an edge $e\in E(T)$ such that $\pi$ admits a non-zero $\Fix_G(e)$-invariant vector where $\Fix_G(e)=\{g\in G\lvert gv=v\qq\forall v\in e\}$ is the fixator of the edge $e$.
		\item \tg{cuspidal} if it is neither spherical nor special.
	\end{itemize}
	It is clear from the definition that every irreducible representation of $G$ belongs to exactly one of these categories. If $G$ is non-compact and transitive on the boundary of the tree, a complete classification of the equivalence classes of both spherical and special representations is achieved without further hypothesis see \cite{FigaNebbia1991} and \cite{Matsumoto}. On the other hand, Ol'shanskii achieved a classification of the equivalence classes of cuspidal representations for groups of automorphisms for which the fixators of a certain family of subtrees satisfy a particular factorization (see Lemma \ref{independence figa nebbia 3.1}). Such a factorization occurs for instance for groups of automorphisms with the Tits independence property \cite{Amann2003}. Among other things, the resulting classification ensures that the cuspidal representations of such groups are induced from compact open subgroups. In particular, they are associated to compactly supported functions of positive type, they belong to the discrete series of $G$ and their respective equivalence classes are isolated points in the unitary dual for the Fell topology. 
	
	The purpose of this paper is to provide an axiomatic framework on non-discrete unimodular totally disconnected locally compact groups under which a similar description of certain irreducible representations can be achieved. This is developed further in the first part of the paper that is Chapters \ref{chaptitre intro} and \ref{section generalisation of Ol'shanskii machinery}. The purpose of the second part of the paper (that is Chapters \ref{Application to Aut T}, \ref{application IPk}, \ref{application IPV1} and \ref{Application to universal groups of right-angled buildings}) is to show how this framework applies to certain groups of automorphisms of trees and right-angled buildings. These applications include groups whose representation theory did not appear  in the literature yet such as groups of automorphisms of trees satisfying property \ref{IPk} \cite{BanksElderWillis2015} and universal groups of certain semi-regular right-angled buildings \cite{Universal2018}. 
	This axiomatic framework is also used in an ad-doc paper\cite{SemalR2022} to obtain a description of the irreducible representations of groups classified by Radu in \cite{Radu2017} and contribute to Nebbia's CCR conjecture on trees \cite{Nebbia1999} (since it leads to the conclusion that they are uniformly admissible).

	\subsection{The axiomatic framework}\label{section intro fertile}
	
	The purpose of this section is to establish the formalism of our axiomatic framework.
	Let $G$ be a locally compact group, let $\pi$ be a representation of $G$ and let $H$ be a closed subgroup of $G$. We denote by $\Hr{\pi}$ the representation space of $\pi$ and by $\Hr{\pi}^H$ the subspace of $H$-invariant vectors of $\Hr{\pi}$. We recall that $\pi$ is \tg{$H$-spherical} if the space $\Hr{\pi}^H$ is not reduced to $\{0\}$. The following Lemma shows that this notion carries a particular flavour among totally disconnected locally compact groups.
	\begin{lemma}\label{existence of U invariant vector}
		Let $\pi$ be a representation of $G$ and let $\mathcal{S}$ be a basis of neighbourhoods of the identity consisting of compact open subgroups. Then,
		\begin{equation*}
		\bigcup_{U\in \mathcal{S}}\Hr{\pi}^{U}= \{\xi \in \Hr{\pi}\mid \exists U\in \mathcal{S}\mbox{ such that }\qq \pi(h)\xi=\xi\qq \forall h\in U\}
		\end{equation*} 
		is dense in $\Hr{\pi}$.
	\end{lemma}	
	A proof of this result can be found in \cite[pg.$85$]{FigaNebbia1991} for the full group of automorphisms of a thick regular tree and the basis of neighbourhoods consisting of fixators of complete finite subtrees. Since their reasoning applies to the above setup and is probably well known by the experts, we do not give a proof in these notes. 
	
	We now recall that totally disconnected locally compact groups are characterised among locally compact groups by \textbf{Van Dantzig}'s Theorem which ensures that a locally compact group is totally disconnected if and only if it admits a basis of neighbourhoods of the identity consisting of compact open subgroups. In particular, the above lemma leads to the observation that every representation $\pi$ of a totally disconnected locally compact group $G$ admits non-zero invariant vectors for some compact open subgroup of $G$. Notice, however, that the statement loses its relevance for discrete groups. It is natural to ask whether further informations can be obtained from a more detailed analysis of those basis of neighbourhoods. The purpose of this paper is to provide a positive answer to that question. We now make few minor observations and establish our formalism.
	
	Let $G$ be a totally disconnected locally compact group and let $\pi$ be a representation of $G$ that admits a non-zero $U$-invariant vectors $\xi\in \Hr{\pi}$ for some $U\leq G$. A direct computation shows for every $g\in G$ that $\pi(g)\xi$ is non-zero $gUg^{-1}$-invariant vector of $\Hr{\pi}$. In particular, the existence of non-zero invariant vectors for a subgroup $U$ is an invariant property of the conjugacy class of $U$. With this in mind, we let $\mathcal{B}$ denote the set of compact open subgroups of $G$, $P(\mathcal{B})$ denote the power set of $\mathcal{B}$ and let
	$$\mathcal{C}: \mathcal{B}\rightarrow P(\mathcal{B})$$
	be the map sending a compact open subgroup to its conjugacy class in $G$. Let $\mathcal{S}$ be a basis of neighbourhoods of the identity consisting of compact open subgroups of $G$ and let $\mathcal{F}_{\mathcal{S}}=\{ \mathcal{C}(U)\lvert U\in \mathcal{S}\}$. We equip $\mathcal{F}_\mathcal{S}$ with the partial order given by the reverse inclusion of representatives ($\mathcal{C}(U)\leq \mathcal{C}(V)$ if there exists $\tilde{U}\in \mathcal{C}(U)$ and $\tilde{V}\in \mathcal{C}(V)$ such that $\tilde{V}\subseteq \tilde{U}$). For a poset $(P, \leq)$ and an element $x\in P$, we say that the \tg{height} of $x$ in $(P,\leq)$ is $L_x-1$ where $L_x$ is the maximal length of a strictly increasing chain in $P_{\leq x}=\{y\in P\lvert y\leq x\}$ if such a maximal length exists and we say that the height is infinite otherwise. 
	\begin{definition}
		A basis of neighbourhoods of the identity $\mathcal{S}$ consisting of compact open subgroups of $G$ is called a \tg{generic filtration} of $G$ if the height of every element in $\mathcal{F}_{\mathcal{S}}$ is finite.
	\end{definition} 
	The following lemma ensures the existence of generic filtrations for unimodular groups.
	\begin{lemma}\label{lemma bounded basis implies generic filtration}
		Let $G$ be a unimodular totally disconnected locally compact group and let $\mathcal{S}$ be a basis of neighbourhoods of the identity consisting of compact open subgroups of $G$ with uniformly bounded Haar measure. Then, $\mathcal{S}$ is a generic filtration of $G$.
	\end{lemma}
	\begin{proof}
		Since the elements of $\mathcal{S}$ have uniformly bounded measure, there exists a Haar measure $\mu$ on $G$ such that $0\lneq\mu(V)\leq 1$ for every $V\in \mathcal{S}$. Let $\mathcal{C}(U)\in \mathcal{F}_\mathcal{S}$. Since $G$ is unimodular, the measure $\mu(U)$ does not depend on the choice of representative $U\in \mathcal{C}(U)$. Now, let $\mathcal{C}(U_0)\leq\mathcal{C}(U_1)\leq ...\leq \mathcal{C}(U_{n-1})\leq \mathcal{C}(U)$ be a strictly increasing chain. Changing representatives if needed, we can suppose that $U\leq U_{n-1}\leq ... \leq U_1\leq U_0$. In particular, notice that
		$$ \lbrack U_0:U\rbrack =\lbrack U_0 : U_1\rbrack \cdots \lbrack U_{n-1}: U\rbrack\geq 2^n.$$
		On the other hand, since $U$ and $U_0$ are both compact open subgroups of $G$ observe that $$\lbrack U_0:U\rbrack= \frac{\mu(U_0)}{\mu(U)}\leq \frac{1}{\mu(U)}.$$
		This proves that $n\leq -\log_2(\mu(U))$ and therefore that the height of $\mathcal{C}(U)$ in $\mathcal{F}_\mathcal{S}$ is finite. 
	\end{proof}
	Every generic filtration $\mathcal{S}$ of $G$ splits as a disjoint union $\mathcal{S}=\bigsqcup_{l\in \N}\mathcal{S}\lb l \rb$ where $\mathcal{S}\lb l\rb$ denotes the set of elements $U\in \mathcal{S}$ such that $\mathcal{C}(U)$ has height $l$ in $\mathcal{F}_{\mathcal{S}}$. The element of $\mathcal{S}\lb l \rb$ are called the elements at \tg{depth} $l$. For every representation $\pi$ of $G$ there exists a smallest non-negative integer $l_\pi\in \N$ such that $\pi$ admits non-zero $U$-invariant vectors for some $U\in \mathcal{S}\lb l_\pi \rb$. This $l_\pi$ is called the \tg{depth} of $\pi$ with respect to $\mathcal{S}$ in analogy with the similar notion of depth for representations of reductive groups over non-Archimedean fields introduced by Moy-Prasad in \cite{Moy1996}.
	
	\begin{example}
		Let $T$ be a $d$-regular tree with $d\geq 3$. The full group of automorphisms $\Aut(T)$ of $T$ equipped with the permutation topology is a non-discrete unimodular totally disconnected locally compact group. We recall that a subtree $\mathcal{T}\subseteq T$ is called complete if all its vertices $v\in V(\mathcal{T})$ have degree $0,1$ or their degree coincides with the degree of $v$ in $T$. Let $\mathfrak{T}$ be the set of all complete finite subtrees of $T$ and let $\mathcal{S}=\{\Fix_{\Aut(T)}(\mathcal{T})\lvert \mathcal{T}\in \mathfrak{T}\}$ be the basis of neighbourhoods of the identity consisting of the groups
		$$\Fix_{\Aut(T)}(\mathcal{T})=\{g\in \Aut(T) \lvert \qq gv=v \qq \forall v\in V(\mathcal{T})\}$$
		where $V(\mathcal{T})$ denotes the set of vertices of $\mathcal{T}$. In Section \ref{section generic filtration of Aut(T)} we show that $\mathcal{S}$ is a generic filtration. This family of compact open subgroups is exactly the one used by Ol'shanskii to obtain a complete classification of the cuspidal representations of $\Aut(T)$. Now, let $\pi$ be an irreducible representation of $\Aut(T)$. According to the terminology introduced on page \pageref{definition de spheric special cuspidal} and as a consequence of Lemma \ref{Lemma la startification de S pour Aut(T)}, the depth of a representation $\pi$ with respect to $\mathcal{S}$ can be interpreted as follows:
		\begin{itemize}
			\item $\pi$ is a spherical representation of $\Aut(T)$ if and only if $l_\pi=0$. 
			\item $\pi$ is a special representation of $\Aut(T)$ if and only if $l_\pi=1$.
			\item $\pi$ is a cuspidal representation of $\Aut(T)$  if and only if  $l_\pi\geq 2$. Furthermore in that case $l_\pi-1$ is the smallest positive integer for which there exists a complete finite subtree $\mathcal{T}$ of  $T$ with that amount of interior vertices such that $\pi$ has a non-zero $\Fix_{\Aut(T)}(\mathcal{T})$-invariant vector.
		\end{itemize}
	\end{example}
	\noindent We now introduce \tg{Ol'shanskii's factorization}.  
	\begin{definition}\label{definition olsh facto}
		Let $G$ be a non-discrete unimodular totally disconnected locally compact group, let $\mathcal{S}$ be a generic filtration of $G$ and let $l$ be a strictly positive integer. We say that $\mathcal{S}$ \tg{factorizes at depth} $l$ if the following conditions hold:
		\begin{enumerate}
			\item For all $U\in\mathcal{S}\lb l \rb$ and every $V$ in the conjugacy class of an element of $\mathcal{S}$ such that $V \not\subseteq U$, there exists $W$ in the conjugacy class of an element of $\mathcal{S}\lb l -1\rb$ such that: $$U\subseteq W \subseteq V U=\{vu\lvert u\in U, v\in V\}.$$
			\item For all $U\in\mathcal{S}\lb l \rb$ and every $V$ in the conjugacy class of an element of $\mathcal{S}$, the set
			\begin{equation*}
			N_G(U, V)= \{g\in G \lvert g^{-1}Vg\subseteq U\}
			\end{equation*}
			is compact.
		\end{enumerate} 
		Furthermore, the generic filtration $\mathcal{S}$ of $G$ is said to \tg{factorize$^+$ at depth} $l$ if in addition for all $U\in\mathcal{S}\lb l\rb$ and every $W$ in the conjugacy class of an element of $\mathcal{S}\lb l-1\rb$ such that $U\subseteq W$ we have
		\begin{equation*}
		W\subseteq N_G(U,U) =\{g\in G \mid g^{-1}Ug\subseteq U\}.
		\end{equation*} 
		Since $G$ is unimodular, notice that the set $N_G(U,U)$ coincides with the normalizer $N_G(U)$ of $U$ in $G$. 
	\end{definition} 
	\begin{remark}
		The factorization at depth $l$ defined here depends on the entire generic filtration $\mathcal{S}$ and not only on the elements of $\mathcal{S}\lb l\rb$ and $\mathcal{S}\lb l-1\rb$.
	\end{remark}
	\begin{remark}
		Notice that the properties required in the above definition are satisfied for some $U\in \mathcal{S}\lb l\rb$ if and only if they are satisfied for each of its conjugate. In particular, a generic filtration $\mathcal{S}$ factorizes at depth $l$ if and only if the conditions \mbox{\rm 1} and  \mbox{\rm 2} are satisfied for every $U$ that is conjugate to an element of $\mathcal{S}\lb l \rb$. The same remark holds for the notion of factorization$^+$.
	\end{remark}

	\subsection{Main results and structure of the paper}
	This paper is divided in two parts. The purpose of the first part of the paper (Chapters \ref{chaptitre intro} and \ref{section generalisation of Ol'shanskii machinery}) is to establish our axiomatic framework and to prove the following Theorem.
	\begin{theoremletter}\label{la version paki du theorem de classification}
		Let $G$ be a non-discrete unimodular totally disconnected locally compact group and let $\mathcal{S}$ be a generic filtration of $G$ that factorizes at depth $l$. Then, every irreducible representation $\pi$ of $G$ at depth $l$ satisfies the following:
		\begin{enumerate}
			\item There exists a unique $C_\pi\in \mathcal{F}_\mathcal{S}=\{\mathcal{C}(U)\lvert U\in \mathcal{S}\}$ with height $l$ such that for all $U\in C_\pi$, $\pi$ admits a non-zero $U$-invariant vector.
			\item For every $U\in C_\pi$, $\pi$ admits a non-zero diagonal matrix coefficient supported in the compact open subgroup $N_G(U)$ of $G$. In particular, $\pi$ is induced from an irreducible representation of $N_G(U)$, belongs to the discrete series of $G$ and its equivalence class is isolated in the unitary dual $\widehat{G}$ for the Fell topology.
		\end{enumerate}
		Furthermore, if $\mathcal{S}$ factorizes$^+$ at depth $l$, there exists a bijective correspondence explicitly given by $\lb$Section \ref{Section the bijective correspondence of theorem A}, Theorem \ref{the theorem of classification for cuspidal representations} $\rb$ between the equivalence classes of irreducible representations of $G$ at depth $l$ such that $C_\pi=\mathcal{C}(U)$ and the equivalence classes of a family of irreducible representations of the finite group $N_G(U)/U$ called $\mathcal{S}$-standard representations. 
	\end{theoremletter}
	\noindent We refer to Section \ref{Section the bijective correspondence of theorem A} for more details about $\mathcal{S}$-standard representations and this bijective correspondence. We have a few remarks to make.
	\begin{remark}
		Theorem \ref{la version paki du theorem de classification} does not guarantee the existence of irreducible representations at depth $l$. Hence, certain generic filtrations of $G$ could just lead to the conclusion that irreducible representations having particular non-zero invariants do not exist.
	\end{remark}
	\begin{remark}
		Different generic filtrations might factorize simultaneously and lead to a different description of the same representations. A concrete example of this phenomenon is given on the full group of automorphisms of a $(d_0,d_1)$-regular tree in Chapters \ref{Application to Aut T}, \ref{application IPk} and \ref{application IPV1} if $d_0\not=d_1$. Furthermore, different generic filtrations might also describe different sets of irreducible representations. A concrete example of this phenomenon occurs for instance if we replace a generic filtration $\mathcal{S}$ of $G$ that factorizes at all positive depth with the generic filtration $\mathcal{S}'=\mathcal{S}-\mathcal{S}\lb 0 \rb$. In that case, $\mathcal{S}'\lb l'\rb =\mathcal{S}\lb l'+1\rb$ $\forall l'\in \N$. In particular, Theorem \ref{la version paki du theorem de classification} does not describe the irreducible representations admitting non-zero invariant vectors for a subgroup $U\in \mathcal{S}\lb 1\rb$ when applied $\mathcal{S}'$ while it does when applied to $\mathcal{S}$.
	\end{remark} 
	
	In the second part of the paper, we look at applications of the axiomatic framework developed in Chapters \ref{chaptitre intro} and \ref{section generalisation of Ol'shanskii machinery}. In Chapter \ref{Application to Aut T} we recover the classification of the cuspidal representations of the full group of automorphisms of a semi-regular tree made by Ol'shanskii \cite{Ol'shanskii1977} and later by Amann \cite{Amann2003}. In particular, the content of this chapter is redundant from the point of view of new results. It serves instead as a chapter allowing the reader to understand and interpret the abstract framework developed in Chapters \ref{chaptitre intro} and \ref{section generalisation of Ol'shanskii machinery} in a concrete and well understood case. 
	
	The purpose of Chapter \ref{application IPk} is to explain how our axiomatic framework can be applied to closed non-discrete unimodular subgroups of the group of automorphisms $\Aut(T)$ of a $(d_0,d_1)$-semi-regular tree satisfying the property \ref{IPk}(Definition \ref{definition IPk}) as defined in \cite{BanksElderWillis2015}. Loosely speaking, this property ensures that the pointwise fixator of any ball of radius $k-1$ around an edge decomposes as a direct product of the subgroup fixing all vertices on one side of the edge and the subgroup fixing all vertices on the other side. The main contributions of this chapter are Theorems \ref{thm B1} and \ref{THM B} below which provide two ways to build generic filtrations that factorizes$^+$. To be more precise, let $T$ be a $(d_0,d_1)$-semi-regular tree with $d_0,d_1\geq 3$ and set of vertices $V(T)$. For every finite subtree $\mathcal{T}$ of $T$ and each integer $r\geq 0$, we denote by $\mathcal{T}^{(r)}$ the ball of radius $r$ around $\mathcal{T}$ for the natural metric $d_T$ on $V(T)$ that is
	$$\mathcal{T}^{(r)}=\{v\in V(T)\lvert \exists w\in V(\mathcal{T})\qq\mbox{s.t.}\qq d_T(v,w)\leq r\}.$$ 
	\begin{theoremletter}\label{thm B1}
		Let $T$ be a $(d_0,d_1)$-semi-regular tree with $d_0,d_1\geq 3$ and let $G\leq \Aut(T)$ be a closed non-discrete unimodular subgroup satisfying the property \ref{IPk} for some integer $k\geq1$. Let $ \mathcal{P}$ be a complete finite subtree of $T$ containing an interior vertex, let $\Sigma_{\mathcal{P}}$ be the set of maximal complete proper subtrees of $ \mathcal{P}$ and let 
		$$ \mathfrak{T}_{\mathcal{P}}=\{\mathcal{R}\in \Sigma_{\mathcal{P}}\lvert \Fix_G((\mathcal{R}')^{(k-1)})\not\subseteq \Fix_G(\mathcal{R}^{(k-1)})\qq \forall \mathcal{R}'\in \Sigma_{\mathcal{P}}-\{\mathcal{R}\}\}.$$ 
		We suppose that: \begin{enumerate}
			\item $\forall \mathcal{R},\mathcal{R}'\in \mathfrak{T}_{ \mathcal{P}}, \forall g\in G$, we do not have $\Fix_G(\mathcal{R}^{(k-1)})\subsetneq\Fix_G(g(\mathcal{R}')^{(k-1)})$.
			\item For all $\mathcal{R}\in \mathfrak{T}_{ \mathcal{P}}$, $\Fix_G( \mathcal{P}^{(k-1)})\not=\Fix_G(\mathcal{R}^{(k-1)})$. 
			Furthermore, if $\Fix_G( \mathcal{P}^{(k-1)})\subsetneq\Fix_G(g\mathcal{R}^{(k-1)})$ we have $ \mathcal{P}\subseteq g\mathcal{R}^{(k-1)}$. 
			\item $\forall n\in \N,\forall v\in V(T)$, $\Fix_G(v^{(n)})\subseteq \Fix_G( \mathcal{P}^{(k-1)})$ implies $ \mathcal{P}^{(k-1)}\subseteq v^{(n)}$.  
			\item For every $g\in G$ such that $g \mathcal{P}\not= \mathcal{P}$, $\Fix_G( \mathcal{P}^{(k-1)})\not= \Fix_G(g \mathcal{P}^{(k-1)})$.
		\end{enumerate}  Then, there exists a generic filtration $\mathcal{S}_{ \mathcal{P}}$ of $G$ that factorizes$^+$ at depth $1$ such that 
		$$\mathcal{S}_{ \mathcal{P}}\lb 0\rb =\{\Fix_G(\mathcal{R}^{(k-1)})\lvert \mathcal{R}\in \mathfrak{T}_{ \mathcal{P}}\}$$
		$$\mathcal{S}_{ \mathcal{P}}\lb 1\rb=\{\Fix_G( \mathcal{P}^{(k-1)})\}.$$
	\end{theoremletter}
	\noindent The author would like to underline how realistic the assumptions of the above Theorem are. Indeed, any closed non-discrete unimodular subgroup $G\leq \Aut(T)$ satisfying the property \ref{IPk} and such that $\Fix_G(\mathcal{T})$ does not admit any fixed point other than the vertices of $\mathcal{T}$ for every complete finite subtree $\mathcal{T}$ of $T$ satisfies the hypotheses of Theorem \ref{thm B1}. In light of Theorem \ref{la version paki du theorem de classification} this allows one to describe the irreducible representations of $G$ which admit non-zero $\Fix_G(\mathcal{P}^{(k-1)})$ invariant vectors but do not admit non-zero $\Fix_G(\mathcal{R}^{(k-1)})$-invariant vectors for any $\mathcal{R}\in \mathfrak{T}_{\mathcal{P}}$. Notice furthermore that the theorem can be applied inductively with different $\mathcal{P}$. This is done for instance in Example \ref{example fixateur point au bords}. 
	
	Under a stronger hypothesis on $G$ (the hypothesis \ref{Hypothese Hq} (Definition \ref{definition de Hq})), we are even able to explicit a generic filtration that factorizes$^+$ at all sufficiently large depths. To be more precise, we have the following result.
	\begin{theoremletter}\label{THM B}
		Let $T$ be a $(d_0,d_1)$-semi-regular tree with $d_0,d_1\geq 3$ and let $G\leq \Aut(T)$ be a closed non-discrete unimodular subgroup satisfying the hypothesis \ref{Hypothese Hq} and the property \ref{IPk} for some integers $q\geq 0$ and $k\geq 1$. Then, there exists a generic filtration $\mathcal{S}_{q}$ of $G$ that factorizes$^+$ at all depth $l\geq L_{q,k}$ where $$L_{q,k}=\begin{cases}
		\max\{1,2k-q-1\} \qq& \mbox{if } q\mbox{  is even.} \\
		\max\{1,2k-q\} & \mbox{if } q\mbox{  is odd.} 
		\end{cases}$$
	\end{theoremletter}
	\noindent The generic filtration $\mathcal{S}_q$ is explicitly described on page \pageref{page Sq filtration}. For concrete applications of this theorem, we refer to Examples \ref{example autT pour changer}, \ref{example Radu} and \ref{example fixateur point au bords}. Furthermore, we treat some existence criteria for the irreducible representations at depth $l$ with respect to $\mathcal{S}_{q}$ in Section \ref{existence of representaions ot depth l for IPk}. Among the  applications that are new to the literature, we have the following (Corollary \ref{corollary this apply to radu group}) where $\mathcal{S}_0$ is $\mathcal{S}_q$ with $q=0$.
	\begin{corollary*}
		Let $T$ be a $(d_0,d_1)$-semi-regular tree with $d_0,d_1\geq 6$ and let $G\leq \Aut(T)$ be a closed subgroup acting $2$-transitively on the boundary $\partial T$ and whose local action at each vertex contains the alternating group of corresponding degree. Then, there exists a generic filtration $\mathcal{S}_{0}$ of $G$  and a constant $k\in \N$ such that the $\mathcal{S}_0$ factorizes$^+$ at all depth $l\geq2k-1$. 
	\end{corollary*}
	\noindent The groups appearing in the above corollary were extensively studied by Radu in \cite{Radu2017} and we refer to them as \textbf{Radu groups}. Among other things, Radu completely classified them and proved that each Radu group $G$ is $k$-closed for some constant $k\in \N$ depending $G$. Since Radu groups also satisfies the hypothesis $H_0$ and since $k$-closed groups satisfy the property \ref{IPk}, the corollary follows. Since this family of groups plays a central role in Nebbia's CCR conjecture on trees \cite{Nebbia1999} it is natural to ask weather more can be said. The following Theorem provides a positive answer without relying on the property \ref{IPk}. 
	\begin{theorem*}[\cite{SemalR2022}]
		Let $T$ be a $(d_0,d_1)$-semi-regular tree with $d_0,d_1\geq 6$ and let $G\leq \Aut(T)$ be a simple Radu group. Then, the generic filtration $\mathcal{S}_0$ factorizes$^+$ at all positive depth. 
	\end{theorem*}
 	\noindent In particular, this factorization leads to a description of the cuspidal representations of any simple Radu group. However, the proof of this theorem is quite technical, relies heavily on Radu's classification and is unrelated to the property \ref{IPk}. Furthermore, since various important consequences such as the description of the cuspidal representations of non-simple Radu groups and the progress related to Nebbia's CCR conjecture on trees need to be tackled in light of the result, the author decided to present a proof in an other paper (see \cite{SemalR2022}).
	
	The purpose of Chapter \ref{application IPV1} is to explain how our axiomatic framework can be applied to closed non-discrete unimodular subgroups of the group of type preserving automorphisms of a locally finite tree $T$ satisfying the property \ref{IPV1}. This serves as a preamble for Chapter \ref{Application to universal groups of right-angled buildings} were we show that the universal groups of certain semi-regular right-angled buildings can be realised as such groups. The main result of Chapter \ref{application IPV1} is the following.
	\begin{theoremletter}\label{theorem Ipv1 letter}
		Let $T$ be a locally finite tree, let $G$ be a closed non-discrete unimodular group of type preserving automorphisms of $T$ satisfying the property \ref{IPV1}(Definition \ref{defintion IPV1}) and the hypothesis \ref{Hypothese HV1}(Definition \ref{definition H v1}). Then, there exists a generic filtration $\mathcal{S}_{V_1}$ of $G$ that factorizes$^+$ at all depth $l\geq 1$. 
	\end{theoremletter} \noindent The generic filtration $\mathcal{S}_{V_1}$ is defined on page \pageref{page de SV1 yesssouille} and an existence criterion for the irreducible representations at depth $l$ with respect to $\mathcal{S}_{V_1}$ is given in Section \ref{existence for IPV1}.\\
	
	The purpose of Chapter \ref{Application to universal groups of right-angled buildings} is to prove that the universal groups of certain semi-regular right-angled buildings as introduced in \cite{Universal2018} can be realised as groups of type-preserving automorphisms of a locally finite tree $T$ in such a way that they satisfy the hypotheses of Theorem \ref{theorem Ipv1 letter}. We refer to Section \ref{section application right angle preliminaire} for details about the following notions. We let $(W,I)$ be a finitely generated right-angled Coxeter system and suppose that $I$ can be partitioned as $I=\bigsqcup_{k=1}^rI_k$ in such a way that $I_k=\{i\}\cup \{i\}^\perp$ for every $i\in I_k$ and for all $k=1,...,r$ where $\{i\}^\perp=\{j\in I\lvert ij=ji\}$. In particular, $W$ is virtually free and isomorphic to a free product $W_1*W_2*....*W_r$ where each of the $W_k$ is a direct product of finitely many copies of the group of order $2$. Let $(q_i)_{i\in I}$ be a set of integers bigger than $2$ and let $\Delta$ be a semi-regular building of type $(W,I)$ and prescribed thickness $(q_i)_{i\in I}$. Let  $(h_i)_{i\in I}$ be a set of legal coloring of $\Delta$, let $Y_i$ be a set of cardinal $q_i$ and let $G_i\leq \Sym(Y_i)$ for every $i\in I$. 
	\begin{theoremletter}\label{Theorem E}
		 There exists a locally finite tree $T$ such that the universal group $\mathcal{U}((h_i,G_i)_{i\in I})$ embeds as a closed subgroup satisfying the property \ref{IPV1} of the group  $\Aut(T)^+$ of type-preserving automorphisms of $T$. Furthermore, if $G_i$ is \textit{2}-transitive on $Y_i$ for every $i\in I$, the group of automorphisms of tree corresponding to $\mathcal{U}((h_i,G_i)_{i\in I})$ satisfies the hypothesis \ref{Hypothese HV1} and is unimodular.
	\end{theoremletter} 
	Together with Theorem \ref{theorem Ipv1 letter}, this proves that every non-discrete universal group $\mathcal{U}((h_i,G_i)_{i\in I})$ with \textit{2}-transitive $G_i$ admits a generic filtration that factorizes$^+$ at strictly positive depth. \\
	\subsection*{Acknowledgements}
	I express my deepest thank to Pierre-Emmanuel Caprace for his uncountable contributions to my work. This paper would not exist without him. I am also very grateful to Sven Raum for the interest he manifested to my work and to the anonymous referee for all the precious comments on the earlier versions of this paper which drastically improved the readability of this manuscript. Needless to say that the remaining inaccuracies are entirely mine.   
	\newpage
	\part{Ol'shanskii's factorization}\label{Part of the article first part}
	\section{Proof of Theorem \ref{la version paki du theorem de classification}}\label{section generalisation of Ol'shanskii machinery}
	\subsection{Direct consequences of Ol'shanskii's factorization}
	Let $G$ be a non-discrete unimodular totally disconnected locally compact group, let $\mu$ be a Haar measure on $G$ and let $\mathcal{S}$ be a generic filtration of $G$. This section explores the first consequences of a factorisation of $\mathcal{S}$ at depth $l$. We begin with the following key lemma.
	
	%´on particular function spaces which will appear later to be the space of functions of positive type on $G$ attached with irreducible representations of $G$ at depth $l$. 
	\begin{lemma}\label{les fonction continue de LSS sont a support compact}
		Suppose that $\mathcal{S}$ factorizes at depth $l$. Let $U$ be conjugate to an element of $\mathcal{S}\lb l \rb$ and let $V\leq G$ be conjugate to an element of $\mathcal{S}$. Suppose that $\varphi: G\rightarrow\C$ is a $U$-right-invariant, $V$-left-invariant function satisfying
		\begin{equation*}
		\int_{W} \varphi(gh)\qq \diff\mu (h)=0 \q \forall g\in G
		\end{equation*}
		for every $W$ that is conjugate to an element of $\mathcal{S}\lb l-1\rb$ and such that $U\leq W$. Then, $\varphi$ is compactly supported and
		\begin{equation*}
		{\rm supp}{(\varphi)}\subseteq N_G(U, V )=\{g\in G\mid g^{-1}Vg \subseteq U\}.
		\end{equation*}
	\end{lemma}
	\begin{remark}
		Since $U$ is a compact open subgroup and since $\varphi$ is $U$-right-invariant, notice that $\varphi$ is automatically continuous. In particular, the integrals $\int_{W} \varphi(gh)\qq \diff\mu (h)$ are all well-defined.
	\end{remark}
	\begin{proof} Since $\mathcal{S}$ factorizes at depth $l$ notice that $N_G(U, V)$ is a compact set. Let $g\not\in N_G(U, V )$ and notice that $g^{-1}Vg\not\subseteq U$. In particular, there exists $W$ in the conjugacy class of an element of $\mathcal{S}\lb l-1\rb $ such that $U\subseteq W\subseteq g^{-1}VgU$. Hence, $gW\subseteq VgU$ and we have by $U$-right-invariance and $V$-left-invariance that $\varphi(gh)= \varphi(g)$ for all $h\in W$. It follows that
		\begin{equation*}
		\varphi(g)=\frac{1}{\mu(W)}\int_{W}\varphi(gh)\qq \diff\mu(h)= 0,
		\end{equation*} 
		which proves as desired that	$${\rm supp}{(\varphi)}\subseteq N_G(U, V )=\{g\in G\mid g^{-1}Vg \subseteq U\}.$$
	\end{proof}
	
	The above lemma has various consequences. Among those, the following lemma proves the point \textit{1} of Theorem \ref{la version paki du theorem de classification}.
	\begin{lemma}\label{Lemma existence of a seed}
		 Let $\pi$ be an irreducible representation of $G$ at depth $l$ and suppose that $\mathcal{S}$ factorizes at depth $l$. Then, there exists a unique $C_\pi\in \mathcal{F}_{\mathcal{S}}$ with height $l$ such that for all $U\in C_\pi$, $\pi$ has a non-zero $U$-invariant vector.   
	\end{lemma}
	\begin{proof}
		Since $\pi$ is at depth $l$, there exists a compact open subgroup $U\in \mathcal{S}\lb l \rb$ at depth $l$ and a non-zero vector $\xi \in \Hr{\pi}^U$. Now, let $V\leq G$ be conjugate to an element of $\mathcal{S}\lb l \rb$ and admitting a non-zero invariant vector $\xi'\in \Hr{\pi}$ and let us show that $\mathcal{C}(U)=\mathcal{C}(V)$. We consider the function $\varphi_{\xi, \xi'} :G \rightarrow \C : g\mapsto \prods{\pi(g)\xi}{\xi'}$. This function is $U$-right-invariant and $V$-left-invariant. On the other hand, notice from our hypothesis that $\pi$ does not admit a non-zero $W$-invariant vector for any subgroup $W\leq G$ that is conjugate to an element of $\mathcal{S}\lb l-1\rb$. In particular, for every such $W$ we obtain that
		\begin{equation*}
		\int_{W} \varphi_{\xi, \xi'}(gh)\qq \diff\mu (h)=\prods{\int_{W} \pi(g)\pi(h)\xi \diff \mu(h)}{\xi'} =0 \q \forall g \in G.
		\end{equation*}
		It follows from Lemma \ref{les fonction continue de LSS sont a support compact} that $\varphi_{\xi, \xi'}$ is supported inside $N_G(U,V )=\{g\in G\mid g^{-1}Vg \subseteq U\}$. On the other hand, since $\pi$ is irreducible, $\xi$ is cyclic and the function $\varphi_{\xi, \xi'}$ is not identically zero. This implies the existence of an element $g\in G$ such that $g^{-1}Vg \subseteq U$. Considering now the function $\varphi_{\xi', \xi} :G \rightarrow \C : g\mapsto \prods{\pi(g)\xi'}{\xi}$ we obtain by symmetry the existence of an element $h\in G$ such that $h^{-1}Uh \subseteq V$. In particular, $h^{-1}Uh\subseteq V \subseteq g U g^{-1}$. Hence, $g^{-1}h^{-1}Uhg \subseteq U$. Since $U$ is a compact open subgroup of $G$ and since $G$ is unimodular, this implies that $g^{-1}h^{-1}Uhg=U$. In particular, we obtain that $h^{-1}Uh= V$ which implies that $\mathcal{C}(U)=\mathcal{C}(V)$. 
	\end{proof}
	In light of this result we make the following definition.
	\begin{definition}\label{definition seed}
		Let $\mathcal{S}$ be a generic filtration of $G$ that factorizes at depth $l$ and let $\pi$ be an irreducible of $G$, the unique element $C_\pi\in \mathcal{F}_\mathcal{S}$ with height $l$ such that for all $U\in C_\pi$, $\pi$ admits a non-zero $U$-invariant vector is called the \tg{seed} of $\pi$.  
	\end{definition}
	The following proposition proves the point \textit{2} of Theorem \ref{la version paki du theorem de classification}.
	\begin{proposition}
		Let $\pi$ be an irreducible representation of $G$ at depth $l$ and suppose that $\mathcal{S}$ factorizes at depth $l$. Then, for every $U\in C_\pi$, $\pi$ admits a non-zero diagonal matrix coefficient supported in the compact open subgroup $N_G(U)$ of $G$. In particular, $\pi$ is induced from an irreducible representation of $N_G(U)$, belongs to the discrete series of $G$ and its equivalence class is isolated in the unitary dual $\widehat{G}$ for the Fell topology. 
	\end{proposition}
	\begin{proof}
		Since $\pi$ is at depth $l$, there exists a compact open subgroup $U\in \mathcal{S}\lb l \rb$ and a non-zero vector $\xi \in \Hr{\pi}^U$. From the proof of Lemma \ref{Lemma existence of a seed}, the diagonal matrix coefficient $\varphi_{\xi,\xi}: G\rightarrow \C: g\mapsto \prods{\pi(g)\xi}{\xi}$ is supported inside the compact set $N_G(U,U)$. Notice furthermore that $N_G(U,U)=N_G(U)$ is a compact open subgroup of $G$. In particular, the GNS construction implies that $\pi\cong \Ind_{N_G(U)}^G(\sigma)$ where $\sigma$ is the irreducible representation of $N_G(U)$ corresponding to the diagonal matrix coefficient $\restriction{\varphi_{\xi,\xi}}{N_G(U)}$. Furthermore, since $\varphi_{\xi,\xi}$ is compactly supported, the representation $\pi$ is both square-integrable and integrable. Hence, $\pi$ belongs to the discrete series of $G$ and \cite[Corollary 1 pg.223 ]{DufloMoore1976} ensures that its equivalence class is open in the unitary dual $\widehat{G}$ for the Fell topology.
	\end{proof}
	\begin{remark}
		Notice that the irreducible representation $\sigma$ of $N_G(U)$ such that $\pi\cong \Ind_{N_G(U)}^G(\sigma)$ comes from an irreducible representation of a finite quotient of $N_G(U)$. Indeed, since $\sigma$ is an irreducible representation of the compact group it is finite dimensional. In particular, $\sigma(N_G(U))$ is a closed subgroup of the Lie group $\mathcal{U}(d)$ of unitary operators of the $d$ dimensional complex Hilbert space for some positive integer $d\in \N$. On the other hand, $\sigma(N_G(U))$ is a quotient of the totally disconnected compact group $N_G(U)$. In particular, $\sigma(N_G(U))$ is a totally disconnected compact Lie group and is therefore finite. This implies that ${\rm Ker}(\sigma)$ is an open subgroup of finite index of $N_G(U)$ and therefore that $\sigma$ is the lifted to $N_G(U)$ from an irreducible representation of the finite group $N_G(U)/{\rm Ker}(\sigma)$. The purpose of the rest of this chapter is to describe more explicitly the irreducible representations of $N_G(U)$ that arise in this manner if $\mathcal{S}$ factorizes$^+$ at depth $l$ and show that under this hypothesis every equivalence class of irreducible representation of $G$ at depth $l$ can be obtained using this procedure. 
	\end{remark}
	\subsection{Subrepresentations of the regular representations}
	Let $G$ be a non-discrete unimodular totally disconnected locally compact group and let $\mathcal{S}$ be a generic filtration of $G$. Let $\mu$ be a Haar measure of $G$, let $L^2(G)$ be separable complex Hilbert space of complex valued functions that are square-integrable with respect to $\mu$ and let $\lambda_G$ (resp. $\rho_G$) be the left-regular (resp. right-regular) representation of $G$. The purpose of this section is to study a function space related to the matrix coefficients of irreducible representations of $G$ at depth $l$ and seed $\mathcal{C}(U)$ (Definition \ref{definition seed}) if $\mathcal{S}$ factorizes at depth $l$. Lemma \ref{les fonction continue de LSS sont a support compact} motivates the following definition.
	\begin{definition}\label{definition of LSS}
		Let $U\leq G$ be conjugate to an element of $ \mathcal{S}\lb l \rb$ and suppose that $\mathcal{S}$ factorizes at depth $l$. We denote by $\mathcal{L}_{\mathcal{S}}(U)$ the closure in $L^2(G)$ of the set of complex valued functions $\varphi:G\rightarrow \C$ which respect the following three properties:
		\begin{enumerate}\label{les 3 prop de LSHS}
			\item $\varphi$ is $U$-right-invariant.
			\item There exists $V\leq G$ that is conjugate to an element of $\mathcal{S}$ such that $\varphi$ is $V$-left-invariant.
			\item For every $W$ that is conjugate to an element of $\mathcal{S}\lb l-1\rb$ and such that $U\leq W$ we have 
			\begin{equation*}
			\int_{W}\qq \varphi(gh)\qq \diff\mu(h)\qq=\qq 0\q \forall g\in G.
			\end{equation*}
		\end{enumerate}
		Equivalently the three properties can be formulated in terms of fixed point subspace and orthogonal complement as follows:
		\begin{enumerate}
			\item $\varphi\in L^2(G)^{\rho_G(U)}.$
			\item $\varphi\in \bigcup_{W\in\mathcal{S},V\in \mathcal{C}(W)}L^2(G)^{\lambda_G(V)}$.
			\item $\varphi\in \bigcap_{V\in \mathcal{S}\lb l-1\rb, U\subseteq W\in \mathcal{C}(V),g\in G} \big(\mathds{1}_{gW}\big)^\perp$.
		\end{enumerate}
	\end{definition}
	Notice that $\mathcal{L}_{\mathcal{S}}(U)$ is a set of equivalence classes of functions up to negligible sets and not a set of functions. However, the following lemma ensures the existence of a canonical choice of representative for any element of $\mathcal{L}_{\mathcal{S}}(U)$. 
	\begin{lemma}\label{representatiove of elements in LSS}
		Let $U\leq G$ be conjugate to an element of $\mathcal{S}\lb l\rb$, suppose that $\mathcal{S}$ factorizes at depth $l$ and let $\tilde{\varphi}\in \mathcal{L}_{\mathcal{S}}(U)$. Then, there exists a unique representative $\varphi$ of $\tilde{\varphi}$ which is $U$-right-invariant. Furthermore, for every $U\leq W$ that is conjugate to an element of $\mathcal{S}\lb l\rb$ we have
		\begin{equation*}
		\int_{W}\qq \varphi(gh)\qq \diff\mu(h)\qq=\qq 0\q \forall g\in G.
		\end{equation*}
	\end{lemma}
	\begin{proof}
		By the definition of $\mathcal{L}_{\mathcal{S}}(U)$, there exists a sequence $(\varphi_i)_{i\in \N}$ of complex valued functions such that: 
		\begin{enumerate}
			\item $\tilde{\varphi}_i\li{i}{\infty} \tilde{\varphi}$ in $L^2(G)$ where $\tilde{\varphi}_i$ denote the equivalence class of $\varphi_i$.
			\item $\varphi_i$ is $U$-right-invariant.
			\item For every $W$ that is conjugate to an element of $\mathcal{S}\lb l-1\rb$ and such that $U\leq W$ we have 
			\begin{equation*}
			\int_{W}\qq \varphi_i(gh)\qq \diff\mu(h)\qq=\qq 0 \q \forall g \in G.
			\end{equation*}
		\end{enumerate}
		Now, let $\varphi':G\rightarrow \C$ be a representative of $\tilde{\varphi}$. The above implies that
		\begin{equation*}
		\int_G \modu{\varphi_i (h)-\varphi'(h)}^2\qq\diff \mu (h)\li{i}{\infty}0.
		\end{equation*}
		On the other hand, since $G$ is a disjoint union of its $U$-left-cosets, we have 
		\begin{equation*}
		\int_G \modu{\varphi_i (h)-\varphi'(h)}^2\qq\diff \mu (h)= \s{gU\in G/U}{}\int_{gU}  \modu{\varphi_i (h)-\varphi'(h)}^2\qq\diff \mu (h).
		\end{equation*}
		Therefore, $\forall g\in G$ we obtain that
		\begin{equation*}
		\int_{gU}  \modu{\varphi_i (h)-\varphi'(h)}^2\qq\diff \mu (h)=\int_{U}  \modu{\varphi_i (gh)-\varphi'(gh)}^2\qq\diff \mu (h)\li{i}{\infty}0.
		\end{equation*}
		In particular, for every $g\in G$ this implies that $\varphi_i (gh)$ converges towards $\varphi'(gh)$ for almost all $h\in U$. Since the $\varphi_i$ are constant on $U$-left-cosets, this implies the existence of a unique representative $\varphi:G\rightarrow \C$ of $\tilde{\varphi}$ such that $\varphi(gh)=\varphi(g)$ $\forall g\in G, \forall h\in U$. On the other hand, since $U$ is a compact set, the convergence $\varphi_i\li{i}{\infty}\varphi$ is uniform on $U$-left-cosets. Now, let $g\in G$ and let $W$ be conjugate to an element of $\mathcal{S}\lb l-1\rb$ and such that $U\leq W$. Notice that $gW$ is a compact subset of $G$ and hence can be covered by finitely many $U$-left-cosets. In particular, the convergence $\varphi_i\li{i}{\infty}\varphi$ is also uniform on $W$-left-cosets. This implies as desired that
		\begin{equation*}
		\begin{split}
		\bigg\lvert\int_{W}\varphi(gh)\diff \mu(h)\bigg\lvert &=\bigg\lvert \int_{W}\varphi(gh)-\varphi_i(gh)\diff \mu(h)\bigg\lvert\leq \int_{W}\modu{\varphi(gh)-\varphi_i(gh)}\diff \mu(h)\\
		&\leq \mu(W)\supp{k\in gW}\modu{\varphi(k)-\varphi_i(k)}\qq\li{i}{\infty}\qq 0\q \forall g \in G.
		\end{split}
		\end{equation*}	
	\end{proof}
	If it leads to no confusion, we will identify any equivalence classes $\tilde{\varphi}\in \mathcal{L}_{\mathcal{S}}(U)$ with its canonical continuous representative. The following lemma shows that $\mathcal{L}_{\mathcal{S}}(U)$ is $G$-left-invariant and therefore defines a subrepresentation of the left-regular representation $(\lambda_G, L^2(G))$. 
	\begin{lemma}\label{LGSHS est G left invaraitn subspacede L2}
		Let $U\leq G$ be conjugate to an element of $\mathcal{S}\lb l \rb$ and suppose that $\mathcal{S}$ factorizes at depth $l$. Then, $\mathcal{L}_{\mathcal{S}}(U)$ is a closed $G$-left-invariant subspace of $L^2(G)$.
	\end{lemma}
	\begin{proof}
		By the definition of $\mathcal{L}_{\mathcal{S}}(U)$, it is enough to prove that the set of functions satisfying the three properties appearing in Definition \ref{definition of LSS} is $G$-left-invariant. To this end, let $k\in G$, let $\varphi$ be such a function and notice that:
		\begin{enumerate}
			\item $\lambda_G(k)\varphi\in L^2(G)^{\rho_G(U)}$.
			\item $\lambda_G(k)\varphi\in \bigcup_{W\in\mathcal{S},V\in \mathcal{C}(W)}L^2(G)^{\lambda_G(kVk^{-1})}=\bigcup_{W\in\mathcal{S},V\in \mathcal{C}(W)}L^2(G)^{\lambda_G(V)}$.
			\item $\lambda_G(k)\varphi\in \bigcap_{ V\in \mathcal{S}\lb l-1\rb, U\subseteq W\in \mathcal{C}(V),g\in G} \big(\mathds{1}_{kgW}\big)^\perp=\bigcap_{ V\in \mathcal{S}\lb l-1\rb, U\subseteq W\in \mathcal{C}(V),g\in G} \big(\mathds{1}_{gW}\big)^\perp$
		\end{enumerate} 
	\end{proof} 
	In light of Lemma \ref{LGSHS est G left invaraitn subspacede L2}, if $\mathcal{S}$ factorizes at depth $l$ and if $U\leq G$ is conjugate to an element of $\mathcal{S}\lb l \rb$ we denote by $T_{\mathcal{S},U}$ the subrepresentation of the left-regular representation $(\lambda_G,L^2(G))$ with representation space $\mathcal{L}_\mathcal{S}(U)$. We start by showing that $T_{\mathcal{S},U}$ depends, up to equivalence, only on the choice of the conjugacy class $C\in \mathcal{F}_\mathcal{S}=\{\mathcal{C}(U)\lvert U\in \mathcal{S}\}$ with height $l$ and not on the choice of $U\in C$. 
	\begin{lemma}
		Suppose that $\mathcal{S}$ factorizes at depth $l$ and let $C\in \mathcal{F}_\mathcal{S}$ be a conjugacy class with height $l$. Then for all $U,U'\in C$, the representations $(T_{\mathcal{S},U}, \mathcal{L}_{\mathcal{S}}(U))$ and $(T_{\mathcal{S},U'}, \mathcal{L}_{\mathcal{S}}(U'))$ are unitarily equivalent.  
	\end{lemma}
	\begin{proof}
		Since $U$, $U'\in C$, there exists a $k\in G$ such that $U'= kUk^{-1}$. Notice that $\rho_G(k):L^2(G)\rightarrow L^2(G)$ is a unitary operator and that $\rho_G(k)\mathcal{L}_{S}(U)=\mathcal{L}_S(U')$. Indeed, for every function $\varphi: G\rightarrow \C$ such that $\varphi\in L^2(G)^{\rho_G(U)},$
		$\varphi\in \bigcup_{W\in\mathcal{S},V\in \mathcal{C}(W)}L^2(G)^{\lambda_G(V)}$
		and $\varphi\in \bigcap_{V\in \mathcal{S}\lb l-1\rb, U\subseteq W\in \mathcal{C}(V),g\in G} \big(\mathds{1}_{gW}\big)^\perp$ we have that:
		\begin{enumerate}
			\item $\rho_G(k)\varphi\in L^2(G)^{\rho_G(kUk^{-1})}=L^2(G)^{\rho_G(U')}.$
			\item $\rho_G(k)\varphi\in \bigcup_{W\in\mathcal{S},V\in \mathcal{C}(W)}L^2(G)^{\lambda_G(V)}$.
			\item $\rho_G(k)\varphi\in \bigcap_{V\in \mathcal{S}\lb l-1\rb, U\subseteq W\in \mathcal{C}(V),g\in G} \big(\mathds{1}_{gWk^{-1}}\big)^\perp$ \\
			$\mbox{ }$ $\mbox{ }$ $\mbox{ }$  $\mbox{ }$  $\mbox{ }$  $\mbox{ }$  $\mbox{ }$  $\mbox{ }$  $\mbox{ }$  $\mbox{ }$  $\mbox{ }$  $\mbox{ }$  $\mbox{ }$  $\mbox{ }$  $\mbox{ }$  $\mbox{ }$  $\mbox{ }$  $=\bigcap_{V\in \mathcal{S}\lb l-1\rb, U'\subseteq kWk^{-1}\in \mathcal{C}(V),gk^{-1}\in G} \big(\mathds{1}_{gk^{-1}(kWk^{-1})}\big)^\perp$
		\end{enumerate}
		In particular,  by a density argument using the continuity of $\rho_G(k)$ we obtain that $\rho_G(k)\mathcal{L}_{S}(U)\subseteq\mathcal{L}_S(U')$. Inverting the role of $U$ and $U'$, the same reasoning shows that $\rho_G(k^{-1})\mathcal{L}_{S}(U')\subseteq\mathcal{L}_S(U)$. Hence, $\rho_G(k)\mathcal{L}_S(U)=\mathcal{L}_S(U')$. Since $\lambda_G(g)\rho_G(k)= \rho_G(k)\lambda_G(g)$ for every $g\in G$, the restriction of $\rho_G(k)$ to $\mathcal{L}_{\mathcal{S}}(U)$ provides the desired intertwining operator.
	\end{proof}
	We showed in the proof of Lemma \ref{Lemma existence of a seed} that for every generic filtration $\mathcal{S}$ that factorizes at depth $l$, for every $U\leq G$ that is conjugate to an element of $\mathcal{S}\lb l\rb$ and for every irreducible representation  $\pi$ of $G$ at depth $l$ with a non-zero $U$-invariant vector $\xi$, the diagonal matrix coefficient $\varphi_{\xi,\xi}$ is a non-zero $U$-bi-invariant element of $\mathcal{L}_\mathcal{S}(U)$. The following lemma asserts that every subrepresentation of $\mathcal{L}_{\mathcal{S}}(U)$ contains such a function. 
	\begin{lemma}\label{Every invariant space in the regular contains a non trivial invariant vector}
	Let $U\leq G$ be conjugate to an element of $\mathcal{S}\lb l \rb$ and suppose that $\mathcal{S}$ factorizes at depth $l$. Then, every $G$-left-invariant subspace of $\mathcal{L}_{\mathcal{S}}(U)$ contains a non-zero $U$-bi-invariant function.
	\end{lemma}
	\begin{proof}
		Suppose that $M$ is a non-zero closed $G$-left-invariant subspace of $\mathcal{L}_{\mathcal{S}}(U)$. Let us consider any non-zero class of functions $\varphi\in M$ and let $t\in G$ be such that  $\varphi(t)\not=0$. Notice that the function
		\begin{equation*}
		\psi= \int_{U} T_{\mathcal{S},U}(kt^{-1})\varphi\qq \diff  \mu(k)
		\end{equation*}
		is $U$-left-invariant, $U$-right-invariant and satisfies 
		\begin{equation*}
		\begin{split}
		\int_{W}\qq \psi(gh)\qq \diff\mu(h)\qq=\qq 0 \q \forall g\in G
		\end{split}
		\end{equation*} 
		for every $W$ that is conjugate to an element of $\mathcal{S}\lb l-1\rb$ and such that $U\leq W$.
		Furthermore, since $M$ is a closed $G$-left-invariant subspace of $\mathcal{L}_{\mathcal{S}}(U)$, we have that $\psi\in M$. Finally, the function $\psi$ is not identically zero since the $U$-right invariance of $\varphi$ implies that
		\begin{equation*}
		\psi(1_G)= \int_{U} \varphi(tk^{-1})\diff \mu(k)= \int_{U}\varphi(t)\qq \diff \mu(k)= \mu(U) \varphi(t)\qq \not=0. 
		\end{equation*}
	\end{proof}
	The following proposition shows that the space of $U$-bi-invariant functions of $\mathcal{L}_{\mathcal{S}}(U)$ is finite-dimensional and therefore that $T_{\mathcal{S},U}$ decomposes as a finite sum of irreducible representations of $G$.
	\begin{proposition}\label{The space of invariant vectors is finite dimensional}
		Let $U\leq G$ be conjugate to an element of $\mathcal{S}\lb l \rb$ and suppose that $\mathcal{S}$ factorizes at depth $l$. Then, the subspace of $U$-bi-invariant functions of $\mathcal{L}_{\mathcal{S}}(U)$ is finite dimensional.  
	\end{proposition}
	\begin{proof}
		Lemmas \ref{les fonction continue de LSS sont a support compact} ensures that the $U$-bi-invariant functions of $\mathcal{L}_{\mathcal{S}}(U)$ are supported inside the compact set $N_G(U)$. On the other hand every $U$-bi-invariant continuous function $\varphi:G\rightarrow \C$ is constant on the $U$-left-cosets and $N_G(U)$ can be covered by finitely many $U$-left-cosets. This proves that the subspace of $U$-bi-invariant functions of $\mathcal{L}_{\mathcal{S}}(U)$ is in the span of finitely many characteristic functions and is therefore finite dimensional. 
	\end{proof}
	\begin{corollary}\label{LSS is the finite sum of irreducible rep with seed S}
		Let $U\leq G$ be conjugate to an element of $\mathcal{S}\lb l \rb$ and suppose that $\mathcal{S}$ factorizes at depth $l$. Then, $T_{\mathcal{S}, U}$ decomposes as a finite sum of irreducible square-integrable representations of $G$ with seed $\mathcal{C}(U)$. 
	\end{corollary}
	\begin{proof}
		Lemma \ref{Every invariant space in the regular contains a non trivial invariant vector} and Proposition \ref{The space of invariant vectors is finite dimensional} ensures that $T_{\mathcal{S},U}$ decomposes as a finite sum of irreducible square-integrable representations of $G$; each of those containing a non-zero $U$-invariant vector. Now, let $W\leq G$ be conjugate to an element of $\mathcal{S}\lb r\rb$ for some $r\lneq l$. Lemma \ref{les fonction continue de LSS sont a support compact} ensures that the $W$-left-invariant functions of $\mathcal{L}_{\mathcal{S}}(U)$ are supported inside $N_G(U,W)=\{g\in G\mid gWg^{-1}\subseteq U\}$. On the other hand, since $\mathcal{C}(W)$ has height $r$, since $\mathcal{C}(U)$ has height $l$ and since $r\lneq l$ the set $N_G(U,W)$ is empty. This proves that $\mathcal{L}_{\mathcal{S}}(U)$ does not contain any non-zero $W$-left invariant function and therefore that every irreducible representation appearing in the decomposition of $T_{\mathcal{S},U} $ has seed $\mathcal{C}(U)$.
	\end{proof}
	It is a natural to ask whether the converse holds. The following provides a positive answer to this question. 
	\begin{lemma}\label{irreducible representation with seed S are in the regular rep LSS}
		Let $U\leq G$ be conjugate to an element of $\mathcal{S}\lb l \rb$ and suppose that $\mathcal{S}$ factorizes at depth $l$. Then, every irreducible representation $\pi$ of $G$ with seed $\mathcal{C}(U)$ is equivalent to a subrepresentation of $T_{\mathcal{S},U}$.
	\end{lemma}
	\begin{proof}
		Let $\pi$ be an irreducible representation of $G$ with seed $\mathcal{C}(U)$ and let $\xi \in \Hr{\pi}^U$ be a non-zero vector. Since $\pi$ has seed $\mathcal{C}(U)$, notice with a similar reasoning than in the proof of Lemma \ref{Lemma existence of a seed} that $\widecheck{\varphi}_{\xi,\xi}$ is a non-zero $U$-bi-invariant function of $\mathcal{L}_\mathcal{S}(U)$ where $\widecheck{\varphi}_{\xi,\xi}: G\rightarrow \C : g\mapsto \prods{\pi(g^{-1})\xi}{\xi}$. Now consider the diagonal matrix coefficient $\psi(\cdot)=\prods{T_{\mathcal{S},U}(\cdot)\widecheck{\varphi}_{\xi,\xi}}{\widecheck{\varphi}_{\xi,\xi}}$ of $T_{\mathcal{S},U}$ and notice from \cite[Theorem 14.3.3]{Dixmier1977} that
		\begin{equation*}
		\begin{split}
		\psi(g)&=\int_G T_{\mathcal{S},U}(g)\widecheck{\varphi}_{\xi,\xi}(h)\overline{\widecheck{\varphi}_{\xi,\xi}(h)}\diff \mu (h)=\int_G {\varphi}_{\xi,\xi}(h^{-1}g)\overline{{\varphi}_{\xi,\xi}(h^{-1})}\diff \mu (h)\\
		&=\int_G {\varphi}_{\pi(g)\xi,\xi}(h^{-1})\overline{{\varphi}_{\xi,\xi}(h^{-1})}\diff \mu (h)=\int_G {\varphi}_{\pi(g)\xi,\xi}(h)\overline{{\varphi}_{\xi,\xi}(h)}\diff \mu (h)\\
		&=d_{\pi}^{-1} \norm{\xi}{}^2\prods{\pi(g)\xi}{\xi} = d_{\pi}^{-1} \norm{\xi}{}^2\varphi_{\xi,\xi}(g)
		\end{split}
		\end{equation*}
		where $d_\pi$ is the formal dimension of $\pi$. The result follows from the uniqueness of the GNS construction \cite[Theorem C.4.10]{BekkaHarpeValette2008}.
	\end{proof}
	\begin{corollary}
		Let $C\in \mathcal{F}_{\mathcal{S}}$ be a conjugacy class with height $l$ and suppose that $\mathcal{S}$ factorizes at depth $l$. Then, there exists at most finitely many inequivalent classes of irreducible representations of $G$ with seed $C$. 
	\end{corollary}

	To improve the clarity of the expository the author decided to gather the results of the previous two sections in a theorem.
	Up to this point, we proved points $1$ and $2$ of Theorem \ref{la version paki du theorem de classification} together with the following theorem.
	\begin{theorem}\label{theorem de correspondance entre les subre repre se LS U et les irrep avec seed CU}
		Let $G$ be a non-discrete unimodular totally disconnected locally compact group and let $\mathcal{S}$ be a generic filtration of $G$, let $U\leq G$ be conjugate to an element of $\mathcal{S}\lb l \rb$ and suppose that $\mathcal{S}$ factorizes at depth $l$. Then, the following hold:
		\begin{enumerate}
			\item The square-integrable representation $(T_{\mathcal{S},U},\mathcal{L}_{\mathcal{S}}(U))$ depends, up to equivalence, only on the conjugacy class $\mathcal{C}(U)$. 
			\item Every $G$-left-invariant subspace of $\mathcal{L}_{\mathcal{S}}(U)$ contains a non-zero $U$-bi-invariant function and every such function is compactly supported inside $N_G(U)$.
			\item $T_{\mathcal{S},U}$ decomposes as a finite sum of square-integrable irreducible representations of $G$ with seed $\mathcal{C}(U)$. 
			\item Every irreducible representation of $G$ with seed $\mathcal{C}(U)$ is a subrepresentation of $T_{\mathcal{S},U}$.
		\end{enumerate} 
	\end{theorem}
	In particular, if $\mathcal{S}$ factorizes at depth $l$ we have a bijective correspondence between the equivalence classes of irreducible subrepresentations of $(T_{\mathcal{S},U}, \mathcal{L}_{\mathcal{S}}(U))$ and the equivalence classes of irreducible representations of $G$ with seed $\mathcal{C}(U)$. In the following sections, we introduce a family of irreducible representations of the finite group $N_G(U)/U$ which can be lifted to representations of $N_G(U)$ and show that the irreducible representations of $G$ with seed $\mathcal{C}(U)$ are induced from those if $\mathcal{S}$ factorizes$^+$ at depth $l$. 
		
	\subsection{Induced representations}\label{induced representation}
	The purpose of this section is to recall the definition of induction and some related results that will be required to end the proof of Theorem \ref{la version paki du theorem de classification}. Even if this notion makes sense in the context of locally compact groups and closed subgroups, this level of generality comes with technicalities that are unnecessary for the rest of our expository. Since most of the complexity vanishes if $H$ is an open subgroup of $G$ (because the quotient space $G/H$ is discrete) and since it is the only set up encountered in this notes, we will work under this hypothesis and refer to \cite[Chapters $2.1$ and $2.2$]{KaniuthTaylor2013} for details.
	
	Let $G$ be a locally compact group, let $H\leq G$ be an open subgroup and let $\sigma$ be a representation of $H$. The induced representation  $\Ind_{H}^G(\sigma)$ is a representation of $G$ with representation space given by
	$$\Ind_H^G(\Hr{\sigma})=\bigg\{\phi:G\rightarrow\Hr{\sigma}\Big\lvert  \phi(gh)=\sigma(h^{-1})\phi(g), \s{gH\in G/H}{}\prods{\phi(g)}{\phi(g)}<+\infty\qq \bigg\}.$$
	For $\psi, \phi\in \Ind_H^G(\Hr{\sigma})$, we let
	\begin{equation*}
	\prods{\psi}{\phi}_{\Ind_H^G(\Hr{\sigma})}= \s{gH\in G/H}{} \prods{\psi(g)}{\phi(g)}.
	\end{equation*}
	Equipped with this inner product, $\Ind_H^G(\Hr{\sigma})$ is a separable complex Hilbert space. The induced representation $\Ind_H^G(\sigma)$ is the representation of $G$ on $\Ind_H^G(\Hr{\sigma})$ defined by
	\begin{equation*}
	\big\lb\Ind_H^G(\sigma)(h)\big\rb\phi(g)=\phi(h^{-1}g)\q \forall \phi \in \Ind_H^G(\Hr{\sigma}) \mbox{ and }\forall g,h\in G.
	\end{equation*}

	The following three results are classical but the author did not find any convenient sources in the literature and chose to give a proof for instructive purposes and for completeness of the argument. 	
	\begin{lemma}\label{la correspondance Hsigma D}
		Let $G$ be a locally compact group, let $H\leq G$ be an open subgroup and let $\sigma$ be a representation of $H$. Then, there exists an isomorphism between the Hilbert spaces $\mathfrak{D}^\sigma_H=\{\phi\in \
		\Ind_H^G(\Hr{\sigma})\lvert \phi(g)=0 \qq \forall g\in G-H\}$ and $\Hr{\sigma}$ that intertwines the representations $\restriction{\Ind_H^G(\sigma)}{H}$ and $\sigma$. Furthermore, we have
		\begin{equation*}\label{la densite des fonction a support dans H dans l induite}
		\Ind_H^G(\Hr{\sigma})=\ub{ \underset{gH\in G/H}{\bigoplus}\qq\big\lb\Ind_H^G(\sigma)(g)\big\rb\mathfrak{D}^\sigma_H}.
		\end{equation*}
	\end{lemma}
	\begin{proof} 
		First notice that $\mathfrak{D}^\sigma_H$ is a closed subspace of $\Ind_H^G(\Hr{\sigma})$ and therefore defines a Hilbert space. Now, for every $\xi \in \Hr{\sigma}$, let \begin{equation*}
		\phi_\xi(g)\qq=\qq \begin{cases}
		\qq\sigma(g^{-1})\xi&\mbox{ if } g\in H\\
		\q \q 0 &\mbox{ if } g\not\in H
		\end{cases}
		\end{equation*}
		 and notice that $\phi_\xi=0$ only if $\xi=0$ and that $\phi_\xi \in \Ind_H^G(\Hr{\sigma})$.
		 On the other hand, every function $\phi \in \mathfrak{D}^\sigma_H$ is uniquely determined by the value it takes on $1_G$ since $\phi(g)=0 \qq \forall g\in G -H$ and $\phi(g)=\sigma(g^{-1})\phi(1_G)$  $\forall g\in H$. Hence, we have that $\mathfrak{D}^\sigma_H=\{\phi_\xi\lvert \xi \in \Hr{\sigma}\}$ and the map
		\begin{equation*}
		\Phi\fct{\Hr{\sigma}}{\mathfrak{D}^\sigma_H}{\xi}{\phi_\xi}\qq
		\end{equation*}
		is a linear isomorphism between $\Hr{\sigma}$ and $\mathfrak{D}^\sigma_H$. Moreover, $\Phi$ is unitary since $\forall \xi,\eta \in \Hr{\sigma}$ we have
		\begin{equation*}
		\begin{split}
		\prods{\Phi(\xi)}{\Phi(\eta)}_{\Ind_H^G(\Hr{\sigma})}&=\s{gH\in G/H}{} \prods{\phi_\xi(g)}{\phi_\eta(g)}=\prods{\phi_\xi(1_G)}{\phi_\eta(1_G)}=\prods{\xi}{\eta}.
		\end{split}
		\end{equation*}
		Finally, $\Phi$ intertwines $\sigma$ and $\restr{\Ind_H^G(\sigma)}{H}$ since $\forall h\in H$, $\forall g\in G$ we have
		\begin{equation*}
		\begin{split}
		\Phi(\sigma(h)\xi)(g)&=\phi_{\sigma(h)\xi}(g)=\begin{cases}
		\sigma(g^{-1}h)\xi &\mbox{ if }g\in H\\
		\q \q 0&\mbox{ if } g\not\in H
		\end{cases}\\
		&=\phi_{\xi}(h^{-1}g)=\big\lb\Ind_H^G(\sigma)(h)\big\rb\phi_{\xi}(g).
		\end{split}
		\end{equation*}
		This proves the first part of the claim. To prove the second part of the claim, let $\phi\in \Ind_H^G(\Hr{\sigma})$ and let us show that 
		$$\phi=\s{gH\in G/H}{}\lb \Ind_H^G(\sigma)(g)\rb\Phi(\phi(g)).$$ 
		First, notice that $\Ind_H^G(\sigma)(g)\mathfrak{D}^\sigma_H$ is the set of functions of $\Ind_{H}^G(\Hr{\sigma})$ that are supported inside $gH$. On the other hand, $\forall g,t\in G$, $\forall h\in H$ we have that 
		\begin{equation*}
		\begin{split}
		\big\lb\big\lb\Ind_{H}^G(\sigma)(gh)\big\rb\Phi(\phi(gh))\big\rb(t)&=\Phi(\phi(gh))(h^{-1}g^{-1}t)\\
		&=\begin{cases}
		\sigma(t^{-1}gh)\phi(gh) &\mbox{ if }h^{-1}g^{-1}t\in H\\
		\q \q 0&\mbox{ if }h^{-1}g^{-1}t\not\in H
		\end{cases}\\
		&=\begin{cases}
		\sigma(t^{-1}g)\phi(g) &\mbox{ if }g^{-1}t\in H\\
		\q \q 0&\mbox{ if }g^{-1}t\not \in H
		\end{cases}\\
		&=\Phi(\phi(g))(g^{-1}t)=\big \lb\lb\Ind_{H}^G(\sigma)(g)\rb\Phi(\phi(g))\big\rb(t).
		\end{split}
		\end{equation*}
		This proves that the map $G/H\rightarrow \Ind_H^G(\Hr{\sigma}): gH\mapsto \big\lb\Ind_{H}^G(\sigma)(g)\big\rb\Phi(\phi(g))$ is well defined.
		On the other hand, $\forall g\in G$, $\forall h\in H$, we have
		\begin{equation*}
		\big\lVert\big\lb\big\lb \Ind_{H}^G(\sigma)(g)\big\rb\Phi(\phi(g))\big\rb(gh)\big\rVert=\norm{\Phi(\phi(g))(h)}{}= \norm{\sigma(h^{-1})\phi(g)}{}= \norm{\phi(g)}{}.
		\end{equation*} 
		Hence, we obtain that
		\begin{equation*}
		\begin{split}
		&\Bigg\lVert\s{gH\in G/H}{} \big\lb\Ind_{H}^G(\sigma)(g)\big\rb\Phi(\phi(g)) \Bigg\rVert_{\Ind^G_H(\Hr{\sigma})}^2 = \s{tH\in G/H}{} \Bigg \lVert \s{gH\in G/H}{} \big\lb \big\lb\Ind_{H}^G(\sigma)(g)\big\rb\Phi(\phi(g))\big \rb(t)\Bigg \rVert^2\\
		&\q\q= \s{tH\in G/H}{} \big \lVert \big\lb \big\lb\Ind_{H}^G(\sigma)(t)\big\rb\Phi(\phi(t))\big \rb(t)\big \rVert^2= \s{tH\in G/H}{}\norm{\phi(t)}{}^2=\norm{\phi}{\Ind_H^G(\Hr{\sigma})}^2.
		\end{split}
		\end{equation*}
		This proves that $\s{gH\in G/H}{} \big\lb\Ind_{H}^G(\sigma)(g)\big \rb\Phi(\phi(g))$ belongs to $\ub{ \underset{gH\in G/H}{\bigoplus}\qq\Ind_H^G(\sigma)(g)\mathfrak{D}^\sigma_H}$. Finally, a direct computation shows that
		\begin{equation*}
		\begin{split}
		&\Bigg\lVert\phi-\s{gH\in G/H}{} \Ind_{H}^G(\sigma)(g)\Phi(\phi(g))\Bigg\lVert^2_{\Ind_{H}^G(\sigma)}=\s{tH\in G/H}{}\Bigg\lVert\phi(t)-\s{gH\in G/H}{}  \Ind_{H}^G(\sigma)(g)\Phi(\phi(g))(t)\Bigg\lVert^2\\
		&\q\q =\s{tH\in G/H}{} \big\lVert\phi(t)-\Ind_{H}^G(\sigma)(t)\Phi(\phi(t))(t)\big\lVert^2= \s{tH\in G/H}{} \norm{\phi(t)- \phi (t)}{}^2=0.
		\end{split}
		\end{equation*}
	\end{proof}
	The following proposition provides a criterion to check irreducibility of representations induced from irreducible representations. 
	\begin{proposition}\label{a criterion so that the induced of an irreducible is irreducible}
		Let $G$ be a locally compact group, let $H\leq G$ be an open subgroup and let $\sigma$ be an irreducible representations of $H$. Then, $\Ind_H^G(\sigma)$ is irreducible if and only if every non-zero closed invariant subspace of $\Ind_H^G(\Hr{\sigma})$ contains a non-zero function supported in $H$. 
	\end{proposition}
	\begin{proof}
		The forward implication is trivial. For the other implication, let $M$ be a non-zero closed invariant subspace of $\Ind_H^G(\Hr{\sigma})$. By the hypothesis, $M\cap \mathfrak{D}_{H}^{\sigma}\not=\es$. In particular, since $\sigma$ is irreducible, the correspondence of Lemma \ref{la correspondance Hsigma D} implies that $\mathfrak{D}_H^{\sigma}\subseteq M$. We conclude from that same lemma that $M=\Ind_{H}^G(\Hr{\sigma})$ since $M$ is a closed invariant subspace of $\Ind_H^G(\Hr{\sigma})$ and since \begin{equation*}
		\Ind_H^G(\Hr{\sigma})=\ub{ \underset{gH\in G/H}{\bigoplus}\qq\big\lb\Ind_H^G(\sigma)(g)\big\rb\mathfrak{D}^\sigma_H}.
		\end{equation*}
	\end{proof}
	The following lemma provides a sufficient condition for induced representations to be inequivalent. 
	\begin{lemma}\label{les induites dinequivalente sont inequivalente}
		Let $G$ be a locally compact group, let $H\leq G$ be an open subgroup and let $\sigma_1,\sigma_2$ be two inequivalent irreducible representations of $H$. Suppose that there exists a subgroup $K\leq H$ such that $\mathfrak{D}^{\sigma_i}_H=(\Ind_H^G(\Hr{\sigma_i})\big)^{{K}}$. Then, $\Ind_H^G(\sigma_1)$ and  $\Ind_H^G(\sigma_2)$ are inequivalent. 
	\end{lemma}
	\begin{proof}
		The proof is by contradiction. Suppose that there exists a unitary operator ${\rm U}$ intertwining $\Ind_H^G(\sigma_1)$ and $\Ind_H^G(\sigma_2)$. In particular, we have ${\rm U}(\Ind_H^G(\Hr{\sigma_1})\big)^{{K}}=(\Ind_H^G(\Hr{\sigma_2})\big)^{{K}}$. Let $\Phi_i: \Hr{\sigma_i}\rightarrow \mathfrak{D}_H^{\sigma_i}$ be the correspondences given by Lemma \ref{la correspondance Hsigma D} and notice that $\Phi_i(\Hr{\sigma_i})=\mathfrak{D}_H^{\sigma_i}=(\Ind_H^G(\Hr{\sigma_i})\big)^{{K}}.$ In particular, since $\Phi_i$ is a unitary operator intertwining $\sigma_i$ and $\restriction{\Ind_H^G(\sigma_i)}{H}$, the unitary operator $\Phi_2^{-1}{\rm U}\Phi_1$ intertwines $\sigma_1$ and $\sigma_2$ which leads to a contradiction.
	\end{proof}

	\subsection{The bijective correspondence of Theorem A}\label{Section the bijective correspondence of theorem A}
	
	Let $G$ be a non-discrete unimodular totally disconnected locally compact group and let $\mathcal{S}$ be a generic filtration of $G$. The purposes of this section is to describe explicitly the bijective correspondence of Theorem \ref{la version paki du theorem de classification}. This requires some formalism that we now introduce.
	
	Let $U$ be conjugate to an element of $\mathcal{S}\lb l \rb$ and notice that if $\mathcal{S}$ factorizes at depth $l$, $N_G(U)/U$ is a finite group. Furthermore, notice that $$N_G(g^{-1}Ug)=g^{-1}N_G(U)g \q \forall g\in G$$ which implies that $$N_G(U)/U\cong N_G(V)/V\q \forall U,V\in C, \qq \forall C\in \mathcal{F}_\mathcal{S}.$$ This motivates the following definition.
	\begin{definition}
		For every conjugacy class $C\in \mathcal{F}_{\mathcal{S}}$, the \textbf{group of automorphisms} $\Aut_{G}(C)$ \textbf{of the seed} $C$ is the finite group $N_G(U)/U$ corresponding to any group $U\in C$. 
	\end{definition}
	\noindent For all $C\in \mathcal{F}_{\mathcal{S}}$ with height $l$ and every $U\in C$, we set
	\begin{equation*}
	\tilde{\mathfrak{H}}_{\mathcal{S}}(U)= \{W \mid \exists g\in G\mbox{ s.t. }gWg^{-1}\in \mathcal{S}\lb l-1 \rb \mbox{ and } U \subseteq  W \}.
	\end{equation*}
	 and let $p_{U}: N_G(U)\mapsto N_G(U)/U$ denote the quotient map. If $\mathcal{S}$ factorizes$^+$ at depth $l$, notice that $p_{U}(W)$ is a non-trivial (possibly not proper) subgroup of $\Aut_{G}(C)$ for every $W\in \tilde{\mathfrak{H}}_{\mathcal{S}}(U)$. Moreover, notice that $$\tilde{\mathfrak{H}}_{\mathcal{S}}(g^{-1}Ug)= g^{-1}\tilde{\mathfrak{H}}_{\mathcal{S}}(U)g \q \forall g\in G.$$ In particular, the subset of non-trivial subgroups of $\Aut_G(C)$ $$\mathfrak{H}_{\mathcal{S}}(C)=\{ p_U(W)\lvert W\in \tilde{\mathfrak{H}}_\mathcal{S}(U)\}$$ does not depend on our choice of representative $U\in C$. 
	\begin{definition}
		An irreducible representation $\omega$ of $\Aut_{G}(C)$ is called $\mathcal{S}$-\tg{standard} if it has no non-zero $H$-invariant vector for any $H\in \mathfrak{H}_{\mathcal{S}}(C)$.
	\end{definition}
	Our goal is to describe the irreducible representations of $G$ with seed $C$ from these $\mathcal{S}$-\tg{standard} representations. To this end, we recall that every representation $\omega$ of $N_G(U)/U$ can be lifted to a representation $\omega \circ p_U$ of $N_G(U)$ acting trivially on $U$ and with representation space $\Hr{\omega}$. Furthermore, notice that $\omega\circ p_U$ is irreducible if and only if $\omega$ is irreducible. We can now use the process of induction recalled in Section \ref{induced representation}. For shortening of the formulation, we denote by $$T(U,\omega)=\Ind_{N_G(U)}^G(\omega\circ p_{U})$$ the resulting representation of $G$. Our purpose will be to show that $T(U,\omega)$ is an irreducible representation of $G$ with seed $\mathcal{C}(U)$ if $\omega$ is $\mathcal{S}$-standard. Conversely, if $\pi$ is an irreducible representation of $G$ with seed $C$, notice that $\Hr{\pi}^{U}$ is a non-zero $N_G(U)$-invariant subspace of $\Hr{\pi}$ for every $U\in C$. In particular, the restriction $(\restriction{\pi}{N_G(U)}, \Hr{\pi}^U)$ defines a representation of $N_G(U)$ whose restriction to $U$ is trivial. This representation passes to the quotient group $N_G(U)/U$ and therefore provides a representation $\omega_\pi$ of $\Aut_{G}(C)$. 
	
	The following theorem describes the bijective correspondence of Theorem \ref{la version paki du theorem de classification} using the above formalism.
	\begin{theorem}\label{the theorem of classification for cuspidal representations}
		Let $G$ be a non-discrete unimodular totally disconnected locally compact group, let $\mathcal{S}$ be a generic filtration of $G$ that factorizes$^+$ at depth $l$ and let $C\in \mathcal{F}_{\mathcal{S}}$ be a conjugacy class at height $l$. Then, there exist a bijective correspondence between the equivalence classes of irreducible representations of $G$ with seed $C$ and the equivalence classes of $\mathcal{S}$-standard representations of $\Aut_G(C)$. More precisely, for every $U\in \mathcal{S}$ the following holds:
		\begin{enumerate}
			\item If $\pi$ is an irreducible representation of $G$ with seed $\mathcal{C}(U)$, the representation $(\omega_\pi, \Hr{\pi}^{U})$ of $\Aut_{G}(\mathcal{C}(U))$ is an $\mathcal{S}$-standard representation of $\Aut_{G}(\mathcal{C}(U))$ such that
			\begin{equation*}
			\pi \cong T(U,\omega_\pi)=\Ind_{N_G(U)}^{G}(\omega_\pi \circ p_{U}).
			\end{equation*}
			\item If $\omega$ is an $\mathcal{S}$-standard representation of $\Aut_{G}(\mathcal{C}(U))$, the representation $T(U,\omega)$ is an irreducible representation of $G$ with seed in $\mathcal{C}(U)$. 
		\end{enumerate}
		Furthermore, if $\omega_1$ and $\omega_2$ are $\mathcal{S}$-standard representations of $\Aut_{G}(\mathcal{C}(U))$, we have that $T(U,\omega_1)\cong T(U,\omega_2)$ if and only if $\omega_1\cong\omega_2$. In particular, the above two constructions are inverse of one an other.
	\end{theorem}
	The structure of the proof is given as follows: 
	\noindent\begin{minipage}[t]{5\linewidth}%
		\centering
		\mbox{%
			\begin{minipage}[t]{\linewidth-2\fboxsep-2\fboxrule}% Remove fbox rule/sep width
				\hspace*{-24em}\raisebox{-13em}{ 
					\tikzstyle{block} = [rectangle, text centered,
					minimum height=10mm,node distance=8em,  rounded corners=1mm]
					\begin{tikzpicture}
					\node[draw,text width=4.5cm] at (0,0) [block ] (n01) {Irreducible subrepresentations of $T_{\mathcal{S},U}$};
					\node[draw,text width=4.5cm] at (3.5,-2) [block ] (n02) {Representations $T(U,\omega)$};
					\node[draw,text width=4.5cm] at (-3.5,-2) [block ] (n03) {Irreducible representations with seed $\mathcal{C}(U)$};

					\path(n02.east)  edge[thick, -latex,out=30,in=0,looseness=1.5]node[anchor=center,fill=white] {\circled{2}}(n01.east) ;
					\path(n01.south) edge[ thick, -latex,out=330,in=0,looseness=1.5]node[anchor=center,fill=white] {\circled{1}}(n03.east);
					\path(n03.west)  edge[thick, -latex,out=150,in=180,looseness=1.5]node[anchor=center,fill=white] {\circled{1}}(n01.west);
					\path(n03) edge[thick, -latex,out=330,in=210,looseness=1]node[anchor=center,fill=white] {\circled{3}}(n02);
					\end{tikzpicture}  % Moved (+3em,+4em)
					}
		\end{minipage}}
	\end{minipage}\hfill
	\begin{center}
		\circled{1} Theorem \ref{theorem de correspondance entre les subre repre se LS U et les irrep avec seed CU} $\mbox{ }$ \circled{2} Proposition \ref{les induite de S standard sont des sous rep de TS}  $\mbox{ }$ \circled{3} Lemma \ref{every irred rep with seed S is the induced rep}
	\end{center}
	The last part of the statement is then handled by Lemma \ref{unicitee des induite par res S standard} below.
	\newpage
	\noindent We start the proof with an intermediate result.
	\begin{lemma}\label{les fonction U invarainte de LS U on inverse dans LsU}
		Let $U\leq G$ be conjugate to an element of $\mathcal{S}\lb l \rb$, suppose that $\mathcal{S}$ factorizes $^+$ at depth $l$, let $\varphi$ be a $U$-bi-invariant function of $\mathcal{L}_{\mathcal{S}}(U)$ and let $\widecheck{\varphi}$ be the function defined by $\widecheck{\varphi}(g)=\varphi(g^{-1})$ $\forall g\in G$. Then, $\widecheck{\varphi}$ is a $U$-bi-invariant function of $\mathcal{L}_{\mathcal{S}}(U)$.
	\end{lemma}
	\begin{proof}
		The function $\widecheck{\varphi}$ is clearly $U$-bi-invariant. Furthermore, since $G$ is unimodular and since $\varphi\in L^2(G)$, notice that $\widecheck{\varphi}\in L^2(G)$. On the other hand, since $\varphi$ is a $U$-bi-invariant function of $\mathcal{L}_{\mathcal{S}}(U)$, Lemma \ref{les fonction continue de LSS sont a support compact} ensures that $\varphi$ and therefore $\widecheck{\varphi}$ is supported inside $N_G(U)$. Now let $W$ be conjugate to an element of $\mathcal{S}\lb l-1\rb$ and such that $U\leq W$. Since $\mathcal{S}$ factorizes$^+$ at depth $l$, we have that $W\leq N_G(U)$. In particular, for every $g\not\in N_G(U)$, $gW\cap N_G(U)=\es$ and $\int_{W}\widecheck{\varphi}(gh)\diff \mu (h)=0.$ On the other hand, if $g\in N_G(U)$, notice that $U=gUg^{-1}\subseteq gWg^{-1}$ which implies that
		$$\int_W\widecheck{\varphi}(gh)\diff\mu(h)=\int_{W} \varphi(hg^{-1})\diff \mu(h)=\int_{gWg^{-1}} \varphi(g^{-1}h)\diff \mu(h)=0.$$
	\end{proof}

	\begin{proposition}\label{les induite de S standard sont des sous rep de TS}
		Let $U\leq G$ be conjugate to an element of $\mathcal{S}\lb l \rb$, suppose that $\mathcal{S}$ factorizes$^+$ at depth $l$ and let $\omega$ be an $\mathcal{S}$-standard representation of $\Aut_{G}(\mathcal{C}(U))$. Then, $\mathfrak{D}^{\omega \circ p_U}_{N_G(U)}=\Hr{T(U,\omega)}^U$ and $T(U,\omega)$ is equivalent to an irreducible subrepresentation of $T_{\mathcal{S},U}$.
	\end{proposition}
	\begin{proof}
		We start by showing that $T(U,\omega)$ is equivalent to a subrepresentation of $T_{\mathcal{S},U}$. Let $\xi\in \Hr{\omega}$ be non-zero and consider the function 
		$$\phi_\xi:G\rightarrow \C : g\mapsto \begin{cases}
		\omega\circ p_U(g^{-1})\xi &\mbox{ if }g\in N_G(U)\\
		\q \q \qq 0  &\mbox{ if }g\not\in N_G(U)
		\end{cases}.$$
		By the definition, we have that $\phi_\xi \in \Hr{T(U,\omega)}$ and since $\omega$ is irreducible, Lemma \ref{la correspondance Hsigma D} ensures that $\phi_\xi$ is cyclic for $T(U,\omega)$. We consider the diagonal matrix coefficient $\varphi_{\phi_\xi,\phi_\xi}(\cdot)=\prods{T(U,\omega)(\cdot)\phi_\xi}{\phi_\xi}$ of $T(U,\omega)$ and let 
		\begin{equation*}
		\varphi_{\xi,\xi}:G\rightarrow \C : g\mapsto\begin{cases}
		\prods{\omega \circ p_U(g)\xi}{\xi} &\mbox{ if }g\in N_G(U)\\
		\q\q \q 0 &\mbox{ if }g\not\in N_G(U)
		\end{cases}.
		\end{equation*} A straight computation shows for every $g\in G$ that
		\begin{equation}\label{equation piece 1 of GNS implication}
		\begin{split}
		\varphi_{\phi_\xi,\phi_\xi}(g)&=\prods{T(U,\omega)(g)\phi_\xi}{\phi_\xi}_{\Ind_{N_G(U)}^G(\Hr{\omega \circ p_U})}=\s{tH\in G/N_G(U)}{}\prods{T(U,\omega)(g)\phi_\xi(t)}{\phi_\xi(t)}\\
		&=\s{tH\in G/N_G(U)}{}\prods{\phi_\xi(g^{-1}t)}{\phi_\xi(t)} =\prods{\phi_\xi(g^{-1})}{\phi_\xi(1_G)}=\varphi_{\xi,\xi}(g).
		\end{split}
		\end{equation}
		We are going to show that $\varphi_{\xi,\xi}$ is a $U$-left-invariant function of $\mathcal{L}_{\mathcal{S}}(U)$. It is clear from the definition that $\varphi_{\xi,\xi}$ is $U$-bi-invariant, compactly supported and hence square integrable. Now, let $W$ be conjugate to an element of $\mathcal{S}\lb l-1\rb$ and such that $U\leq W$. Since $\mathcal{S}$ factorizes$^+$ at depth $l$, notice that $W\leq N_G(U)$. In particular, for every $g\not\in N_G(U)$ we have that $gN_G(U)\cap N_G(U)=\es$ and therefore $\int_W\varphi_{\xi,\xi}(gh)\diff\mu(h)=0$. On the other hand, for every $g\in N_G(U)$, since $\omega$ is an $\mathcal{S}$-standard representation of $\Aut_G(\mathcal{C}(U))$ we have 
		$$\int_W\varphi_{\xi,\xi}(gh)\diff\mu(h)=\prods{\int_W\omega\circ p_U(h)\xi\diff\mu(h)}{\omega\circ p_U(g^{-1})\xi}=0.$$
		This proves, as desired, that $\varphi_{\xi,\xi}$ is a $U$-left-invariant function of $\mathcal{L}_{\mathcal{S}}(U)$. In particular, Lemma \ref{les fonction U invarainte de LS U on inverse dans LsU} ensures that $\widecheck{\varphi}_{\xi,\xi}\in \mathcal{L}_\mathcal{S}(U)$. Now, consider the diagonal matrix coefficient $\psi(\cdot)=\prods{T_{\mathcal{S},U}(\cdot)\widecheck{\varphi}_{\xi,\xi}}{\widecheck{\varphi}_{\xi,\xi}}$ of $T_{\mathcal{S},U}$, let $\mu$ be the Haar measure of $G$ renormalised in such a way that $\mu(N_G(U))=1$ and notice from \cite[Theorem 14.3.3]{Dixmier1977} that
		\begin{equation}\label{equation piece 2 of GNS implication}
		\begin{split}
		\psi(g)&=\int_G T_{\mathcal{S},U}(g)\widecheck{\varphi}_{\xi,\xi}(h)\overline{\widecheck{\varphi}_{\xi,\xi}(h)}\diff \mu (h)=\int_{N_G(U)} {\varphi}_{\xi,\xi}(h^{-1}g)\overline{{\varphi}_{\xi,\xi}(h^{-1})}\diff \mu (h)\\&=\int_{N_G(U)} {\varphi}_{\xi,\xi}(hg)\overline{{\varphi}_{\xi,\xi}(h)}\diff \mu (h) =d_{\omega}^{-1} \norm{\xi}{}^2\varphi_{\xi,\xi}(g)
		\end{split}
		\end{equation}
		where $d_\omega$ is the dimension of $\omega$. In particular, renormalising $\xi$ if needed, the equalities  \eqref{equation piece 1 of GNS implication}, \eqref{equation piece 2 of GNS implication} and the uniqueness of the GNS construction \cite[Theorem C.4.10]{BekkaHarpeValette2008} imply that $T(U,\omega)$ is a subrepresentation of $T_{\mathcal{S},U}$.
		
		Now, let us show that $\mathfrak{D}^{\omega \circ p_U}_{N_G(U)}=\big(\Ind_{N_G(U)}^G(\Hr{\omega\circ p_{U}})\big)^U=\Hr{T(U,\omega)}^U$. Let  $\mathcal{U}: \Hr{T(U,\omega)}\rightarrow \mathcal{L}_\mathcal{S}(U)$ be a unitary operator intertwining $T(U,\omega)$ and $T_{\mathcal{S},U}$ and notice that $\mathfrak{D}^{\omega \circ p_U}_{N_G(U)}\subseteq\Hr{T(U,\omega)}^U$. In particular, for every $g\in G$, we have that $\lb T(U,\omega)\big\rb(g)\mathfrak{D}^{\omega\circ p_U}_{N_G(U)}\subseteq\Hr{T(U,\omega)}^{gUg^{-1}}$. Therefore, $\forall \phi\in \lb T(U,\omega)\big\rb(g)\mathfrak{D}^{\omega\circ p_U}_{N_G(U)}$, the function $\mathcal{U}(\phi)$ is a $gUg^{-1}$-left invariant function of $\mathcal{L}_\mathcal{S}(U)$ and Lemma \ref{les fonction continue de LSS sont a support compact} ensures that $\mathcal{U}(\phi)$ is supported inside $$N_G(gUg^{-1},U)=\{t\in G\lvert t^{-1}gUg^{-1}t\subseteq U\}=g N_G(U).$$
		Now, let $\phi\in \Big(\mathfrak{D}^{\omega \circ p_U}_{N_G(U)}\Big)^\perp\cap\Hr{T(U,\omega)}^U$. As a consequence of Lemma \ref{la correspondance Hsigma D}, we have that
		$$  \Big(\mathfrak{D}^{\omega \circ p_U}_{N_G(U)}\Big)^\perp = \overline{\underset{tN_G(U)\in G/N_G(U)-\{N_G(U)\}}{\bigoplus}\qq\big\lb T(U,\omega)\big\rb(g)\mathfrak{D}^{\omega\circ p_U}_{N_G(U)}}.$$
		In particular, the above discussion implies that  $\mbox{\rm supp}(\mathcal{U}(\phi))\subseteq G-N_G(U).$ On the other hand, since $\phi\in \Hr{T(U,\omega)}^U$, the function $\mathcal{U}(\phi)$ is a $U$-left invariant function of $\mathcal{L}_\mathcal{S}(U)$ and is therefore supported inside $N_G(U)$. This implies that $\mathcal{U}(\phi)=0$. Hence, $\phi=0$ and $\mathfrak{D}^{\omega \circ p_U}_{N_G(U)}= \Hr{T(U,\omega)}^U$.
		
		Finally, we prove the irreducibility of $T(U,\omega)$ with Lemma \ref{a criterion so that the induced of an irreducible is irreducible}. Let $M$ be a non-zero closed invariant subspace of $\Hr{T(U,\omega)}$. Then $\mathcal{U}(M)$ is a non-zero closed invariant subspace of $T_{\mathcal{S},U}$ and Lemma \ref{Every invariant space in the regular contains a non trivial invariant vector} ensures the existence of a non-zero $U$-bi-invariant function $\varphi\in \mathcal{U}(M)$. In particular, $\mathcal{U}^{-1}(\varphi)$ is a non-zero $U$-invariant function of $\Hr{T(U,\omega)}$ contained in $M$. The result follows from the fact that $\Hr{T(U,\omega)}^U=\mathfrak{D}^{\omega \circ p_U}_{N_G(U)}$.
	\end{proof}
 	Together with Corollary \ref{LSS is the finite sum of irreducible rep with seed S}, this proves \circled{2}. Now let us prove \circled{3}.
	\begin{lemma}\label{every irred rep with seed S is the induced rep}
		Let $U\leq G$ be conjugate to an element of $\mathcal{S}\lb l \rb$, suppose that $\mathcal{S}$ factorizes$^+$ at depth $l$ and let $\pi$ be an irreducible representation of $G$ with seed $\mathcal{C}(U)$. Then $\omega_\pi$ is an $\mathcal{S}$-standard representation of $\Aut_G(\mathcal{C}(U))$ (in particular it is irreducible) and
		\begin{equation*}
		\pi \qq \cong \qq T(U,\omega_\pi)\qq =\qq \Ind_{N_G(U)}^G(\omega_\pi \circ p_{U}).
		\end{equation*}
	\end{lemma}
	\begin{proof}
		We start by showing that $(\omega_\pi, \Hr{\pi}^{U})$ is an $\mathcal{S}$-standard representation of $\Aut_G(\mathcal{C}(U))$. Let $M$ be a non-zero closed $\Aut_G(\mathcal{C}(U))$-invariant subspace of $\Hr{\pi}^{U}$ for $\omega_\pi$ and let $\xi\in \Hr{\pi}^U$ be non-zero. Since $\xi$ is cyclic for $\pi$ notice that, for every non-zero $\eta\in M$, the function $\varphi_{\xi, \eta }\fct{G}{\C}{g}{\prods{\pi(g)\xi}{\eta}}$ is not identically zero. In particular, since it is supported inside $N_G(U)$ by Lemma \ref{les fonction continue de LSS sont a support compact}, there exists an element $g\in N_G(U)$ such that $0\not= \varphi_{\xi,\eta}(g)=\prods{\xi}{\omega_{\pi}\circ p_U(g^{-1})\eta}$. Since $M$ is $\Aut_G(\mathcal{C}(U))$-invariant this proves the existence of a vector $\eta'\in M$ such that $\prods{\xi}{\eta'}\not=0$. It follows that the orthogonal complement of $M$ in $\Hr{\pi}^{U}$ is trivial. Hence, we obtain that $M=\Hr{\pi}^{U}$ and $\omega_\pi$ is irreducible. Now, let $W\in \tilde{\mathfrak{H}}_{\mathcal{S}}(U)$ where we recall that 
		\begin{equation*}
		\tilde{\mathfrak{H}}_{\mathcal{S}}(U)= \{W \mid \exists g\in G \mbox{ s.t. } gWg^{-1}\in \mathcal{S}\lb l-1 \rb \mbox{ and } U \subseteq  W \}.
		\end{equation*}
		Since $\pi(g)=\omega_\pi\circ p_{U}(g)$ for every $g\in N_G(U)$ and $\pi$ has seed $C$, there exists no non-zero $W$-invariant vector in $\Hr{\pi}^{U}$ for $\omega_\pi \circ p_{U}$. This proves that $\omega_\pi$ is $\mathcal{S}$-standard. 
		
		We now prove that $\pi\cong T(U,\omega_\pi).$ To this end, let $\xi\in \Hr{\pi}^U$ and let us consider the diagonal matrix coefficient $\varphi_{\xi,\xi}:G\rightarrow\C: g\mapsto \prods{ \pi(g)\xi}{\xi}$. The proof of Lemma \ref{Lemma existence of a seed} ensures that $\varphi_{\xi,\xi}$ is a $U$-bi-invariant function of $\mathcal{L}_{\mathcal{S}}(U)$ and is therefore compactly supported inside $N_G(U)$. In particular, we obtain that
		\begin{equation*}
		\varphi_{\xi,\xi}(g)=\begin{cases}
		\prods{\omega_\pi \circ p_U(g)\xi}{\xi} &\mbox{ if }g\in N_G(U)\\
		\q\q \q 0 &\mbox{ if }g\not\in N_G(U)
		\end{cases}.
		\end{equation*}
		On the other hand, the Equality \eqref{equation piece 1 of GNS implication} in the proof of Proposition \ref{les induite de S standard sont des sous rep de TS} ensures that $\varphi_{\xi,\xi}$ is a diagonal matrix coefficient of $\text{ \rm Ind}_{N_G(U)}^{G}(\omega_\pi\circ p_{U})$. Since $\text{ \rm Ind}_{N_G(U)}^{G}(\omega_\pi\circ p_{U})$ is irreducible by Proposition \ref{les induite de S standard sont des sous rep de TS}, the result follows from the
		uniqueness of the GNS construction \cite[Theorem C.4.10]{BekkaHarpeValette2008}.
	\end{proof}
	\begin{lemma}\label{unicitee des induite par res S standard}
		Let $U\leq G$ be conjugate to an element of $\mathcal{S}\lb l \rb$, suppose that $\mathcal{S}$ factorizes$^+$ at depth $l$ and let $\omega_1,\qq \omega_2$ be $\mathcal{S}$-standard representations of $\Aut_G(\mathcal{C}(U))$. Then, $T(U,\omega_1)$ and $T(U,\omega_2)$ are equivalent if and only if $\omega_1$ and $\omega_2$ are equivalent.
	\end{lemma}
	\begin{proof}
		Notice from the Proposition \ref{les induite de S standard sont des sous rep de TS} that  $\mathfrak{D}^{\omega_i \circ p_U}_{N_G(U)}=\Hr{T(U,\omega_i)}^U$. The result therefore follows from Lemma \ref{les induites dinequivalente sont inequivalente} applied with $H=N_G(U)$, $K=U$.
	\end{proof}
	\subsection{Existence criteria}\label{section existence des supercuspidal}
	Let $G$ be a non-discrete unimodular totally disconnected locally compact group and let $\mathcal{S}$ be a generic filtration of $G$. If $\mathcal{S}$ factorizes$^+$ at depth $l$, Theorem \ref{la version paki du theorem de classification} provides a bijective correspondence between the equivalence classes of irreducible representations of $G$ at depth $l$ with seed $C\in \mathcal{F}_{\mathcal{S}}$ and the $\mathcal{S}$-standard representations of $\Aut_G(C)$. However, it does not guarantee the existence of an irreducible representations of $G$ at depth $l$. The purpose of this section is to provide some existence criteria that will be used in the second parts of the paper. The proofs of the following results are essentially covered by \cite[Lemma 3.6, Lemma 3.7, Lemma 3.8 and Theorem 3.9]{FigaNebbia1991} but we recall them for completeness of the argument. The following results were used in \cite{FigaNebbia1991} to prove the existence cuspidal representations of the full group of automorphisms $\Aut(T)$ of a regular tree. 
	\begin{lemma}\label{les rep de Qsur Qi}
		Let $Q$ be a finite group with $\modu{Q}\gneq 2$ acting $2$-transitively on a finite set $X=\{1,...,d\}$, then there exists an irreducible representation of $Q$ without non-zero $\Fix_Q(i)$-invariant vector for all $i\in X$. 
	\end{lemma}
	\begin{proof}
		Since $Q$ acts transitively on $X$, notice that $\Fix_Q(i)$ and $\Fix_Q(j)$ are conjugate to one an other for all $i,j\in X$. In particular, a representation $\pi$ of $Q$ admits a non-zero $\Fix_Q(i)$-invariant vector for all $i\in X$ if and only if it admits a non-zero $\Fix_Q(1)$-invariant vector. In light of those considerations we are going to prove the existence of an irreducible representation of $Q$ without non-zero $\Fix_Q(1)$-invariant vectors.
		We recall that the quasi-regular representation $\sigma$ of $Q/\Fix_{Q}(1)$ is the representation $\text{\rm Ind}_{\Fix_{Q}(1)}^{Q}(\mathds{1}_{\Fix_{Q}(1)})$ of $Q$ induced by the trivial representation of $\Fix_{Q}(1)$. On the other hand, for every representation $\pi$ of $Q$, the Frobenius reciprocity implies that 
		\begin{equation*}
		\prods{\text{\rm Res}_{\Fix_{Q}(1)}^Q(\pi)}{\mathds{1}_{\Fix_{Q}(1)}}_{\Fix_{Q}(1)}\qq=\qq \prods{\pi}{\text{\rm Ind}_{\Fix_{Q}(1)}^{Q}(\mathds{1}_{\Fix_{Q}(1)})}_{Q}\qq=\qq  \prods{\pi}{\sigma}_{Q}.
		\end{equation*}
		In particular, every irreducible representation $\pi$ of $Q$ with a non-zero $\Fix_{Q}(1)$-invariant vector is a subrepresentation of $\sigma$. Moreover, since $Q$ acts $2$-transitively on $X$, \cite[Corollary 5.17]{Isaacs1976} ensures the existence of an irreducible representation $\psi$ of $Q$ such that $\sigma= \mathds{1}_{Q} \oplus \psi$. Suppose for a contradiction that every irreducible representation of $Q$ has a non-zero $\Fix_{Q}(1)$-invariant vector and is therefore contained in $\sigma$. This implies that $Q$ has two conjugacy classes and is therefore isomorphic to the cyclic group of order two which contradicts our hypothesis that $\modu{Q}\gneq2$.
	\end{proof}
	An other useful criterion that we adapt to the context of trees below is given by the following result.
	\begin{lemma}[\cite{FigaNebbia1991}, Lemma 3.7]\label{critede de k non deg rep for finite group}
		Let $Q$ be a finite group, let $H\leq Q$ be a direct product $H_1\times H_2\times ...\times H_s$ of non-trivial subgroups $H_i$ of $Q$ and suppose that the group of inner automorphisms of $Q$ acts by permutation on the set $\{H_1,...,H_s\}$. Then, there exists an irreducible representation $\pi$ of $Q$ without non-zero $H_i$-invariant vectors for every $i=1,...,s$. 
	\end{lemma}
	\noindent  If $T$ is a locally finite tree and if  $\mathcal{T}$ is a subtree of $T$ we set 
	$$\Stab_G(\mathcal{T})=\{g\in G\mid g\mathcal{T}\subseteq\mathcal{T}\}\mbox{ and }\Fix_G(\mathcal{T})=\{g\in G\mid gv=v\qq\forall v\in V(\mathcal{T})\}.$$
	We obtain the following proposition.
	\begin{proposition}\label{existence criterion}
		Let $T$ be a locally finite tree, let $G\leq \Aut(T)$ be a closed subgroup, let $\mathcal{T}$ be a finite subtree of $T$ and let $\{\mathcal{T}_1,\mathcal{T}_2,...,\mathcal{T}_s\}$ be a set of distinct finite subtrees of $T$ contained in $\mathcal{T}$ such that $\mathcal{T}_i\cup\mathcal{T}_j=\mathcal{T}$ for every $i\not=j$. Suppose that $\Stab_G(\mathcal{T})$ acts by permutation on the set $\{\mathcal{T}_1,\mathcal{T}_2,...,\mathcal{T}_s\}$ and that $\Fix_G(\mathcal{T})\subsetneq \Fix_G(\mathcal{T}_i)\subsetneq \Stab_G(\mathcal{T})$. Then, there exists an irreducible representation of $\Stab_G(\mathcal{T})/\Fix_G(\mathcal{T})$ without non-zero $\Fix_G(\mathcal{T}_i)/\Fix_G(\mathcal{T})$-invariant vector for every $i=1,...,s$.
	\end{proposition}
	\begin{proof}
		Since $\mathcal{T}$ is a finite subtree of $T$, $\Stab_G(\mathcal{T})$ and $\Fix_G(\mathcal{T})$ are compact open subgroups of $G$. Since $\Stab_G(\mathcal{T})$ is the normalizer of $\Fix_G(\mathcal{T})$ notice that $\Stab_G(\mathcal{T})/\Fix_G(\mathcal{T})$ is a finite group and our hypotheses ensures that every $H_i=\Fix_G(\mathcal{T}_i)/\Fix_G(\mathcal{T})$ is a non-trivial  subgroup of $\Stab_G(\mathcal{T})/\Fix_G(\mathcal{T})$. On the other hand, for every $i\not=j$, $\mathcal{T}_i\cup \mathcal{T}_j=\mathcal{T}$ which ensures that $H_i\cap H_j=\{1_{\Stab_G(\mathcal{T})/\Fix_G(\mathcal{T})}\}$ and that the supports of elements of $\Fix_G(\mathcal{T}_i)$ and $\Fix_G(\mathcal{T}_j)$ are disjoint from one an other. This implies that the elements of $H_i$ and $H_j$ commute with one another and that the subgroup of $\Stab_G(\mathcal{T})/\Fix_G(\mathcal{T})$ generated by $\bigcup_{i=1}^s H_i$ is isomorphic to $H_1\times H_2\times ...\times H_s$. Since the elements of $\Stab_G(\mathcal{T})$ act by permutation on  $\{\mathcal{T}_1,...,\mathcal{T}_s\}$, notice that the group of inner automorphisms of $\Stab_G(\mathcal{T})$ acts by permutation on $\{\Fix_G(\mathcal{T}_1),...,\Fix_G(\mathcal{T}_s)\}$. The result follows from Lemma \ref{critede de k non deg rep for finite group}.
	\end{proof}
	\newpage
	
	\part{Applications}
	\section{The full group of automorphisms of a semi-regular tree}\label{Application to Aut T}
	
	The purpose of this chapter is to apply our axiomatic framework to the full group of automorphisms $\Aut(T)$ of a thick semi-regular tree $T$ and more generally to groups of automorphisms of trees satisfying the Tits independence property (Definition \ref{equation prop ind de tits}). This chapter is therefore redundant from the point of view of new results (see \cite{Ol'shanskii1977},\cite{Amann2003}). It serves instead as a section allowing the reader to interpret the machinery developed in Chapters \ref{chaptitre intro} and \ref{section generalisation of Ol'shanskii machinery} in concrete and well understood cases.
	
	Let $T$ be a $(d_0,d_1)$-semi-regular tree with $d_0,d_1\geq 3$, let $\Aut(T)$ be the group of automorphisms of $T$ equipped with the permutation topology and let $\mathfrak{T}$ be the set of non-empty complete finite subtrees of $T$. 
	\begin{definition}\label{definition Hypoth H}
		A closed subgroup $G\leq \Aut(T)$ is said to satisfy the hypothesis \ref{Hypothese H Tree} if for all $\mathcal{T},\mathcal{T}'\in \mathfrak{T}$ we have 
		\begin{equation}\tag{$H$}\label{Hypothese H Tree}
		\Fix_G(\mathcal{T}')\leq \Fix_G(\mathcal{T})\mbox{ if and only if }\mathcal{T}\subseteq \mathcal{T}'. 
		\end{equation}
	\end{definition}
	\noindent For such groups, we are going to show that $\mathcal{S}$ is a generic filtration and that:
	\begin{itemize}
		\item $\mathcal{S}\lb 0\rb=\{\Fix_{G}(v)\lvert v\in V(T)\}$.
		\item $\mathcal{S}\lb 1\rb=\{\Fix_{G}(e)\lvert e\in E(T)\}$.
		\item $\mathcal{S}\lb l\rb=\{\Fix_{G}(\mathcal{T})\lvert \mathcal{T}\in \mathfrak{T} \mbox{ and } \mathcal{T}\mbox{ has }l-1 \mbox{ interior vertices}\}\q \forall l\geq 2.$
	\end{itemize}
	In particular, the spherical, special and cuspidal representations as defined on page \pageref{definition de spheric special cuspidal} correspond respectively to the irreducible representations at depth $0$, $1$ and bigger than $2$. The purpose of this chapter is to prove the following theorem.
	\begin{theorem}\label{Theorem classif pour Aut(T)}
		Let $G\leq \Aut(T)$ be a closed non-discrete unimodular subgroup satisfying the hypothesis \ref{Hypothese H Tree} and the Tits independence property, then $\mathcal{S}$ is a generic filtration of $G$ that factorizes$^+$ at all depth $l\geq 2$. 
	\end{theorem}
	\noindent Together with Theorem \ref{la version paki du theorem de classification}, this provides a description of the equivalence classes of cuspidal representations in terms of irreducible representations of a family of finite groups; that is the family of $\mathcal{S}$-standard representations of the group of automorphisms the corresponding seed. In Chapters \ref{application IPk}, we will give examples of different generic filtrations that also factorizes$^+$ for $\Aut(T)$. In particular, different descriptions of the equivalence classes of cuspidal representations of $\Aut(T)$ can be found in these notes.
	\begin{remark}
		If $G$ is locally $2$-transitive, the above generic filtration $\mathcal{S}$ does not factorizes at depth $1$. Indeed, let $e$ and $f$ be two different edges containing a common vertex $v\in V(T)$, let $U=\Fix_G(e)$, $V=\Fix_G(f)$ and let $W$ be in the conjugacy class of an element of $\mathcal{S}\lb 0 \rb$ such that $U\subseteq W$. We are going to show that $W \not\subseteq V U$. From the definition of $W$, Lemma \ref{Lemma la startification de S pour Aut(T)} below ensures the existence of a vertex $w\in V(T)$ such that $W=\Fix_G(w)$. Since $U\subseteq W$ and since $G$ satisfies the hypothesis \ref{Hypothese H Tree}, we must have $w\in e$. However, if $w\not=v$, there exists an element $g\in W$ that does not fix $v$ (because $\Fix_G(e)\not=\Fix_G(v)$) but every element of $VU$ fixes $v$. Moreover, if $w=v$, there exists $g\in W$ such that $ge=f$ (because $G$ is locally $2$-transitive) but no element of $VU$ maps $e$ to $f$. In both cases, this proves that $W\not\subseteq VU$ and therefore that $\mathcal{S}$ does not factorize at depth $1$.
	\end{remark}

	\subsection{Preliminaries}\label{section application to aut T termino}
	The purpose of this section is to settle our formalism on trees and their groups of automorphisms. If $\Gamma$ is a graph, we denote by $V(\Gamma)$ its set of vertices, by $E(\Gamma)$ its set of edges and we equip $V(\Gamma)$ with the metric $d_\Gamma$ given by the length of geodesics. An isometry $g:V(\Gamma)\rightarrow V(\Gamma)$ for this metric is called an \tg{automorphism} of $\Gamma$ and we denote by $\Aut(\Gamma)$ the group of automorphisms of $\Gamma$. This group embeds naturally in $\Sym(V(\Gamma))$ and is therefore naturally equipped with the \tg{permutation topology}. A basis of neighbourhoods of the identity of $\Aut(\Gamma)$ for this topology is given by the sets
	$$\Fix_{\Aut(\Gamma)}(F)=\{g\in \Aut(\Gamma)\mid gv=v\qq \forall v\in F \}$$
	where $F\subseteq V(\Gamma)$ is a finite set of vertices. We recall that for a locally finite graph $\Gamma$, each $\Fix_{\Aut(\Gamma)}(F)$ with finite $F\subseteq V(\Gamma)$ is a compact open subgroup of $\Aut(\Gamma)$ and that $\Aut(\Gamma)$ is a second-countable totally disconnected locally compact group. A \tg{tree} is connected graph with no loop nor cycles and it is called \textbf{thick} if the degree of each vertex is bigger than $3$. If $T$ is a tree, a graph $\mathcal{T}$ is called a \tg{subtree} of $T$ if $\mathcal{T}$ is a tree and  $V(\mathcal{T})\subseteq V(T)$. Notice that a subtree $\mathcal{T}\subseteq T$ is completely determined by its set of vertices. Therefore, when it leads to no confusion, we identify $\mathcal{T}$ with its set of vertices. A tree $T$ is called $(d_0,d_1)$-\tg{semi-regular} if there exists a bipartition $V(T)=V_0\sqcup V_1$ of $T$ such that each vertex of $V_i$ has degree $d_i$ and every edge of $T$ contains exactly one vertex in each $V_i$. Finally, we recall that for a thick semi-regular tree  $T$, the group $\Aut(T)$ is non-discrete and unimodular. 
	 
	\subsection{Factorization of the generic filtration $\mathcal{S}$}\label{section generic filtration of Aut(T)}
	The purpose of this section is to prove Theorem \ref{Theorem classif pour Aut(T)}. Let $T$ be a thick semi-regular tree. Let $G$ be a closed subgroup of $\Aut(T)$, let $\mathfrak{T}$ be the set of non-empty complete finite subtrees of $T$ and let $\mathcal{S}=\{\Fix_{G}(\mathcal{T})\lvert \mathcal{T}\in \mathfrak{T}\}$. The following lemma ensures that $\mathcal{S}$ is a generic filtration of $G$ and provides a description of the conjugacy classes at height $l$ with respect to $\mathcal{S}$ if $G$ satisfies the hypothesis \ref{Hypothese H Tree}.
	\begin{lemma}\label{Lemma la startification de S pour Aut(T)}
		Let $G$ be a closed non-discrete unimodular subgroup of $\Aut(T)$ satisfying the hypothesis \ref{Hypothese H Tree}. Then, $\mathcal{S}$ is a generic filtration of $G$ and we have:
		\begin{itemize}
			\item $\mathcal{S}\lb 0\rb=\{\Fix_{G}(v)\lvert v\in V(T)\}.$
			\item $\mathcal{S}\lb 1\rb=\{\Fix_{G}(e)\lvert e\in E(T)\}.$
			\item $\mathcal{S}\lb l\rb=\{\Fix_{G}(\mathcal{T})\lvert \mathcal{T}\in \mathfrak{T} \mbox{ has }l-1 \mbox{ interior vertices}\}\q \forall l\geq 2$
		\end{itemize}
	\end{lemma}
	\begin{proof}
		For every $\mathcal{T}\in \mathfrak{T}$ and every $g\in \Aut(T)$, $g\Fix_{G}(\mathcal{T})g^{-1}$ coincides with the fixator $\Fix_{G}(g\mathcal{T})$. In particular, the elements of $\mathcal{F}_\mathcal{S}=\{\mathcal{C}(U)\lvert U\in \mathcal{S}\}$ are of the form
		\begin{equation*}
		\mathcal{C}(\Fix_{G}(\mathcal{T}))=\{\Fix_{G}(g\mathcal{T})\lvert g\in G\}
		\end{equation*}
		with $\mathcal{T}\in \mathfrak{T}$. Therefore, $\forall \mathcal{T},\mathcal{T}'\in \mathfrak{T}$ we have that $\mathcal{C}(\Fix_{G}(\mathcal{T}'))\leq \mathcal{C}(\Fix_{G}(\mathcal{T}))$ if and only if there exists $g\in G$ such that $\Fix_{G}(\mathcal{T})\leq \Fix_{G}(g\mathcal{T}')$. Since $G$ satisfies the hypothesis \ref{Hypothese H Tree}, this implies that $\mathcal{C}(\Fix_{G}(\mathcal{T}'))\leq \mathcal{C}(\Fix_{G}(\mathcal{T}))$ if and only if there exists $g\in G$ such that $g\mathcal{T}'\subseteq \mathcal{T}$. Since $\mathfrak{T}$ is stable under the action of $\Aut(T)$, for every chain $C_0\lneq C_1\lneq...\lneq C_{n-1}\lneq C_{n}$ of elements of $\mathcal{F}_{\mathcal{S}}$, there exists a chain $\mathcal{T}_0\subsetneq\mathcal{T}_1\subsetneq...\subsetneq\mathcal{T}_{n-1}\subsetneq\mathcal{T}_n$ of elements of $\mathfrak{T}$ such that $C_r=\mathcal{C}(\Fix_G(\mathcal{T}_r))$. Reciprocally, for every chain $\mathcal{T}_0\subsetneq\mathcal{T}_1\subsetneq...\subsetneq\mathcal{T}_{n-1}\subsetneq\mathcal{T}$ of elements of $\mathfrak{T}$ contained in a subtree $\mathcal{T}\in \mathfrak{T}$ we obtain a chain $\mathcal{C}(\Fix_G(\mathcal{T}_0))\lneq\mathcal{C}(\Fix_G(\mathcal{T}_1))\lneq...\lneq\mathcal{C}(\Fix_G(\mathcal{T}_{n-1}))\lneq\mathcal{C}(\Fix_G(\mathcal{T}))$ of elements of $\mathcal{F}_{\mathcal{S}}$ with maximal element $\mathcal{C}(\Fix_G(\mathcal{T}))$. This proves that the height of $\mathcal{C}(\Fix_G(\mathcal{T}))$  is the maximal length of a strictly increasing chain of elements of $\mathfrak{T}$ contained in $\mathcal{T}$. The result then follows from the following observations. If $\mathcal{T}$ is a vertex, it does not contain any non-empty proper subtree. If $\mathcal{T}$ is an edge, every maximal strictly increasing chain of non-empty complete subtree of $\mathcal{T}$ is of the form $\mathcal{T}_0\subsetneq \mathcal{T}$ where $\mathcal{T}_0$ is a vertex. If $\mathcal{T}$ is a complete finite subtrees of $T$ with $r\geq 1$ interior vertices, every maximal strictly increasing chain of non-empty complete subtrees of $\mathcal{T}$ is of the form $\mathcal{T}_0\subsetneq\mathcal{T}_1\subsetneq...\subsetneq\mathcal{T}_{n-1}\subsetneq\mathcal{T}$ 
		where $\mathcal{T}_0$ is a vertex, $\mathcal{T}_1$ is an edge and $\mathcal{T}_i$ contains $i-1$ interior vertices for every $i\geq 2$.
	\end{proof}
	
	\begin{lemma}\label{la diffenrence entre les fixing group alors la difference entre les graph}
		 $\Aut(T)$ satisfies the hypothesis \ref{Hypothese H Tree}. 
	\end{lemma}
	\begin{proof}
		It is clear that $\Fix_{\Aut(T)}(\mathcal{T}')\leq \Fix_{\Aut(T)}(\mathcal{T})$ if $\mathcal{T}\subseteq \mathcal{T}'$. Now, suppose that $\mathcal{T}\not\subseteq \mathcal{T}'$ and let us show that $\Fix_{\Aut(T)}(\mathcal{T}')\not\subseteq \Fix_{\Aut(T)}(\mathcal{T})$. Since $\mathcal{T}\not\subseteq\mathcal{T}'$ there exists a vertex $v_\mathcal{T}\in V(\mathcal{T})-V(\mathcal{T}')$ and since $\mathcal{T}$,$\mathcal{T}'$ are subtrees of $T$, there exists a unique vertex $w_{\mathcal{T}}\in V(\mathcal{T'})$ that is closer to $v_{\mathcal{T}}$ than every other vertex of $\mathcal{T}'$. In particular $\mathcal{T}'$ is a subtree of the half-tree $$T(w_{\mathcal{T}},v_\mathcal{T})=\{v\in V(T)\lvert d_T(w_\mathcal{T},v)\lneq d_T(v_\mathcal{T},v)\}$$  and we have that $\Fix_{\Aut(T)}(T(w_{\mathcal{T}},v_\mathcal{T}))\subseteq \Fix_{\Aut(T)}(\mathcal{T}')$. On the other hand, $\Fix_{\Aut(T)}(T(w_{\mathcal{T}},v_\mathcal{T}))$ acts transitively on the set  $$\{v\in V(T)-V(T(w_{\mathcal{T}},v_\mathcal{T}))\lvert d_T(w_\mathcal{T},v)=d_T(w_\mathcal{T},v_\mathcal{T})\}.$$ 
		Since $T$ is thick, this set of vertices contains $v_{\mathcal{T}}$ and at least one other vertex. This proves the existence of an element of $\Fix_{\Aut(T)}(T(w_{\mathcal{T}},v_\mathcal{T}))$ which does not fix $v_\mathcal{T}$ and that $\Fix_{\Aut(T)}(\mathcal{T}')\not\subseteq \Fix_{\Aut(T)}(\mathcal{T})$. 
	\end{proof}
	
	Our next task is to prove that $\mathcal{S}$ factorizes$^+$ at all depth $l\geq 2$ for groups satisfying the Tits independence property.
	\begin{definition}\label{equation prop ind de tits}
		a group $G\leq \Aut(T)$ satisfies the \tg{Tits independence property} if for any two adjacent vertices $v,v'\in V(T)$, the fixator of the edge $\{v,v'\}$ satisfies 
		$$\Fix_G(\{v,v'\})= \Fix_G(T(v,v')) \Fix_G(T(v',v))$$ where $T(w,v)=\{u\in V(T)\lvert d_T(w,u)\lneq d_T(v,u)\}$.
	\end{definition}
	
	In fact, if $G$ satisfies the Tits independence property and if $\mathcal{T}$ is a complete proper subtree of $T$ containing an edge we have that
	\begin{equation}\label{equation pour la prop d'indep de tits}
	\Fix_G(\mathcal{T})=\prod_{f\in \partial E_o(\mathcal{T})} \Fix_{G}(Tf)\cap \Fix_G(\mathcal{T})
	\end{equation}
	where $ \partial E_o(\mathcal{T})$ denotes the set of terminal edges of $\mathcal{T}$ oriented in such a way that the terminal vertex of any $f\in  \partial E_o(\mathcal{T})$ is a leaf of $\mathcal{T}$ and where $Tf$ denotes the half-tree of $T$ of vertices which are closer to the origin of $f$ than to its terminal vertex. A more general version of this result is given by Lemma \ref{independence on trees} or by Proposition \ref{alternative definition iof property IPk} below. The following result is well-known by the expert but we prove it for completeness of the argument. We refer to \cite[Lemma $3.1$]{FigaNebbia1991} for the full group of automorphisms of a regular tree and to \cite[Lemma $19$]{Amann2003} for groups satisfying the Tits independence property.
	
	\begin{lemma}\label{independence figa nebbia 3.1}
		Let $G\leq \Aut(T)$ be a subgroup satisfying the Tits independence property, let $\mathcal{T}$, $\mathcal{T}'$ be a complete finite subtrees of $T$ such that $\mathcal{T}$ contains at least one interior vertex and such that $\mathcal{T}'$ does not contain $\mathcal{T}$. Then, there exists a complete proper subtree $\mathcal{R}$ of $\mathcal{T}$ such that $\Fix_{G}(\mathcal{R})\subseteq \Fix_{G}(\mathcal{T}') \Fix_{G}(\mathcal{T})$. Furthermore, if $G$ satisfies in addition the hypothesis \ref{Hypothese H Tree}, $\mathcal{R}$ can be chosen in such a way that $\mathcal{C}(\Fix_{G}(\mathcal{R}))$ has the height of $\mathcal{C}(\Fix_G(\mathcal{T}))$ less $1$ in $\mathcal{F}_\mathcal{S}$.
	\end{lemma}
	\begin{proof}
		Since $\mathcal{T}$ and $\mathcal{T}'$ are complete and since $\mathcal{T}'$ does not contain $\mathcal{T}$ there exists a vertex $w_\mathcal{T}$ of $\mathcal{T}$ such that all vertices adjacent to $w_\mathcal{T}$ but possibly one are leaves of $\mathcal{T}$ and none of those leaves is a vertex of $\mathcal{T}'$. Since $\mathcal{T}$ is complete it contains every of the oriented edges of $T$ with terminal vertex $w_\mathcal{T}$. Furthermore, for one of those oriented edges that we denote by $e$, the half-tree $Te\cup e$ contains $\mathcal{T}'$. Let $\mathcal{R}=\lb Te\cup e\rb\cap \mathcal{T}$. Since $\mathcal{R}$ is complete, contains an edge and since $G$ satisfies the Tits independence property, notice that $$\Fix_G(\mathcal{R})=\big\lb\Fix_{G}(Te)\cap  \Fix_G(\mathcal{R})\big\rb\prod_{f\in \partial E_o(\mathcal{R})-\{e\}} \big\lb\Fix_{G}(Tf)\cap \Fix_G(\mathcal{R})\big\rb.$$ On the other hand, $\forall f\in \partial E_o(\mathcal{R})-\{e\}$ we have that $\mathcal{T}\subseteq Tf$ and therefore that $\Fix_{G}(Tf)\cap \Fix_G(\mathcal{R})\subseteq \Fix_G(\mathcal{T})$. Furthermore, since $\mathcal{T}'\subseteq Te\cup e$, we obtain that $\Fix_{G}(Te)\cap \Fix_G(\mathcal{R})\subseteq \Fix_G(\mathcal{T}')$. The desired inclusion follows.
		Now, let us prove the second part of the lemma. Suppose that $G$ satisfies the hypothesis \ref{Hypothese H Tree} and let $\mathcal{P}$ be any complete proper subtree of $\mathcal{T}$ such that $$\Fix_{G}(\mathcal{P})\subseteq \Fix_{G}(\mathcal{T}') \Fix_{G}(\mathcal{T})$$ (notice that the existence of such $\mathcal{P}$ is guaranteed by the first part of the proof). If $\mathcal{T}$ has exactly one interior vertex, there exists an edge $\mathcal{R}$ in $\mathcal{T}$ such that $\mathcal{P}\subseteq \mathcal{R} \subsetneq \mathcal{T}$. On the other hand, if $\mathcal{T}$ has more than one interior vertex, there exists a complete subtree $\mathcal{R}$ of $\mathcal{T}$ with one less interior vertex than $\mathcal{T}$ such that $\mathcal{P}\subseteq \mathcal{R}\subsetneq \mathcal{T}$. In both cases, this implies that $\Fix_{G}(\mathcal{R})\subseteq \Fix_{G}(\mathcal{P})\subseteq \Fix_{G}(\mathcal{T}')\Fix_{G}(\mathcal{T})$ and $\mathcal{R}$ is as desired. 
	\end{proof}
	\begin{proof}[Proof of Theorem \ref{Theorem classif pour Aut(T)}]
		To prove that $\mathcal{S}$ factorizes$^+$ a depth $l\geq 2$, we shall successively verify the three conditions of the Definition \ref{definition olsh facto}. 
		
		First, we need to prove that for all $U$ in the conjugacy class of an element of $\mathcal{S}\lb l \rb$ and every $V$ in the conjugacy class of an element of $\mathcal{S}$ such that $V \not\subseteq U$, there exists a $W$ in the conjugacy class of an element of $\mathcal{S}\lb l -1\rb$ and $U\subseteq W \subseteq V U.$
		Let $U$ and $V$ be as above. Since $\mathfrak{T}$ is stable under the action of $G$, since $G$ satisfies the hypothesis \ref{Hypothese H Tree} and as a consequence of Lemma \ref{Lemma la startification de S pour Aut(T)} there exists subtrees $\mathcal{T},\mathcal{T}'\in \mathfrak{T}$ such that $\mathcal{T}$ has  $l-1$ interior vertices, $\mathcal{T}'$ does not contain $\mathcal{T}$, $U=\Fix_{G}(\mathcal{T})$ and $V=\Fix_{G}(\mathcal{T}')$. Hence, as a consequence of Lemma \ref{independence figa nebbia 3.1}, there exists a proper subtree $\mathcal{R}$ of $\mathcal{T}$ such that $\Fix_{G}(\mathcal{R})$ is conjugate to an element of $\mathcal{S}\lb l-1 \rb$ and  $$\Fix_{G}(\mathcal{R})\subseteq \Fix_{G}(\mathcal{T}')\Fix_{G}(\mathcal{T}).$$ This proves the first condition. Next, we need to prove that 
				\begin{equation*}
				N_{G}(U, V)= \{g\in {G} \lvert g^{-1}Vg\subseteq U\}
				\end{equation*}
				is compact for every $V$ in the conjugacy class of an element of $\mathcal{S}$. Just as before, notice that $V=\Fix_{G}(\mathcal{T}')$ for some $\mathcal{T}'\in \mathfrak{T}$. Furthermore, since $G$ satisfies the hypothesis \ref{Hypothese H Tree} we have 
				\begin{equation*}
				\begin{split}
				N_{G}(U, V)&= \{g\in {G} \lvert g^{-1}Vg\subseteq U\}= \{g\in {G} \lvert g^{-1}\Fix_{G}(\mathcal{T}')g\subseteq \Fix_{G}(\mathcal{T})\}\\
				&= \{g\in {G} \lvert \Fix_G(g^{-1}\mathcal{T}')\subseteq \Fix_G(\mathcal{T})\}= \{g\in {G} \lvert g\mathcal{T}\subseteq \mathcal{T}'\}
				\end{split}
				\end{equation*}
				and this set is compact since both $\mathcal{T}$ and $\mathcal{T}'$ are finite. This proves the second condition. Finally, we need to prove that for any subgroup $W$ in the conjugacy class of an element of $\mathcal{S}\lb l-1\rb$ such that $U\subseteq W$ we have
				\begin{equation*}
				W\subseteq N_{G}(U,U) =\{g\in {G} \mid g^{-1}Ug\subseteq U\}.
				\end{equation*} 
				For the same reasons as before, there exist a complete subtree $\mathcal{R}\in \mathfrak{T}$ such that $W = \Fix_{G}(\mathcal{R})$. Furthermore, since $U\subseteq W$ and since $G$ satisfies the hypothesis \ref{Hypothese H Tree}, we have that $\mathcal{R}\subseteq \mathcal{T}$. On the other hand, since $\mathcal{R}$ and $\mathcal{T}$ are both complete finite subtrees and since $\mathcal{R}$ has exactly one less interior vertex than $\mathcal{T}$, every interior vertex of $\mathcal{T}$ belongs to $\mathcal{R}$. Hence, $g\mathcal{T}\subseteq  \mathcal{T}$ for every $g\in \Fix_{G}(\mathcal{R})$. This implies that 
				\begin{equation*}
				\begin{split}
				W= \Fix_{G}(\mathcal{R})&\subseteq \{g\in G \mid g\mathcal{T}\subseteq \mathcal{T}\}= \{g\in {G} \lvert \Fix_{G}(\mathcal{T})\subseteq \Fix_{G}(g\mathcal{T})\}\\
				&=\{g\in {G} \lvert g^{-1}\Fix_{G}(\mathcal{T})g\subseteq \Fix_{G}(\mathcal{T})\}\\
				&=N_{G}(\Fix_{G}(\mathcal{T}),\Fix_{G}(\mathcal{T}))= N_{G}(U,U).
				\end{split}
				\end{equation*} 
	\end{proof}
	In particular, for every complete finite subtree $\mathcal{T}$ containing an interior vertex, Theorem \ref{la version paki du theorem de classification} provides a bijective correspondence between the equivalence classes of cuspidal representations of $G$ with seed $\mathcal{C}(\Fix_G(\mathcal{T}))$ and the equivalence classes of $\mathcal{S}$-standard representations of $\Aut_{G}(\mathcal{C}(\Fix_{G}(\mathcal{T})))$. On the other hand, notice from the above computations that $\Aut_{G}(\mathcal{C}(\Fix_{G}(\mathcal{T})))$ can be be identified with the group of automorphisms of $\mathcal{T}$ coming from the action of $\Stab_{G}(\mathcal{T})=\{g\in G\lvert g\mathcal{T}\subseteq \mathcal{T}\}$ on $\mathcal{T}$ and that, under this identification, the $\mathcal{S}$-standard representations of $\Aut_{G}(\mathcal{C}(\Fix_{G}(\mathcal{T})))$ are the irreducible representations that do not admit non-zero invariant vectors for the fixator of any maximal proper complete subtree  of $\mathcal{T}$. The existence of $\mathcal{S}$-standard representations of $\Aut_{G}(C)$ is shown for any seed $C$ at height $l\geq 2$ in \cite[Theorem 3.9]{FigaNebbia1991}.
	\newpage

	\section{Groups of automorphisms of trees with the property ${\rm IP}_k$}\label{application IPk}
	In this chapter, we apply our machinery to groups of automorphisms of semi-regular trees satisfying the property \ref{IPk}(Definition \ref{definition IPk}) for some positive integer $k\geq 1$ as introduced in \cite{BanksElderWillis2015}. We use the same notations and terminology as in Chapter \ref{Application to Aut T}. Let $T$ be a $(d_0,d_1)$-semi-regular tree with $d_0,d_1\geq 3$. For every finite subtree $\mathcal{T}$ of $T$ we denote by $B_T(\mathcal{T},r)$ or simply by $\mathcal{T}^{(r)}$ the ball of radius $r\geq 0$ around $\mathcal{T}$ that is 
	$$\mathcal{T}^{(r)}=\{v\in V(T)\lvert \exists w\in V(\mathcal{T})\qq\mbox{s.t.}\qq d_T(v,w)\leq r\}.$$
	The purpose of this chapter is to prove Theorem \ref{Theorem IPK avec un arbre} and Theorem \ref{le theorem pour IP_k}. The following proves Theorem \ref{thm B1}.
	\begin{theorem}\label{Theorem IPK avec un arbre}
		Let $G\leq \Aut(T)$ be a closed non-discrete unimodular subgroup satisfying the property \ref{IPk} for some integer $k\geq1$. Let $ \mathcal{P}$ be a complete finite subtree of $T$ containing an interior vertex, let $\Sigma_{\mathcal{P}}$ be the set of maximal complete proper subtrees of $ \mathcal{P}$ and let 
		$$ \mathfrak{T}_{\mathcal{P}}=\{\mathcal{R}\in \Sigma_{\mathcal{P}}\lvert \Fix_G((\mathcal{R}')^{(k-1)})\not\subseteq \Fix_G(\mathcal{R}^{(k-1)})\qq \forall \mathcal{R}'\in \Sigma_{\mathcal{P}}-\{\mathcal{R}\}\}.$$ 
		We suppose that: \begin{enumerate}
			\item $\forall \mathcal{R},\mathcal{R}'\in \mathfrak{T}_{ \mathcal{P}}, \forall g\in G$, we do not have $\Fix_G(\mathcal{R}^{(k-1)})\subsetneq\Fix_G(g(\mathcal{R}')^{(k-1)})$.
			\item For all $\mathcal{R}\in \mathfrak{T}_{ \mathcal{P}}$, $\Fix_G( \mathcal{P}^{(k-1)})\not=\Fix_G(\mathcal{R}^{(k-1)})$. 
			Furthermore, if $\Fix_G( \mathcal{P}^{(k-1)})\subsetneq\Fix_G(g\mathcal{R}^{(k-1)})$ we have $ \mathcal{P}\subseteq g\mathcal{R}^{(k-1)}$. 
			\item $\forall n\in \N,\forall v\in V(T)$, $\Fix_G(v^{(n)})\subseteq \Fix_G( \mathcal{P}^{(k-1)})$ implies $ \mathcal{P}^{(k-1)}\subseteq v^{(n)}$.  
			\item For every $g\in G$ such that $g \mathcal{P}\not= \mathcal{P}$, $\Fix_G( \mathcal{P}^{(k-1)})\not= \Fix_G(g \mathcal{P}^{(k-1)})$.
		\end{enumerate}  Then, there exists a generic filtration $\mathcal{S}_{ \mathcal{P}}$ of $G$ that factorizes$^+$ at depth $1$ such that 
		$$\mathcal{S}_{ \mathcal{P}}\lb 0\rb =\{\Fix_G(\mathcal{R}^{(k-1)})\lvert \mathcal{R}\in \mathfrak{T}_{ \mathcal{P}}\}$$
		$$\mathcal{S}_{ \mathcal{P}}\lb 1\rb=\{\Fix_G( \mathcal{P}^{(k-1)})\}.$$
	\end{theorem}
	Together with Theorem \ref{la version paki du theorem de classification} this lead to a description of the irreducible representations of $G$ that admit non-zero $\Fix_G( \mathcal{P}^{(k-1)})$ invariant vectors but do not admit a non-zero $\Fix_G(\mathcal{R}^{(k-1)})$-invariant vectors for any $\mathcal{R}\in \Sigma_{\mathcal{P}}$.
	
	Furthermore, under additional hypothesis on $G$ we are able to construct an explicit generic filtration that factorizes$^+$ at all depth bigger than a certain constant. Let $q\in \N$ be a non-negative integer. If $q$ is even, let $$\mathfrak{T}_{q}=\bigg\{B_T(v,r)\Big\lvert v\in V(T), r\geq \frac{q}{2}+1\bigg\}\sqcup\bigg\{B_T(e,r)\Big\lvert e\in E(T), r\geq \frac{q}{2}\bigg\}.$$
	If $q$ is odd, let  $$\mathfrak{T}_{q}=\bigg\{B_T(v,r)\Big\lvert v\in V(T), r\geq \frac{q+1}{2}\bigg\}\sqcup\bigg\{B_T(e,r)\Big\lvert e\in E(T), r\geq \frac{q+1}{2}\bigg\}.$$ 
	For any closed subgroup $G\leq \Aut(T)$, we consider the set 
	\begin{equation*} \label{page Sq filtration}
	\mathcal{S}_{q}=\{\Fix_G(\mathcal{T})\lvert \mathcal{T}\in \mathfrak{T}_{q}\}
	\end{equation*}
	of fixators of those subtrees. 
	\begin{definition}\label{definition de Hq}
		A group $G\leq\Aut(T)$ is said to satisfy the hypothesis \ref{Hypothese Hq} if for all $\mathcal{T},\mathcal{T'}\in \mathfrak{T}_{q}$ we have
		\begin{equation}\tag{$H_q$}\label{Hypothese Hq}
		\Fix_G(\mathcal{T}')\leq \Fix_G(\mathcal{T})\mbox{ if and only if }\mathcal{T}\subseteq \mathcal{T}'. 
		\end{equation}
	\end{definition}
	\noindent If $G\leq \Aut(T)$ is a closed non-discrete unimodular subgroup of $\Aut(T)$ satisfying the hypothesis \ref{Hypothese Hq}, Lemma \ref{la forme des Sl pour IPk} below ensures that $\mathcal{S}_{q}$ is a generic filtration of $G$ and:
	\begin{itemize}
		\item  $\mathcal{S}_{q}\lb l\rb = \{\Fix_G(B_T(e,\frac{l+q}{2}))\lvert e\in E(T)\}$ if $q+l$ is even.
		\item  $\mathcal{S}_{q}\lb l\rb = \{\Fix_G(B_T(v,\frac{l+q+1}{2}))\lvert v\in V(T)\}$ if $q+l$ is odd.
	\end{itemize}
	The following proves Theorem \ref{THM B}.
	\begin{theorem}\label{le theorem pour IP_k}
		Let $T$ be a $(d_0,d_1)$-semi-regular tree with $d_0,d_1\geq 3$ and let $G\leq \Aut(T)$ be a closed non-discrete unimodular subgroup satisfying the hypothesis \ref{Hypothese Hq} and the property \ref{IPk} for some integers $q\geq 0$ and $k\geq 1$. Then $\mathcal{S}_{q}$ factorizes$^+$ at all depth $l\geq L_{q,k}$ where $$L_{q,k}=\begin{cases}
		\max\{1,2k-q-1\} \qq& \mbox{if } q\mbox{  is even.} \\
		\max\{1,2k-q\} & \mbox{if } q\mbox{  is odd.} 
		\end{cases}$$
	\end{theorem}
	In particular, for $\Aut(T)$, the generic filtration $\mathcal{S}_{0}$ factorizes$^+$ at all positive depth and Theorem \ref{la version paki du theorem de classification} leads to a second description of the cuspidal representations of $\Aut(T)$. Furthermore, Theorem \ref{le theorem pour IP_k} also applies to certain groups for which the representation theory did not appear in the literature before. For instance we have the following corollary. 
	\begin{corollary}\label{corollary this apply to radu group}
		Let $T$ be a $(d_0,d_1)$-semi-regular tree with $d_0,d_1\geq 6$ and let $G\leq \Aut(T)$ be a closed subgroup acting $2$-transitively on the boundary $\partial T$ and whose local action at each vertex contains the alternating group of corresponding degree. Then, there exists a constant $k\in \N$ such that the generic filtration $\mathcal{S}_{0}$ of $G$ factorizes$^+$ at all depth $l\geq 2k-1$. 
	\end{corollary}
	In fact, if $G$ is in addition simple, on can show (without relying on the property \ref{IPk}) that the generic filtration $\mathcal{S}_0$ factorizes$^+$ at all positive depth (see \cite{SemalR2022}) which leads to a classification of the cuspidal representations for those groups.
	
	\subsection{Preliminaries}\label{reminders prop IP k}
	The purpose of this section is to recall the property ${\rm IP}_k$ and give an equivalent formulations that will be useful to prove that the generic filtrations below factorize. Let $E_o(T)$ be the set of oriented edges of $T$. For every $e\in E_o(T)$, let $\bar{e}$ be the oriented edge with opposite orientation, let $o(e)$ denotes the origin of $e$ and $t(e)$ its terminal vertex. Recall from page \pageref{equation prop ind de tits} that for $e\in E_o(T)$, $Te$ denotes the set of vertices that are closer to $o(e)$ than to $t(e)$ that is
	\begin{equation*}
	Te= T(o(e),t(e)) =\{ v\in V(T)\mid d_T(o(e),v)\lneq d_T(t(e),v)\}.
	\end{equation*}
	\begin{definition}[\cite{BanksElderWillis2015}]\label{definition IPk} A group $G\leq \Aut(T)$ satisfies the \tg{property ${\rm IP}_k$} for some integer $k\geq 1$ if for all $e\in E_o(T)$ we have 
	\begin{equation}\tag{${\rm IP}_k$}\label{IPk}
		\Fix_{G}(e^{(k-1)})=\big\lb\Fix_{G}(Te)\cap\Fix_{G}(e^{(k-1)}) \big\rb \big\lb \Fix_{G}(T\bar{e})\cap\Fix_{G}(e^{(k-1)})\big\rb.
	\end{equation}
	\end{definition} 
	In particular, under our convention, ${\rm IP}_1$ is the Tits independence property. We now recall a few results from \cite{BanksElderWillis2015}.
	\begin{lemma}[{\cite[Proposition 5.3.]{BanksElderWillis2015}}]\label{IPk then IPk' for k' geq k}
		Let $G\leq \Aut(T)$ satisfies \ref{IPk}. Then, $G$ satisfies ${\rm IP}_{k'}$ for every $ k'\geq k$.
	\end{lemma}
	\noindent Furthermore, natural examples of groups satisfying the property \ref{IPk} can be constructed by $k$-closure. We now recall this notion.
	\begin{definition}\label{definition kclosure}
		Let $G\leq \Aut(T)$. The $k$-\tg{closure} $G^{(k)}$ of $G$ in $\Aut(T)$ is 
		\begin{equation*}
		G^{(k)}= \left\{h\in \Aut(T) \mid \forall v\in V(T), \qq \exists g\in G  \mbox{ with } \restriction{h}{B_T(v,k)}=\restriction{g}{B_T(v,k)}  \right\}.
		\end{equation*}
		A group which coincide with its $k$-closure is called a $k$-\tg{closed} group.
	\end{definition}
	\begin{lemma}[{\cite[Lemmas 3.2, 3.4 and Proposition 5.2]{BanksElderWillis2015}}]\label{lemme Burger mozes 2}
		Let $G\leq \Aut(T)$ and let $k\geq 1$ be an integer. The $k$-closure $G^{(k)}$ is a closed subgroup of $\Aut(T)$ that contains $G$ and satisfies the property \ref{IPk}.
	\end{lemma}
	
	The purpose of the rest of this section is to prove Proposition \ref{G(W) for Sk-1} which plays a similar role than Lemma \ref{independence figa nebbia 3.1} in Chapter \ref{Application to Aut T}. Following this purpose, we reformulate \ref{IPk}. For every subtree $\mathcal{T}$ of $T$ we denote by $\partial\mathcal{T}$ the boundary of $\mathcal{T}$ and by $\partial E_o(\mathcal{T})$ the set of oriented edges $e\in E_o(\mathcal{T})$ of $\mathcal{T}$ such that $t(e)\in \partial\mathcal{T}$. For every $H\leq G$, for every complete finite subtree $\mathcal{T}\subseteq  T$ and $\forall f,f'\in \partial E_o(\mathcal{T})$ notice that the elements of $\Fix_{H}(Tf)$ and $\Fix_{H}(Tf')$ commute with one another since their respective support  are disjoint. This lift any abuse of notation in the group equality of the following proposition.  
	\begin{proposition}\label{alternative definition iof property IPk}
		Let $G\leq \Aut(T)$. Then $G$ satisfies the property ${\rm IP}_k$ if and only if for every complete finite subtree $\mathcal{T}\subseteq T$ containing an edge we have 
		\begin{equation*}\label{formula for G(Bx,k-1) if IPk}
			\Fix_G(\mathcal{T}^{(k-1)})=\prod_{f\in \partial E_o(\mathcal{T})}^{}\big\lb \Fix_{G}(Tf)\cap \Fix_G(\mathcal{T}^{(k-1)})\big\rb.
		\end{equation*}
	\end{proposition}
	\begin{proof}
		The reverse implication is trivial. To prove the forward implication we apply an induction on the number of interior vertices of $\mathcal{T}$ and treat multiple cases. If $\mathcal{T}$ has no interior vertex, it is an edge and the group equality corresponds exactly to the property ${\rm IP}_k$. If $\mathcal{T}$ has exactly one interior vertex, there exists $v\in V(T)$ such that $\mathcal{T}=B_T(v,1)$. In particular, we have $\mathcal{T}^{(k-1)}=v^{(k)}$. Let $g\in \Fix_G(v^{(k)})$ and let $\partial E_o(\mathcal{T})=\{f_1,..., f_n\}$. For every $f\in E_o(\mathcal{T})$ notice that $f^{(k-1)} \subseteq v^{(k)}$ and therefore that $\Fix_G(v^{(k)})\leq \Fix_G(f^{(k-1)})$. In particular, since $G$ satisfies the property \ref{IPk}, $g=g_{f_1}g_{\bar{f_1}}$ where $g_{f_1}\in \Fix_{G}(Tf_1)\cap \Fix_G(f_1^{(k-1)})$ and $g_{\bar{f_1}}\in  \Fix_{G}(T\bar{f_1})\cap\Fix_G(f_1^{(k-1)})$. On the other hand, since $f_1\in \partial E_o(\mathcal{T})$ we have that $v^{(k)}\subseteq  T{f_1} \cup f_1^{(k-1)}$ which implies that $\Fix_{G}(T{f_1})\cap \Fix_G(f_1^{(k-1)})\subseteq \Fix_G(v^{(k)})$. This proves that $g_{f_1}\in \Fix_{G}(Tf_1)\cap \Fix_G(v^{(k)})$ and since $g\in \Fix_G(v^{(k)})$ we obtain that $g_{\bar{f_1}}\in  \Fix_{G}(T\bar{f_1})\cap \Fix_G(v^{(k)})$. In particular, $g_{\bar{f_1}}\in \Fix_G(f_2^{(k-1)})$ and since $G$ satisfies the property \ref{IPk} we obtain that $g_{\bar{f_1}}=g_{f_2} g_{\bar{f_2}}$ where $g_{f_2}\in \Fix_{G}(T{f_2})\cap \Fix_G(f_2^{(k-1)})$ and $g_{\bar{f_2}}\in \Fix_{G}(T\bar{f_2})\cap \Fix_G(f_2^{(k-1)})$. Furthermore, just as before, since $f_2\in \partial E_o(\mathcal{T})$, we have that $v^{(k)}\subseteq Tf_2 \cup f_2^{(k-1)}$ which implies that $\Fix_{G}(T\bar{f_2})\cap \Fix_G(f_2^{(k-1)})\subseteq \Fix_G(v^{(k)})$ and we conclude just as above that $g_{\bar{f_2}}\in\Fix_{G}(T{\bar{f_2}})\cap \Fix_G(\mathcal{T}^{(k-1)})$. On the other hand, $T{\bar{f_1}}\subseteq T{f_2}$ which implies that $g_{f_2}\in \Fix_G(T{\bar{f_1}})$. Since $g_{\bar{f_1}}\in \Fix_G(T{\bar{f_1}})$, the above decomposition of $g_{\bar{f_1}}$ implies that $g_{\bar{f_2}}\in \Fix_G(T{\bar{f_1}})$  and we obtain that $g_{\bar{f_2}}\in \Fix_{G}(T\bar{f_1}\cup T\bar{f_2})\cap \Fix_G(v^{(k)})$. Proceeding iteratively, we obtain that 
		\begin{equation*}
			g= g_{f_1}g_{f_2}... g_{f_n}g_{\bar{f_n}}
		\end{equation*}
		for some $g_{f_i}\in \Fix_{}(T{f_i})\cap \Fix_G(v^{(k)})$ and some $g_{\bar{f_n}}\in \Fix_{G}(\bigcup_{i=1}^nT{\bar{f_i}})\cap \Fix_G(v^{(k)})$. In particular since $\{f_1,..., f_n\}=\partial E_o(\mathcal{T})$, notice that $g_{\bar{f_n}}\in \Fix_G(T)=\{1_G\}$. This proves as desired that
		\begin{equation*}
			g\in \prod_{f\in \partial E_o(\mathcal{T})}^{} \big\lb\Fix_{G}(T{f})\cap \Fix_G(v^{(k)})\big\rb.
		\end{equation*}
			
		If $\mathcal{T}$ has two interior vertices or more, the reasoning is similar. We let $g\in \Fix_G(\mathcal{T}^{(k-1)})$, let $v$ be an interior vertex  of $\mathcal{T}$ at distance $1$ form the boundary $\partial\mathcal{T}$ and let $\mathcal{R}$ be the unique maximal proper complete subtree of $\mathcal{T}$ such that $v$ is not an interior vertex. Notice that $\mathcal{R}$ is a complete subtree of $T$ with one interior vertex less. In particular, our induction hypothesis, ensures that
		\begin{equation*}
			\Fix_G(\mathcal{R}^{(k-1)})= \prod_{f\in \partial E_o(\mathcal{R})}^{} \big\lb\Fix_{G}(Tf)\cap \Fix_G(\mathcal{R}^{(k-1)})\big\rb.
		\end{equation*}
		Notice that exactly one extremal edge $e\in \partial E_o(\mathcal{R})$ of $\mathcal{R}$ is not extremal in $\mathcal{T}$; the oriented edge $e$ such that $o(e)=v$. Furthermore, since $\mathcal{R}^{(k-1)}\subseteq \mathcal{T}^{(k-1)}$, we have $g\in\Fix_G(\mathcal{R}^{(k-1)})$ which implies that
			\begin{equation*}
			g= g_e\prod_{f\in \partial E_o(\mathcal{T})\cap \partial E_o(\mathcal{R})}^{} g_f
			\end{equation*} 
			for some $g_e\in \Fix_{G}(Te)\cap \Fix_G(\mathcal{R}^{(k-1)})$ and some $g_f\in \Fix_{G}(Tf)\cap \Fix_G(\mathcal{R}^{(k-1)})$. Furthermore, for every $f\in \partial E_o(\mathcal{T})\cap \partial E_o(\mathcal{R})$ notice that $$\Fix_{G}(Tf)\cap \Fix_G(\mathcal{R}^{(k-1)})=\Fix_{G}(Tf)\cap\Fix_G(\mathcal{T}^{(k-1)})$$ which implies that $g_f \in \Fix_{G}(Tf)\cap \Fix_G(\mathcal{T}^{(k-1)})$. In particular, since $g\in \Fix_G(\mathcal{T}^{(k-1)})$, the above decomposition implies that $g_e\in\Fix_G(Te)\cap \Fix_G(\mathcal{T}^{(k-1)})$. On the other hand, $v^{(k)}\subseteq T{e} \cup\mathcal{T}^{(k-1)}$ which implies that $g_e\in \Fix_G(v^{(k)})$. Hence, by the first part of the proof, we have
			\begin{equation*}
			g_e = g_{\bar{e}} \prod_{b\in \partial E_o(\mathcal{T}) \cap \partial E_o(v^{(1)})} g_{b}
			\end{equation*}
			for some $g_{\bar{e}}\in \Fix_{G}(T{\bar {e}})\cap \Fix_G(v^{(k-1)})$ and some $g_{b}\in \Fix_{G}(Tb)\cap \Fix_G(v^{(k-1)})$. Furthermore, for every $b\in  \partial E_o(\mathcal{T}) \cap \partial E_o(v^{(1)})$, notice that $\mathcal{T}^{(k-1)}\subseteq Tb \cup v^{(1)}$ and therefore that $g_{b}\in \Fix_{G}(Tb)\cap \Fix_G(\mathcal{T}^{(k-1)})$. In particular, since $g_e\in \Fix_{G}(Te)\cap \Fix_G(\mathcal{T}^{(k-1)})$ the above decomposition of $g$ implies that $g_{\bar{e}}=1_G$. This proves as desired that $g$ belongs to 
			\begin{equation*}
			\prod_{f\in \partial E_o(\mathcal{T})\cap \lb \partial E_o(\mathcal{R})\cup\partial E_o(v^{(1)})\rb }^{} \big\lb\Fix_{G}(Tf)\cap \Fix_G(\mathcal{T}^{(k-1)})\big\rb.
			\end{equation*}
	\end{proof}	 
	\begin{proposition}\label{G(W) for Sk-1}
		Let $G\leq \Aut(T)$ be a subgroup satisfying the property \ref{IPk}, let $\mathcal{T}$, $\mathcal{T}'$ be a complete finite subtrees of $T$ such that $\mathcal{T}$ contains at least one interior vertex and such that $\mathcal{T}'$ does not contain $\mathcal{T}$. Then, there exists a maximal complete proper subtree $\mathcal{R}$ of $\mathcal{T}$ such that $$\Fix_G(\mathcal{R}^{(k-1)})\subseteq \Fix_G((\mathcal{T}')^{(k-1)})\Fix_G(\mathcal{T}^{(k-1)}).$$
	\end{proposition}
	\begin{proof}
		Since $\mathcal{T}$ and $\mathcal{T}'$ are complete and since $\mathcal{T}'$ does not contain $\mathcal{T}$, there exists an extremal edge $b$ of $\mathcal{T}$ which does not belong to $\mathcal{T}'$ and such that $\mathcal{T}'\subseteq T{b}$. Let $\mathcal{R}$ be the maximal complete subtree of $\mathcal{T}$ such that $b\not \subseteq \mathcal{R}$. Notice that there exists a unique $e\in \partial E_o(\mathcal{R})$ which is not extremal in $\mathcal{T}$. Furthermore, we observe that $\mathcal{T}'\subseteq Te\cup e$, that $\mathcal{R}$ is a complete subtree of $T$ containing an edge and that $\mathcal{R}$ has one less interior vertex than $\mathcal{T}$. In particular, Proposition \ref{alternative definition iof property IPk} ensures that 
		\begin{equation*}
		\Fix_G(\mathcal{R}^{(k-1)})= \prod_{f\in \partial E_o(\mathcal{R})} \big \lb\Fix_{G}(Tf)\cap \Fix_G(\mathcal{R}^{(k-1)})\big\rb.
		\end{equation*} 
		However, by construction, we have $(\mathcal{T}')^{(k-1)}\subseteq Te \cup \mathcal{R}^{(k-1)}$ which implies that $$\Fix_{G}(Te)\cap \Fix_G(\mathcal{R}^{(k-1)})\subseteq \Fix_G((\mathcal{T}')^{(k-1)}).$$ 
		On the other hand, notice that $\partial E_o(\mathcal{R})-\{e\}\subseteq \partial E_o(\mathcal{T})$. Therefore, for every $f\in \partial E_o(\mathcal{R})-\{e\}$ we have $\mathcal{T}^{(k-1)}\subseteq Tf \cup \mathcal{R}^{(k-1)}$ which implies that $$\Fix_{G}(Tf)\cap \Fix_G(\mathcal{R}^{(k-1)})\subseteq \Fix_G(\mathcal{T}^{(k-1)}).$$ 
		This proves as desired that
		\begin{equation*}
		\begin{split}
		\Fix_G(\mathcal{R}^{(k-1)})&= \big\lb \Fix_{G}(Te)\cap \Fix_G(\mathcal{R}^{(k-1)})\big\rb \prod_{f\in \partial E_o(\mathcal{R})-\{e\}} \big\lb \Fix_G(Tf)\cap \Fix_G(\mathcal{R}^{(k-1)})\big \rb\\
		& \subseteq  \Fix_G((\mathcal{T}')^{(k-1)})\Fix_G(\mathcal{T}^{(k-1)}).
		\end{split}
		\end{equation*}  
	\end{proof}
	\subsection{Factorization of the generic filtrations}\label{generic filtration fertile IP k}
	Let $T$ be a $(d_0,d_1)$-semi-regular tree with $d_0,d_1\geq 3$. Let $G\leq \Aut(T)$ be a closed non-discrete unimodular subgroup and let $\mu$ be a Haar measure of $G$. The purpose of this section is to prove Theorems \ref{Theorem IPK avec un arbre} and \ref{le theorem pour IP_k}. This requires some preliminaries.
	\begin{lemma}\label{Lemme on peut etendre la la filtration generic de mathcalT}
		 Let $ \mathcal{P}$ be a complete finite subtree of $T$ containing an interior vertex , let $\Sigma_{\mathcal{P}}$ be the set of maximal complete proper subtrees of $ \mathcal{P}$ and let 
		 $$ \mathfrak{T}_{\mathcal{P}}=\{\mathcal{R}\in \Sigma_{\mathcal{P}}\lvert \Fix_G((\mathcal{R}')^{(k-1)})\not\subseteq \Fix_G(\mathcal{R}^{(k-1)})\qq \forall \mathcal{R}'\in \Sigma_{\mathcal{P}}-\{\mathcal{R}\}\}.$$ Suppose that: 
		\begin{enumerate}
			\item $\forall \mathcal{R},\mathcal{R}'\in \mathfrak{T}_{ \mathcal{P}}, \forall g\in G$, we do not have $\Fix_G(\mathcal{R}^{(k-1)})\subsetneq\Fix_G(g(\mathcal{R}')^{(k-1)})$.
			\item For every $\mathcal{R}\in \mathfrak{T}_{ \mathcal{P}}$, $\Fix_G( \mathcal{P}^{(k-1)})\not=\Fix_G(\mathcal{R}^{(k-1)}).$ 
		\end{enumerate}
	Then, there exists a family $\mathfrak{R}\subseteq\mathfrak{T}_{ \mathcal{P}}\cup \{ \mathcal{P}\}\cup\{v^{(n)}\lvert n\in \N, v\in V(T)\}$  of complete finite subtrees of $T$ such that $\mathcal{S}_{ \mathcal{P}}=\{\Fix_G(\mathcal{T}^{(k-1)})\lvert \mathcal{T}\in\mathfrak{R}\}$ is a generic filtration of $G$ such that:
		$$\mathcal{S}_{ \mathcal{P}}\lb 0\rb =\{\Fix_G(\mathcal{R}^{(k-1)})\lvert \mathcal{R}\in \mathfrak{T}_{ \mathcal{P}}\},\qq \mathcal{S}_{ \mathcal{P}}\lb 1\rb=\{\Fix_G( \mathcal{P}^{(k-1)})\}$$
		and $\mu(\Fix_G( \mathcal{P}^{(k-1)})\not = \mu(\Fix_G(\mathcal{T}^{(k-1)}))$ for every $\mathcal{T}\in \mathfrak{R}-\{ \mathcal{P}\}$.
	\end{lemma}
	\begin{proof}
		Since $G$ is non-discrete, there exists a vertex $v\in V(T)$ and an integer $N\geq k$ such that $\mathcal{P}\subsetneq v^{(N)}$ and $\Fix_G(v^{(N+k-1)})\lneq\Fix_G( \mathcal{P}^{(k-1)}).$ 
		We set $\mathfrak{R}=\mathfrak{T}_{ \mathcal{P}}\sqcup\{ \mathcal{P}\}\sqcup\{ v^{(n)}\lvert n\geq N\}$ and let $\mathcal{S}_{ \mathcal{P}}=\{\Fix_G(\mathcal{T}^{(k-1)})\lvert \mathcal{T}\in\mathfrak{R}\}.$ Notice by construction, that for every $\mathcal{T}\in \mathfrak{R}$ there exists $\mathcal{R}\in \mathfrak{T}_{ \mathcal{P}}$ such that $\mathcal{R}\subseteq \mathcal{T}$. In particular, for every $U\in \mathcal{S}_{ \mathcal{P}}$ there exists $\mathcal{R}\in \mathfrak{T}_{ \mathcal{P}}$ such that $U\leq \Fix_G(\mathcal{R}^{(k-1)})$. On the other hand, since $ \mathcal{P}$ is a finite subtree of $T$ notice that $\mathfrak{T}_{ \mathcal{P}}$ is finite. This implies that 
		$$\mu(U)\leq \qq \maxx{\mathcal{R}\in \mathfrak{T}_{ \mathcal{P}}}\big \lb\mu( \Fix_G(\mathcal{R}^{(k-1)}))\big \rb<+\infty\q \q \forall U\in \mathcal{S}_{ \mathcal{P}}$$ 
		and Lemma \ref{lemma bounded basis implies generic filtration} ensures that $\mathcal{S}_{ \mathcal{P}}$ is a generic filtration of $G$. Now, notice for every $\mathcal{T}\in \mathfrak{R}- (\mathfrak{T}_{ \mathcal{P}}\sqcup\{ \mathcal{P}\})$ and every $\mathcal{R}\in \mathfrak{T}_{ \mathcal{P}}$ we have that $\mathcal{R}\subseteq  \mathcal{P}\subseteq \mathcal{T}$. In particular, this implies that $ \Fix_G(\mathcal{T}^{(k-1)}) \subseteq \Fix_G( \mathcal{P}^{(k-1)})\subseteq \Fix_G(\mathcal{R}^{(k-1)})$. On the other hand, by the hypotheses none of those inclusion is an equality which implies that $$\mu(\Fix_G(\mathcal{T}^{(k-1)}))\lneq\mu(\Fix_G( \mathcal{P}^{(k-1)}))\lneq \mu(\Fix_G(\mathcal{R}^{(k-1)})).$$ 
		This proves, as desired, that $\mu(\Fix_G( \mathcal{P})^{(k-1)})\not = \mu(\Fix_G(\mathcal{T}^{(k-1)}))$ for every $\mathcal{T}\in \mathfrak{R}-\{ \mathcal{P}\}$. Furthermore, since $G$ is unimodular, the measure $\mu(U)$ is an invariant of the conjugacy class $\mathcal{C}(U)$ and one realises that $$\mathcal{C}(\Fix_G(\mathcal{R}^{(k-1)}))\lneq\mathcal{C}(\Fix_G( \mathcal{P}^{(k-1)}))\lneq \mathcal{C}(\Fix_G(\mathcal{T}^{(k-1)})).$$ In particular, the depth of every subgroup  $\Fix_G(\mathcal{T}^{(k-1)})\in \mathcal{S}_{ \mathcal{P}}$ for which $\mathcal{T}\in \mathfrak{R} -(\mathfrak{T}_{ \mathcal{P}}\sqcup \{  \mathcal{P}\})$ is at least $2$. Now, let us prove that $$\mathcal{S}_{ \mathcal{P}}\lb 0\rb =\{\Fix_G(\mathcal{R}^{(k-1)})\lvert \mathcal{R}\in \mathfrak{T}_{ \mathcal{P}}\} \mbox{ and }\mathcal{S}_{ \mathcal{P}}\lb 1\rb=\{\Fix_G( \mathcal{P}^{(k-1)})\}.$$
		To this end, let $\mathcal{T}\in \mathfrak{R}$ and let $\mathcal{R}\in \mathfrak{T}_{ \mathcal{P}}$ be such that
		$\mathcal{C}(\Fix_G(\mathcal{T}^{(k-1)}))\leq \mathcal{C}(\Fix_G(\mathcal{R}^{(k-1)})).$ By the definition, this implies the existence of an element $g\in G$ such that $\Fix_G(\mathcal{R}^{(k-1)})\leq  \Fix_G(g\mathcal{T}^{(k-1)})$ and by the first part of the proof we have that $\mathcal{T}\in \mathfrak{T}_{ \mathcal{P}}$. Our hypotheses imply that $\Fix_G(g\mathcal{T}^{(k-1)})=\Fix_G(\mathcal{R}^{(k-1)})$ and therefore that $\mathcal{C}(\Fix_G(\mathcal{T}^{(k-1)}))= \mathcal{C}(\Fix_G(\mathcal{R}^{(k-1)}))$. In particular, this proves that $\{\Fix_G(\mathcal{R}^{(k-1)})\lvert \mathcal{R}\in \mathfrak{T}_{ \mathcal{P}}\}\subseteq \mathcal{S}_{ \mathcal{P}}\lb 0\rb$. On the other hand, we have proved that the depth of every subgroup  $\Fix_G(\mathcal{T}^{(k-1)})$ with  $\mathcal{T}\in \mathfrak{R} -(\mathfrak{T}_{ \mathcal{P}}\sqcup \{  \mathcal{P}\})$ is at least $2$. Since there must exist an element at depth $1$, this implies that 
		$\mathcal{S}_{ \mathcal{P}}\lb 1\rb=\{\Fix_G( \mathcal{P}^{(k-1)})\}$
		and it follows that $\mathcal{S}_{ \mathcal{P}}\lb 0\rb =\{\Fix_G(\mathcal{R}^{(k-1)})\lvert \mathcal{R}\in \mathfrak{T}_{ \mathcal{P}}\}.$
	\end{proof}
	\begin{proof}[Proof of Theorem \ref{Theorem IPK avec un arbre}]
		Let $\mathfrak{R}$ be the family of subtrees of $T$ given by Lemma \ref{Lemme on peut etendre la la filtration generic de mathcalT} and let us consider the generic filtration $\mathcal{S}_{ \mathcal{P}}=\{\Fix_G(\mathcal{T}^{(k-1)})\lvert \mathcal{T}\in \mathfrak{R}\}$ of $G$. In order to prove that $\mathcal{S}_{ \mathcal{P}}$ factorizes$^+$ at depth $1$ we shall successively verify the three conditions of the Definition \ref{definition olsh facto}. 
		
		First, we need to prove that for all $U$ in the conjugacy class of an element of $\mathcal{S}_{\mathcal{P}}\lb l \rb$ and every $V$ in the conjugacy class of an element of $\mathcal{S}_\mathcal{P}$ such that $V \not\subseteq U$, there exists a $W$ in the conjugacy class of an element of $\mathcal{S}_\mathcal{P}\lb l -1\rb$ such that $U\subseteq W \subseteq V U.$
		Let $U$ and $V$ be as above. By the definition of $\mathcal{S}_{ \mathcal{P}}$ there exist $t,h\in G$ and some $\mathcal{T} \in \mathfrak{R}$ such that $U=\Fix_G(t\mathcal{P}^{(k-1)})$ and $V= \Fix_G(h\mathcal{T}^{(k-1)})$. Furthermore, since $V\not \subseteq U$, we have that $ \Fix_G(t^{-1}h\mathcal{T}^{(k-1)})\not \subseteq\Fix_G(\mathcal{P}^{(k-1)})$. In particular, we obtain that $\mathcal{P}\not \subseteq t^{-1}h\mathcal{T}$. This follows from hypothesis $3$ if $\mathcal{T}=v^{(n)}$ for some $n\geq N$, from hypothesis $4$ if $\mathcal{T}= \mathcal{P}$ and from the fact that $t^{-1}h\mathcal{T}^{(k-1)}$ can never contain $ \mathcal{P}^{(k-1)}$ if $\mathcal{T}\in \mathfrak{T}_{ \mathcal{P}}$. In particular, Proposition \ref{G(W) for Sk-1} ensures the existence of a complete finite subtree $\mathcal{Q}\in \Sigma_{ \mathcal{P}}$ such that
		$$\Fix_G(t\mathcal{Q}^{(k-1)})\subseteq \Fix_G(h\mathcal{T}^{(k-1)})\Fix_G( t\mathcal{P}^{(k-1)})= VU.$$
		On the other hand, by the definition of $\mathfrak{T}_\mathcal{P}$ there exists $\mathcal{R}\in \mathfrak{T}_\mathcal{P}$ such that $\Fix_G(t\mathcal{R}^{(k-1)})\subseteq \Fix_G(t\mathcal{Q}^{(k-1)}).$  
		Furthermore, by the definition of $\mathcal{S}_\mathcal{P}$, the group $\Fix_G(t\mathcal{R}^{(k-1)})$ is conjugate to an element of $\mathcal{S}_{ \mathcal{P}}\lb 0\rb$ which proves the first condition. 
		
		Next, we need to prove that $N_{G}(U, V)= \{g\in {G} \lvert g^{-1}Vg\subseteq U\}$
		is compact for every $V$ in the conjugacy class of an element of $\mathcal{S}_\mathcal{P}$. To this end, notice that $V= \Fix_G(h\mathcal{T}^{(k-1)})$ for some some $\mathcal{T} \in \mathfrak{R}$ and some $h\in G$ and that
		\begin{equation*}
		\begin{split}
		N_{G}(U, V)&= \{g\in G \lvert g^{-1}Vg\subseteq U\}\\
		&= \{g\in G \lvert \Fix_{G}(g^{-1}h\mathcal{T}^{(k-1)})\subseteq \Fix_{G}( t\mathcal{P}^{(k-1)})\}
		\end{split}
		\end{equation*}
		This leads to three cases. If $\mathcal{T}=v^{(n)}$ for some $n\in \N$, the hypothesis $3$ implies that $N_{G}(U, V)=\{g\in G\lvert t\mathcal{P}\subseteq g^{-1}hv^{(n)}\}$ which is a compact set. If $\mathcal{T}=\mathcal{T}'$, the hypothesis $4$ ensures that $N_G(V,U)=\{g\in G\lvert g t\mathcal{P}\subseteq  t\mathcal{P}\}$ which is a compact set. If $\mathcal{T}\in \mathfrak{T}_{ \mathcal{P}}$ notice from Lemma \ref{Lemme on peut etendre la la filtration generic de mathcalT} that $\mu(V)$ is strictly smaller than $\mu(U)$ which implies that $N_G(U,V)=\es$. In every cases, this proves the second condition.

		Finally, we need to prove that for any subgroup $W$ in the conjugacy class of an element of $\mathcal{S}_{ \mathcal{P}}\lb 0\rb$ such that $U\subseteq W$ we have
		\begin{equation*}
		W\subseteq N_G(U,U) =\{g\in G \mid g^{-1}Ug\subseteq U\}=\{g\in G \mid g^{-1}Ug=U\}.
		\end{equation*} 
		The hypothesis $4$ of the theorem implies that $N_G(U,U)=\{g\in G\lvert g t\mathcal{P}= t\mathcal{P}\}$. On the other hand, by construction of $\mathcal{S}_{ \mathcal{P}}$, for every $W$ in the conjugacy class of an element of $\mathcal{S}_{ \mathcal{P}}\lb 0\rb$, there exists an element $h\in G$ and a subtree $\mathcal{R}\in \mathfrak{T}_{ \mathcal{P}}$ such that $W=\Fix_G(h\mathcal{R}^{(k-1)})$. Furthermore, if $U\subseteq W$ the hypothesis $2$ implies that $ \mathcal{P}\subseteq t^{-1}h\mathcal{R}^{(k-1)}$. In particular, we obtain that  $W=\Fix_G(h\mathcal{R}^{(k-1)})\subseteq \Fix_G(t\mathcal{P})\subseteq N_G(U,U)$ which proves the third condition. 
	\end{proof}
	
	Our next task is to prove the Theorem \ref{le theorem pour IP_k} concerning groups satisfying the hypothesis \ref{Hypothese Hq}(Definition \ref{definition de Hq}). To this end, let $q\in \N$ be a non-negative integer. If $q$ is even, let $$\mathfrak{T}_{q}=\bigg\{B_T(v,r)\Big\lvert v\in V(T), r\geq \frac{q}{2}+1\bigg\}\sqcup\bigg\{B_T(e,r)\Big\lvert e\in E(T), r\geq \frac{q}{2}\bigg\}.$$
	If $q$ is odd, let  $$\mathfrak{T}_{q}=\bigg\{B_T(v,r)\Big\lvert v\in V(T), r\geq \frac{q+1}{2}\bigg\}\sqcup\bigg\{B_T(e,r)\Big\lvert e\in E(T), r\geq \frac{q+1}{2}\bigg\}.$$ 
	For any closed subgroup $G\leq \Aut(T)$, we consider the set $$\mathcal{S}_{q}=\{\Fix_G(\mathcal{T})\lvert \mathcal{T}\in \mathfrak{T}_{q}\}$$
	of fixators of those subtrees.  
	\begin{lemma}\label{la forme des Sl pour IPk}
		Let $G\leq \Aut(T)$ be a closed non-discrete unimodular subgroup satisfying the hypothesis \ref{Hypothese Hq} for some integer $q\geq 0$. Then, $\mathcal{S}_{q}$ is a generic filtration of $G$ and:
		\begin{itemize}
			\item $\mathcal{S}_{q}\lb l\rb = \{\Fix_G(B_T(e,\frac{l+q}{2}))\lvert e\in E(T)\}$ if $q+l$ is even.
			\item $\mathcal{S}_{q}\lb l \rb=\{ \Fix_G(B_T(v,(\frac{l+q+1}{2})))\mid v\in V(T)\}$ if $q+l$ is odd.
		\end{itemize} 
	\end{lemma}
	\begin{proof}
		For shortening of the formulation and readability, for all $v\in V(T)$ and every $e\in E(T)$ we denote by $Gv$ and $Ge$ their respective orbit under the action of $G$ on $V(T)$ and $E(T)$ and by $v^{(r)}$ and $e^{(r)}$ the balls of radius $r$ around $v$ and $e$ respectively. Since $g\Fix_G(v^{(r)})g^{-1}=\Fix_G(gv^{(r)})$ $\forall g\in G, \forall v\in V(T)$ and $\forall r\in \N$, notice that $\mathcal{C}(\Fix_G(v^{(r)})=\{\Fix_G(w^{(r)})\lvert w\in Gv\}$. Similarly, we have that  $\mathcal{C}(\Fix_G(e^{(r)}))=\{\Fix_G(f^{(r)})\lvert f\in Ge\}$ $\forall e\in E(T)$. Furthermore, for every $\mathcal{T}, \mathcal{T}'\in \mathfrak{T}_{q}$, we have that $\mathcal{C}(\Fix_G(\mathcal{T}'))\leq \mathcal{C}(\Fix_G(\mathcal{T}))$ if and only if there exist $g\in G$ such that $\Fix_G(\mathcal{T})\leq \Fix_G(g\mathcal{T}')$. Therefore, since $G$ satisfies the hypothesis \ref{Hypothese Hq},  we have $\mathcal{C}(\Fix_G(\mathcal{T}'))\leq \mathcal{C}(\Fix_G(\mathcal{T}))$ if and only if there exists $g\in G$ such that $g\mathcal{T}'\subseteq \mathcal{T}$. Furthermore, notice that  $\mathcal{T}_{q}$ is stable under the action of $G$. In particular, for every increasing chain $C_0\lneq C_1\lneq...\lneq C_{n-1}\lneq C_{n}$ of elements of $\mathcal{F}_{\mathcal{S}_{q}}$ there exists a strictly increasing chain $\mathcal{T}_0\subsetneq\mathcal{T}_1\subsetneq...\subsetneq\mathcal{T}_{n-1}\subsetneq\mathcal{T}_n$ of elements of $\mathfrak{T}_{q}$ such that $C_t=\mathcal{C}(\Fix_G(\mathcal{T}_t))$. It follows that the height of an element $\mathcal{C}(\Fix_G(\mathcal{T}))\in \mathcal{F}_{\mathcal{S}_{q}}$ is bounded above by the maximal length of a strictly increasing chain of elements of $\mathfrak{T}_{q}$ contained in $\mathcal{T}$. On the other hand, for every strictly increasing chain $\mathcal{T}_0\subsetneq\mathcal{T}_1\subsetneq...\subsetneq\mathcal{T}_{n}\subsetneq\mathcal{T}$ of elements of $\mathfrak{T}_{q}$ contained in $\mathcal{T}$, we can build a strictly increasing chain $\mathcal{C}(\Fix_G(\mathcal{T}_0))\lneq\mathcal{C}(\Fix_G(\mathcal{T}_1))\lneq...\lneq\mathcal{C}(\Fix_G(\mathcal{T}_{n}))\lneq\mathcal{C}(\Fix_G(\mathcal{T}))$ of elements of $\mathcal{F}_{\mathcal{S}_{q}}$. This proves that the height of $\mathcal{C}(\Fix_G(\mathcal{T}))$ is the maximal length of a strictly increasing chain of elements of $\mathfrak{T}_{q}$ contained in $\mathcal{T}$. The results therefore follows from the following observation.
		If $q$ is even, every such chain is of the form
		$$e_1^{(\frac{q}{2})}\subseteq v_{1}^{(\frac{q}{2}+1)}\subseteq e_2^{(\frac{q}{2}+1)}\subseteq v_{2}^{(\frac{q}{2}+2)}\subseteq...\subseteq\mathcal{T}$$
		where $v_t\in V(T)$ and $e_t\in E(T)$ for all $t$ and if $q$ is odd, every such chain is of the form
			$$v_1^{(\frac{q+1}{2})}\subseteq e_{1}^{(\frac{q+1}{2})}\subseteq v_2^{(\frac{q+1}{2}+1)}\subseteq e_{2}^{(\frac{q+1}{2}+1)}\subseteq...\subseteq\mathcal{T}$$
			where the $v_t\in V(T)$ and $e_t\in E(T)$ for every $t$. 
	\end{proof}
	\begin{lemma}\label{the lemma imply fertility for G S on IPk groups}
		Let $G\leq \Aut(T)$ be a closed unimodular subgroup satisfying the hypothesis \ref{Hypothese Hq} and the property \ref{IPk} for some integers $q\geq 0$, $k\geq 1$ and let $$L_{q,k}=\begin{cases}
		\max\{1,2k-q-1\} \qq& \mbox{if } q\mbox{  is even.} \\
		\max\{1,2k-q\} & \mbox{if } q\mbox{  is odd.} 
		\end{cases}$$ Suppose further that $l,l'\in \N$ are such that $l\geq L_{q,k}$ and $l \leq l'$. Then, for every $U$ in the conjugacy class of an element of $\mathcal{S}_q\lb l \rb$ and every $V$ in the conjugacy class of an element of $\mathcal{S}_q\lb l'\rb$ such that $V\not \subseteq U$, there exists $W\in\mathcal{S}_q\lb l-1\rb$ such that $U\subseteq W\subseteq VU$. 
	\end{lemma}
	\begin{proof} For every $t\in \N$, let $\mathfrak{T}_q\lb t\rb=\{\mathcal{T}\in \mathfrak{T}_q\lvert \Fix_G(\mathcal{T})\in \mathcal{S}\lbrack t\rbrack\}$.
		Notice that $\mathcal{T}_q\lb t\rb$ is stable under the action of $G$. In particular, there exist $\mathcal{T}\in \mathfrak{T}_q\lbrack l\rbrack$ and $\mathcal{T}'\in \mathfrak{T}_{q}\lb l'\rb$ such that $U= \Fix_G(\mathcal{T})$ and $V= \Fix_G(\mathcal{T}')$. Since $V\not \subseteq U$, we have that $\mathcal{T}\not\subseteq \mathcal{T}'$.
		There are four cases to treat depending on the parity of $q$ and $l$. We suppose that $q$ is even (the reasoning with odd $q$ is similar). If $l$ is even, let $k'=\frac{l+q}{2}$. Lemma \ref{la forme des Sl pour IPk} implies the existence of an edge $e\in E(T)$ such that $U=\Fix_G(B_T(e, k'))$. Furthermore, since $l\geq L_{q,k}$ and since $l$ is even, we have $k'\geq k$ and Lemma \ref{IPk then IPk' for k' geq k} implies that $G$ satisfies the property ${\rm IP}_{k'}$. Therefore, Proposition \ref{G(W) for Sk-1} ensures the existence of a vertex $v\in e$ such that 
		$$\Fix_G(B_T(v,k'))\subseteq VU$$
		and Lemma \ref{la forme des Sl pour IPk} ensures that $\Fix_G(B_T(v,k'))\in \mathcal{S}_q\lb l-1\rb$. Finally, notice that $U=\Fix_G(B_T(e,k'))\subseteq \Fix_G(B_T(v,k'))$ since $v\in e$. 
		
		If $l$ is odd, let $k'=\frac{l+q+1}{2}$. Lemma \ref{la forme des Sl pour IPk} implies the existence of some  $v\in V(T)$  such that $U=\Fix_G(B_T(v, k'))$. Furthermore, since $l\geq L_{q,k}$, we have $k'\geq k$ and Lemma \ref{IPk then IPk' for k' geq k} implies that $G$ satisfies the property ${\rm IP}_{k'}$. Therefore, Proposition \ref{G(W) for Sk-1} ensures the existence of an edge $e\subsetneq B_T(v,1)$ such that
		$$\Fix_G(B_T(e,k'-1))\subseteq VU$$
		and Lemma \ref{la forme des Sl pour IPk} ensures that $\Fix_G(B_T(e,k'-1)\in \mathcal{S}_q\lb l-1\rb$. Finally, since $e\in E(B_T(v,1))$, notice that $\Fix_G(B_T(v,k'))\subseteq \Fix_G(B_T(e,k'-1))$. 
	\end{proof}
	\begin{proof}[Proof of Theorem \ref{le theorem pour IP_k}]
		To prove that $\mathcal{S}_{q}$ factorizes$^+$ at depth $l\geq L_{q,k}$ we shall successively verify the three conditions of Definition \ref{definition olsh facto}.
		
		First, we need to prove that for every $U$ in the conjugacy class of an element of $\mathcal{S}_q\lb l \rb$ and every $V$ in the conjugacy class of an element of $\mathcal{S}_q$ with $V \not\subseteq U$, there exists a $W$ in the conjugacy class of an element of $\mathcal{S}_q\lb l -1\rb$ such that $U\subseteq W \subseteq V U.$ Let $U$ and $V$ be as above. If $V$ is conjugate to an element of $\mathcal{S}_q\lb l'\rb$ for some $l'\geq l$ the result follows directly from Lemma \ref{the lemma imply fertility for G S on IPk groups}. Therefore, we suppose that $l'\lneq l$. By the definition of $\mathcal{S}_q$ and since $\mathfrak{T}_q$ is stable under the action of $G$, there exist $\mathcal{T},\mathcal{T}'\in \mathfrak{T}_q$ such that $U=\Fix_G(\mathcal{T})$ and $V=\Fix_G(\mathcal{T}')$. We have two cases. Either, $\mathcal{T}'\subseteq \mathcal{T}$ and there exists a subtree $\mathcal{R}\in \mathfrak{T}_{q}$ such that $\mathcal{T}'\subseteq \mathcal{R}\subseteq \mathcal{T}$ and $\Fix_G(\mathcal{R})\in \mathcal{S}_q\lb l-1\rb$. In that case $$\Fix_G(\mathcal{T})\subseteq \Fix_G(\mathcal{R})\subseteq \Fix_G(\mathcal{T}')\subseteq \Fix_G(\mathcal{T}')\Fix_G(\mathcal{T}).$$ Or else, $\mathcal{T}'\not \subseteq \mathcal{T}$ and since $l'\lneq l$, this implies the existence of a subtree $ \mathcal{P}\in \mathfrak{T}_q$ such that $\mathcal{T}'\subseteq  \mathcal{P}\not= \mathcal{T}$ and $\Fix_G( \mathcal{P})\in \mathcal{S}_q\lb l\rb$. In particular, Lemma \ref{the lemma imply fertility for G S on IPk groups} ensures the existence of a $W\in \mathcal{S}_q\lb l-1\rb$ such that $U\subseteq W \subseteq \Fix_G( \mathcal{P})U$. Since $\Fix_G( \mathcal{P}) \subseteq \Fix_G(\mathcal{T}')$, this proves the first condition. 
		
		Next, we need to prove that $N_{G}(U, V)= \{g\in {G} \lvert g^{-1}Vg\subseteq U\}$ is compact for every $V$ in the conjugacy class of an element of $\mathcal{S}_q$. Just as before, notice that $V=\Fix_G(\mathcal{T}')$ for some $ \mathcal{T}'\in \mathfrak{T}_q$. Since $G$ satisfies the hypothesis \ref{Hypothese Hq} notice that
			\begin{equation*}
			\begin{split}
			N_{G}(U, V)&= \{g\in G \lvert g^{-1}Vg\subseteq U\}=\{g\in G \lvert g^{-1}\Fix_G(\mathcal{T}')g\subseteq \Fix_G(\mathcal{T})\}\\ 
			&=\{g\in G \lvert \Fix_G(g^{-1}\mathcal{T}')\subseteq \Fix_G(\mathcal{T})\}= \{g\in G \lvert g\mathcal{T}\subseteq \mathcal{T}'\}.
			\end{split}
			\end{equation*}
			In particular, since both $\mathcal{T}$ and $\mathcal{T}'$ are finite subtrees of $T$, $N_G(U,V)$ is a compact subset of $G$ which proves the second condition.
			
			Finally, we need to prove that for every $W$ in the conjugacy class of an element of $\mathcal{S}_q\lb l-1\rb$ with $U\subseteq W$ we have
			\begin{equation*}
			W\subseteq N_G(U,U) =\{g\in G \mid g^{-1}Ug\subseteq U\}.
			\end{equation*} 
			For the same reasons as before, there exists $\mathcal{R}\in \mathfrak{T}_q$ such that $W= \Fix_G(\mathcal{R})$. On the other hand, since $U\subseteq W$ and since $G$ satisfies the hypothesis \ref{Hypothese Hq}, notice that $\mathcal{R}\subseteq \mathcal{T}$. Furthermore, notice that $\Fix_G(\mathcal{R})$ has depth $l-1$ and therefore that $\mathcal{R}$ contains every interior vertex of $\mathcal{T}$. Since $G$ is unimodular and satisfies the hypothesis \ref{Hypothese Hq} this implies that
			\begin{equation*}
			\begin{split}
			\Fix_G(\mathcal{R})&\subseteq \{h\in G \lvert h\mathcal{T}\subseteq \mathcal{T}\}= \{h\in G \lvert \Fix_G(\mathcal{T})\subseteq\Fix_G(h\mathcal{T})\}\\
			&=\{h\in G \lvert h^{-1}\Fix_G(\mathcal{T})h\subseteq\Fix_G(\mathcal{T})\}= N_G(U,U)
			\end{split}
			\end{equation*}
			which proves the third condition.
	\end{proof}
	In particular Theorem \ref{la version paki du theorem de classification} provides a bijective correspondence between the equivalence classes of irreducible representations of $G$ at depth $l\geq L_{q,k}$ with seed $C\in \mathcal{F}_{\mathcal{S}_q}$ and the $\mathcal{S}_q$-standard representations of $\Aut_G(C)$. We now give examples of groups which satisfy the hypotheses of Theorem \ref{le theorem pour IP_k}. 
	\begin{example}\label{example autT pour changer}
		The full group $\Aut(T)$ of automorphisms of $T$. By Lemma \ref{la diffenrence entre les fixing group alors la difference entre les graph}, $\Aut(T)$ satisfies the hypothesis \ref{Hypothese H Tree} and therefore the hypothesis $H_0$. Furthermore, $\Aut(T)$ coincides with its $1$-closure and therefore satisfies the property \ref{IPk} for every integer  $k\geq 1$ (Lemmas \ref{IPk then IPk' for k' geq k} and \ref{lemme Burger mozes 2}). Since $\Aut(T)$ is non-discrete and unimodular Theorem \ref{le theorem pour IP_k} applies and the generic filtrations $\mathcal{S}_0$ factorizes$^+$ at all depth $l\geq 1$. Notice that $\mathcal{S}_{0}$ is quite different from the generic filtration $\mathcal{S}$ of $\Aut(T)$ considered in Chapter \ref{Application to Aut T}. On the other hand, the cuspidal representations of $\Aut(T)$ are still exactly the representations at depth $l\geq 1$ for $\mathcal{S}_0$. In particular, this new generic filtration also leads to a complete description of the cuspidal representations of $\Aut(T)$.
	\end{example} 
	\begin{example}[Proof of corollary \ref{corollary this apply to radu group}]\label{example Radu} Let $T$ be a $(d_0,d_1)$-semi-regular tree with $d_0,d_1\geq 6$. Then, any closed subgroup $G\leq \Aut(T)$ that acts transitively on the boundary and whose local action at every vertex contains the alternating group of the corresponding degree satisfies the hypotheses of Theorem \ref{le theorem pour IP_k}. Those groups were extensively studied by Radu in \cite{Radu2017}. Among other things, he showed that every such group $G$ is non-discrete  \cite[Theorem G]{Radu2017} and  $k$-closed for some $k$ depending on the group (Definition \ref{definition kclosure}) \cite[Theorem H]{Radu2017}. On the other hand, such a group is unimodular and since the local action at each vertex contains the alternating group a similar proof than the one of Lemma \ref{la diffenrence entre les fixing group alors la difference entre les graph} shows that $G$ satisfies the hypothesis \ref{Hypothese H Tree} (hence \ref{Hypothese Hq} $\forall q\in \N$).  In particular, Theorem \ref{le theorem pour IP_k} applies and the generic filtrations $\mathcal{S}_{0}$ factorizes$^+$ at all depth $l\geq L_{0,k}=2k-1$. In a forthcoming paper \cite{SemalR2022}, we show without use of the property \ref{IPk} that the generic filtration $\mathcal{S}_0$ factorizes$^+$ at all positive depth if the group is in addition simple. The proof of that last statement is quite technical, relies heavily on Radu's classification and is unrelated to the property \ref{IPk}. Furthermore, since various important consequences such as Nebbia's CCR conjecture on trees need to be discussed in light of the result, the author decided to not present a proof in this notes. 
	\end{example}
	\begin{example}\label{example fixateur point au bords}
		Let $G\leq \Aut(T)$ be a $k$-closed subgroup (Definition \ref{definition kclosure}), let $\omega\in \partial T$ be an end $T$ and consider the stabiliser of the $\omega$-horocycles $$G_\omega^0 =\{g\in G\lvert g\omega = \omega \mbox{ and }\exists v\in V(T)\mbox{ s.t. }gv=v\}.$$ Notice that $G_\omega^0$ is still $k$-closed and hence satisfies the property \ref{IPk}. Furthermore, $G_\omega^0$ is a union of compact groups and is therefore unimodular. In particular, if $G_\omega^0$ is non-discrete and satisfies the hypothesis \ref{Hypothese Hq}, it satisfies the hypotheses of Theorem \ref{le theorem pour IP_k}. However, notice that $G_\omega^0$ never satisfies the hypothesis $H_0$ since for any edges $e,f\in E(T)$ along an infinite geodesic with end $\omega$ we have either that $\Fix_{G^0_\omega}(e)\subseteq \Fix_{G^0_\omega}(f)$ or that $\Fix_{G^0_\omega}(f)\subseteq \Fix_{G^0_\omega}(e)$. Nevertheless, in certain cases, a description of the remaining cuspidal representations of $G$ can be obtained using Theorem \ref{Theorem IPK avec un arbre}. For instance let $G=\Aut(T)$. In that case, $G_\omega^0$ satisfies the hypothesis $H_1$ and the generic filtration $\mathcal{S}_1$ factorizes$^+$ at all depths $l\geq 1$. In particular, by Theorem \ref{le theorem pour IP_k} we obtain a description of the cuspidal representations admitting non-zero invariant vectors for the fixator of a ball of radius one around an edge or bigger but not for the fixator of a ball of radius one around a vertex. To obtain a description of the cuspidal representations admitting non-zero invariant vectors for the fixator of a ball $B_T(v,1)$ of radius $1$ around a vertex $v\in V(T)$, we let $\mathcal{P}=B_T(v,1)$ we notice that $\Sigma_{\mathcal{P}}=\{e\}$ where $e$ is the only edge of $B_T(v,1)$ contained in the geodesic $\lb v,\omega\rb$  and we apply Theorem \ref{Theorem IPK avec un arbre}. Notice that the reaming irreducible representations of $G^0_\omega$ are spherical (for $G=\Aut(T)$ they are classified in \cite{Nebbia1990}).
	\end{example}
	Other applications of Theorem \ref{Theorem IPK avec un arbre} and Theorem \ref{le theorem pour IP_k} could be made for instance on the $k$-closure of certain groups of automorphisms of trees and on the generalisation of Burger-Mozes groups described in \cite{Tornier2020}. 
	\subsection{Existence of $\mathcal{S}_q$-standard representations}\label{existence of representaions ot depth l for IPk}
	Let $T$ be a $(d_0,d_1)$-semi-regular tree with $d_0,d_1\geq 3$ and let $q\in \N$ be a non-negative integer. If $q$ is even, let $$\mathfrak{T}_{q}=\bigg\{B_T(v,r)\Big\lvert v\in V(T), r\geq \frac{q}{2}+1\bigg\}\sqcup\bigg\{B_T(e,r)\Big\lvert e\in E(T), r\geq \frac{q}{2}\bigg\}.$$
	If $q$ is odd, let  $$\mathfrak{T}_{q}=\bigg\{B_T(v,r)\Big\lvert v\in V(T), r\geq \frac{q+1}{2}\bigg\}\sqcup\bigg\{B_T(e,r)\Big\lvert e\in E(T), r\geq \frac{q+1}{2}\bigg\}.$$ 
	For any closed non-discrete unimodular subgroup $G\leq\Aut(T)$ satisfying the hypothesis \ref{Hypothese Hq}, we have shown that  $$\mathcal{S}_q=\{\Fix_G(\mathcal{T})\lvert \mathcal{T}\in \mathfrak{T}_q\}$$ is a generic filtration of $G$. Furthermore, if $G$ satisfies the property \ref{IPk} for some integer $k\geq 1$ we have shown that $\mathcal{S}_q$ factorizes$^+$ at all even depth $l\geq L_{q,k}$ where$$L_{q,k}=\begin{cases}
	\max\{1,2k-q-1\} \qq& \mbox{if } q\mbox{  is even.} \\
	\max\{1,2k-q\} & \mbox{if } q\mbox{  is odd.} 
	\end{cases}$$ 
	In particular, Theorem \ref{la version paki du theorem de classification} provides a bijective correspondence between the equivalence classes of irreducible representations of $G$ at depth $l\geq L_{q,k}$ with seed $C\in \mathcal{F}_{\mathcal{S}_q}$ and the $\mathcal{S}_q$-standard representations of $\Aut_G(C)$. The purpose of the present chapter is to study the existence of those ${\mathcal{S}_q}$-standard representations. The following result ensures the existence of ${\mathcal{S}_q}$-standard representations of $\Aut_{G}(C)$ for all $C\in \mathcal{F}_{\mathcal{S}_q}$ with height $l\geq L_{q,k}$ if $q$ and $l$ have the same parity.
	\begin{proposition}\label{existence of standard for edges for IPK}
		Let $G\leq \Aut(T)$ be a closed non-discrete unimodular subgroup satisfying the hypothesis \ref{Hypothese Hq} and the property \ref{IPk} for some integers $q\geq 0$, $k\geq 1$ and let $l\geq L_{q,k}$. Suppose that one of the following happens:
		\begin{itemize}
			\item $q$ and $l$ are even.
			\item $q$ and $l$ are odd but $l\not=1$.
			\item $q$ is odd, $l=1$ and $\Fix_G(v^{(\frac{q}{2}+1)})\not= \{g\in G\lvert ge=e\}$  $\forall e\in E(T)$, $\forall v\in e$.
		\end{itemize} Then, there exists a ${\mathcal{S}_q}$-standard representation of $\Aut_{G}(C)$ for every $C\in \mathcal{F}_{\mathcal{S}_q}$ at height $l$.
	\end{proposition} 
	\begin{proof}
		Let $C\in \mathcal{F}_{\mathcal{S}_q}$ be at height $l$. Since $q$ and $l$ have the same parity, Lemma \ref{la forme des Sl pour IPk} ensures the existence of an edge $e\in E(T)$ and an integer $r\geq k$ and such that $B_T(e,r)\in \mathfrak{T}_q$ and $C=\mathcal{C}(\Fix_{G}(B_T(e,r))$. For shortening of the formulation we let $\mathcal{T}$ denote the subtree $B_T(e,r)$. Since $G$ satisfies the hypothesis \ref{Hypothese Hq} and as a consequence of Lemma \ref{la forme des Sl pour IPk}, notice that $N_G(\Fix_G(\mathcal{T}))=\{g\in G \mid g\mathcal{T}\subseteq \mathcal{T}\}=\Stab_{G}(\mathcal{T})=\{g\in G \lvert ge=e\}$, that $\Aut_G(C)\simeq \Stab_G(\mathcal{T})/\Fix_G(\mathcal{T})$ and that 
		\begin{equation*}
		\begin{split}
		\tilde{\mathfrak{H}}_{\mathcal{S}_q}(\Fix_G(\mathcal{T}))&= \{W \mid \exists g\in G \mbox{ s.t. } gWg^{-1}\in \mathcal{S}_q\lb l-1 \rb \mbox{ and } \Fix_G(\mathcal{T}) \subseteq  W \}\\ &=\{\Fix_G(B_T(v,r))\lvert v\in e\}.
		\end{split}
		\end{equation*}
		Let $v_0,v_1$ denote the two vertices of $e$, let $\mathcal{T}_i=B_T(v_i,r)$ and notice that $\mathcal{T}_0\cup \mathcal{T}_1=\mathcal{T}$. In particular, the action of $N_G(\Fix_G(\mathcal{T}))$ on $\mathcal{T}$ permutes the subtrees $\{\mathcal{T}_0, \mathcal{T}_1\}$. On the other hand, since $G$ satisfies the hypothesis \ref{Hypothese Hq} our hypotheses imply that $$\Fix_G(\mathcal{T})\subsetneq \Fix_G(\mathcal{T}_i)\subsetneq \Stab_G(\mathcal{T}).$$  The result therefore follows from Proposition \ref{existence criterion}.
	\end{proof}
	The following two results ensures the existence of ${\mathcal{S}_q}$-standard representations of $\Aut_{G}(C)$ for all $C\in \mathcal{F}_{\mathcal{S}_q}$ with height $l\geq L_{q,k}$ if $q$ and $l$ have opposite parity. We start with the degenerate case $q=0,k=1$ and $l=1$ where Proposition \ref{existence criterion} does not apply.
	\begin{lemma}\label{existence of standard for vertices}
		Let $G\leq \Aut(T)$ be a closed non-discrete unimodular subgroup satisfying the hypothesis $H_0$, the Tits independence property ${\rm IP}_1$ and such that $\Fix_G(v)$ is $2$-transitive on the set of edges of $B_T(v,1)$ for every $v\in V(T)$. Then, there exists an $\mathcal{S}_{0}$-standard representation of $\Aut_{G}(C)$ for every $C\in \mathcal{F}_{\mathcal{S}_0}$ at height $1$. 
	\end{lemma}
	\begin{proof}
		Let $C\in \mathcal{F}_{\mathcal{S}_{0}}$ be at height $1$. Lemma \ref{la forme des Sl pour IPk} ensures the existence of a vertex $v\in V(T)$ such that $C=\mathcal{C}(\Fix_{G}(B_T(v,1))$. Let $U=\Fix_{G}(B_T(v,1))$ and notice that
		$$N_G(U)= \{g\in G \mid gB_T(v,1)\subseteq B_T(v,1)\}= \{ g\in G \lvert gv=v\}=\Fix_G(v).$$
		Furthermore, since $G$ satisfies the hypothesis $H_0$, Lemma \ref{la forme des Sl pour IPk} implies that 
		\begin{equation*}
		\begin{split}
		\tilde{\mathfrak{H}}_{\mathcal{S}_{0}}(U)&= \{W \mid \exists g\in G \mbox{ s.t. } gWg^{-1}\in \mathcal{S}_q\lb l-1 \rb \mbox{ and } \Fix_G(B_T(v,1)) \subseteq  W \}\\ &= \{\Fix_G(e)\mid e\in E(B_T(v,1))\}
		\end{split}
		\end{equation*}
		where $E(B_T(v,1))$ denotes the set of edges of $B_T(v,1)$. Let $d$ be the degree of $v$ in $T$, let $X=E(B_T(v,1))$ and notice that our hypotheses imply that $\Aut_{G}(C)\simeq \Fix_G(v)/\Fix_G(B_T(v,1))$ is $2$-transitive $X$. In particular, Lemma \ref{les rep de Qsur Qi} implies the existence of an irreducible representation $\sigma$ of $\Aut_{G}(C)$ without non-zero $\Fix_{\Aut_{G}(C)}(e)$-invariant vectors for all $e\in X$. Since $\mathfrak{H}_{\mathcal{S}_{0}}(U)=\{p_U(\Fix_G(e))\lvert e\in X\}=\{\Fix_{\Aut_{G}(C)}(e)\lvert e\in X\}$, this proves the existence of an $\mathcal{S}_{0}$-standard representation of $\Aut_{G}(C)$.  
	\end{proof}
	\noindent The following result treats the remaining cases.
	\begin{lemma}\label{existence of standard for odd depth}
		Let $G\leq \Aut(T)$ be a closed non-discrete unimodular subgroup satisfying the hypothesis \ref{Hypothese Hq} and the property \ref{IPk} for some integers $q\geq 0$, $k\geq 1$ and let $l\geq L_{q,k}$. Suppose further that  $$\Fix_G((B_T(v,1)\backslash\{w\})^{(r)})\not= \Fix_G(B_T(v,r+1))\qq \forall v\in V(T),\forall w\in B_T(v,1)-\{v\}$$ for all $r\geq \frac{q}{2}$ if $q$ is even and for all $r\geq \frac{q-1}{2}$ if $q$ is odd and that one of the following happens:
		\begin{itemize}
			\item $q$ is odd and $l$ is even.
			\item $q$ is even, $l$ is odd and $l\not=1$.
			\item $q$ is even, $q\not=0$, $l=1$ and $$\Fix_G((B_T(v,1)\backslash\{w\})^{(\frac{q}{2})})\not= \Fix_G(v)\q \forall v\in V(T),\qq \forall w\in  B_T(v,1)-\{v\}.$$ 
		\end{itemize} 
	 Then, there exists an $\mathcal{S}_q$-standard representation of $\Aut_{G}(C)$ for every $C\in \mathcal{F}_{\mathcal{S}_q}$ at height $l$. 
	\end{lemma}
	\begin{proof}
		Suppose that $C\in \mathcal{F}_{\mathcal{S}_q}$ is at height $l$. Since $q$ and $l$ have opposite parity, Lemma \ref{la forme des Sl pour IPk} ensures the existence of a vertex $v\in V(T)$ and an integer $r\geq k-1$ such that $B_T(v,r+1)\in \mathfrak{T}_q$ and $C=\mathcal{C}(\Fix_{G}(B_T(v,r+1)))$. For shortening of the formulation we let $\mathcal{T}$ denote the subtree $B_T(v,r+1)$. Since $G$ satisfies the hypothesis \ref{Hypothese Hq} and as a consequence of Lemma \ref{la forme des Sl pour IPk}, notice that $N_G(\Fix_G(\mathcal{T}))=\{g\in G \mid g\mathcal{T}\subseteq\mathcal{T}\}= \Stab_G(\mathcal{T})=\{g\in G \lvert gv=v\}$, that $\Aut_G(C)\simeq \Stab_G(\mathcal{T})/\Fix_G(\mathcal{T})$ and that 
		\begin{equation*}
		\begin{split}
		\tilde{\mathfrak{H}}_{\mathcal{S}_q}(\Fix_G(\mathcal{T}))&= \{W\mid \exists g\in G \mbox{ s.t. }gWg^{-1}\in \mathcal{S}_q\lb l-1 \rb \mbox{ and }\Fix_G(\mathcal{T}) \leq W \}\\
		&= \{\Fix_G(B_T(e,r))\lvert e\in E(B_T(v,1))\}.
		\end{split}
		\end{equation*}
		Now, let $\{w_1,...,w_d\}$ be the leaves of $B_T(v,1)$, let $\mathcal{T}_i= (B_T(v,1)\backslash\{w_i\})^{(r-1)}$ $ i=1,...,d$ and notice that $\mathcal{T}_i\cup \mathcal{T}_j=\mathcal{T}$ $\forall i\not=j$. On the other hand, the action of $\Stab_G(\mathcal{T})$ on $T$ permutes the subtrees $\{\mathcal{T}_1,...,\mathcal{T}_d\}$ and since each $\mathcal{T}_i$ contains $v$ we have that $\Fix_G(\mathcal{T}_i)\subseteq \Stab_G(\mathcal{T})$  $\forall i=1,...,d$. Furthermore, the hypotheses on $G$ imply that $\Fix_G(\mathcal{T})\subsetneq \Fix_G(\mathcal{T}_i)\subsetneq \Stab_G(\mathcal{T})$. In particular, Proposition \ref{existence criterion} ensures the existence of an irreducible representation $\sigma$ of $\Aut_{G}(C)$ without non-zero $p_{\Fix_G(\mathcal{T})}(\Fix_G(\mathcal{T}_i))$-invariant vectors. Moreover, for every edge $e\in E(B_T(v,1))$ there exists some $i\in \{1,...,d\}$ such that $B_T(e,r)\subseteq \mathcal{T}_i$ which implies that $p_{\Fix_G(\mathcal{T})}(\Fix_G(\mathcal{T}_i))\subseteq p_{\Fix_G(\mathcal{T})}(\Fix_G(B_T(e,r))$. Hence, $\sigma$ is an $\mathcal{S}_q$-standard representation of $\Aut_{G}(C)$.
	\end{proof}
	\newpage
	
	\section{Groups of automorphisms of trees with the property ${\rm IP}_{V_1}$}\label{application IPV1}
	In this chapter, we apply our machinery to groups of automorphisms of locally finite trees satisfying the property \ref{IPV1}(Definition \ref{defintion IPV1}). We use the same notations and terminology as in Chapter \ref{Application to Aut T}. Let $T$ be a locally finite tree and let $\Aut(T)^+$ be the group of type-preserving automorphisms of $T$. 
	\begin{definition}\label{defintion IPV1}
		A group $G\leq \Aut(T)^+$ is said to satisfy the \tg{property \ref{IPV1}}, if there exists a bipartition $V(T)=V_0\sqcup V_1$ such that every edge of $T$ contains exactly one vertex in each $V_i$ and such that $\forall w\in V_1$ we have 
		\begin{equation}\tag{${\rm IP }_{V_1}$}\label{IPV1}
		\Fix_G(B_T(w,1))=\prod_{v\in B_T(w,1)-\{w\}}\Fix_G(T(w,v)) 
		\end{equation}
		where $T(w,v)=\{u\in V(T) \lvert d_T(w,u)\lneq d_T(v,u)\}$.
	\end{definition}
	\begin{example}
		Let $T$ be a locally finite semi-regular tree and let $G\leq \Aut(T)^+$ satisfy the property ${\rm IP}_1$. Then, $G$ satisfies the property ${\rm IP}_{V_1}$. 
	\end{example}
	\noindent Other examples are given in Chapter \ref{Application to universal groups of right-angled buildings} where we show that the universal groups of certain semi-regular right-angled buildings can be realised as closed subgroups of $ \Aut(T)^+$ satisfying the property ${\rm IP}_{V_1}$ but where $T$ is in general not semi-regular (Theorem \ref{theorem pour les right angled buildings}).
	 
	The purpose of the present chapter is to prove Theorem \ref{Theorem pour IPV1} which provides an explicit generic filtration that factorizes$^+$ at all positive depth for subgroups $G\leq \Aut(T)^+$ satisfying the property \ref{IPV1} and the hypothesis \ref{Hypothese HV1}(Definition \ref{definition H v1}). This requires some formalism that we now introduce. Let $V(T)=V_0\sqcup V_1$ be a bipartition of $T$ such that every edge of $T$ contains exactly one vertex in each $V_i$. For every subtree $\mathcal{T}\subseteq T$, we set
	$$Q_\mathcal{T}=\{v\in V_0\lvert B_T(v,2) \subseteq \mathcal{T}\}$$
	and we define $\mathfrak{T}_{V_1}$ as follows:
	\begin{enumerate}
		\item $\mathfrak{T}_{V_1}\lb 0\rb=\{ B_T(v,1)\lvert v\in V_1\}.$
		\item For every $l\in \N$ such that $l\geq 0$, we define iteratively
		\begin{equation*}\label{page Tv1}
		\begin{split}
		\mathfrak{T}_{V_1}\lb l+1\rb=\{ \mathcal{T}\subseteq T\mid \exists \mathcal{R}\in \mathfrak{T}_{V_1}\lb l\rb, \exists w\in (V&(\mathcal{R})- Q_{\mathcal{R}})\cap V_0 \\
		&\mbox{ s.t. }  \mathcal{T}=\mathcal{R}\cup B_T(w,2)\}.
		\end{split}
		\end{equation*}
		\item We set $\mathfrak{T}_{V_1} =\bigsqcup_{l\in \N}\mathfrak{T}_{V_1}\lb l\rb$.
	\end{enumerate}
	For every closed subgroup $G\leq\Aut(T)^+$ we set
	\begin{equation*}\label{page de SV1 yesssouille}
	\mathcal{S}_{V_1}=\{\Fix_G(\mathcal{T})\lvert \mathcal{T}\in \mathfrak{T}_{V_1}\}.
	\end{equation*}
	\begin{definition}\label{definition H v1}
		A group $G\leq\Aut(T)^+$ is said to satisfy the hypothesis \ref{Hypothese HV1} if for all $\mathcal{T},\mathcal{T'}\in \mathfrak{T}_{V_1}$ we have
		\begin{equation}\tag{$H_{V_1}$}\label{Hypothese HV1}
		\Fix_G(\mathcal{T}')\leq \Fix_G(\mathcal{T})\mbox{ if and only if }\mathcal{T}\subseteq \mathcal{T}'. 
		\end{equation}
	\end{definition}
	\noindent If $G\leq \Aut(T)^+$ is a closed non-discrete unimodular subgroup of $\Aut(T)^+$ satisfying the hypothesis \ref{Hypothese HV1}, Lemma \ref{la forme des Sl pour IPV1} below ensures that $\mathcal{S}_{V_1}$ is a generic filtration of $G$ and that
	$$\mathcal{S}_{V_1}\lb l\rb =\{\Fix_G(\mathcal{T})\lvert \mathcal{T}\in \mathfrak{T}_{V_1}\lb l \rb\}.$$
	The following proves Theorem \ref{theorem Ipv1 letter}.
	\begin{theorem}\label{Theorem pour IPV1}
		Let $T$ be a locally finite tree and let $G\leq \Aut(T)^+$ be a closed non-discrete unimodular subgroup that satisfies the hypothesis \ref{Hypothese HV1} and the property \ref{IPV1}. Then, the generic filtration $\mathcal{S}_{V_1}$ factorizes$^+$ at all depth $l\geq 1$. 
	\end{theorem}
	
	\subsection{Preliminaries}
	Let $T$ be a locally finite tree and let $V(T)=V_0\sqcup V_1$ be a bipartition of $T$ such that every edge of $T$ contains exactly one vertex in each $V_i$. The purpose of the present section is to describe further the elements of $\mathfrak{T}_{V_1}$. For every subtree $\mathcal{T}\subseteq T$, we associated a set $Q_\mathcal{T}=\{v\in V_0\lvert B_T(v,2) \subseteq \mathcal{T}\}$. The purpose of the following two lemmas is to give a characterization of the elements $\mathcal{T}\in \mathfrak{T}_{V_1}$ in terms of their corresponding sets $Q_\mathcal{T}$. 
	\begin{lemma}\label{the form of trees in the family}
		The elements of $\mathfrak{T}_{V_1}$ satisfy the following:
		\begin{enumerate}
			\item[\rm (1)] Every $\mathcal{T}\in\mathfrak{T}_{V_1}$ is a complete finite subtree of $T$ with leaves in $V_0$.
			\item[\rm(2)] For every $\mathcal{T}\in \mathfrak{T}_{V_1}- \mathfrak{T}_{V_1}\lb 0 \rb$ we have that $\mathcal{T}= \bigcup_{v\in Q_\mathcal{T}}B_T(v,2)$. 
			\item[\rm(3)] For every $\mathcal{T}\in \mathfrak{T}_{V_1}$, we have $\mathcal{T}\in \mathfrak{T}_{V_1}\lb l\rb$ if and only if $\modu{Q_\mathcal{T}}=l$.
		\end{enumerate}
	\end{lemma}
	\begin{proof}
		Since each element of $\mathfrak{T}_{V_1}$ belongs to some $\mathfrak{T}_{V_1}\lb l\rb$ for some $l\in \N$, in order to show (3) it is enough to show that  $\forall \mathcal{T}\in \mathfrak{T}_{V_1}\lb l\rb$, $\modu{Q_\mathcal{T}}=l$. We prove (1), (2) and that $\modu{Q_\mathcal{T}}=l$ for every $\mathcal{T}\in \mathfrak{T}_{V_1}\lb l\rb$ by induction on $l$. If $l=0$, $\mathcal{T}=B_T(v,1)$ for some $v\in V_1$. Hence, $\mathcal{T}$ is a complete finite subtree with leaves in $V_0$ and $\modu{Q_\mathcal{T}}=0$. Similarly, if $l=1$, $\mathcal{T}=B_T(v,2)$ for some $v\in V_0$. In particular, $\mathcal{T}$ is a complete finite subtree of $T$ with leaves in $V_0$ and since $Q_\mathcal{T}=\{v\}$ we have that $\mathcal{T}= \bigcup_{v\in Q_\mathcal{T}}B_T(v,2)$ and $\modu{Q_\mathcal{T}}=1$. If $l\geq 2$, by construction, there exist $\mathcal{R}\in \mathfrak{T}_{V_1}\lb l-1\rb$ and $w\in (V(\mathcal{R})-Q_\mathcal{R})\cap V_0$ such that $\mathcal{T}= \mathcal{R}\cup B_T(w,2)$. By the induction hypothesis we have:
		\begin{enumerate}
			\item[\rm (1')] $\mathcal{R}$ is a finite complete subtree of $T$ with leaves in $V_0$.
			\item[\rm(2')] $\mathcal{R}= \bigcup_{v\in Q_{\mathcal{R}}} B_T(v,2).$ 
			\item[\rm(3')] $\modu{Q_{\mathcal{R}}}=l-1.$
		\end{enumerate}
		Since $\mathcal{T}= \mathcal{R}\cup B_T(w,2)$, (1') implies that $\mathcal{T}$ is a complete finite subtree with leaves in $V_0$ which proves (1). On the other hand, (2') implies that $Q_{\mathcal{R}} \cup \{w\}\subseteq Q_\mathcal{T}$ and therefore that $\mathcal{T}\subseteq \bigcup_{v\in Q_\mathcal{T}}B_T(v,2)$. The reverse inclusion follows trivially from the definition of $Q_\mathcal{T}$ which proves (2). Now, let $w'\in Q_{\mathcal{T}} - Q_{\mathcal{R}}$. To prove that $\modu{Q_\mathcal{T}}=l$ we have to prove that $w'=w$. Since $w'\in Q_{\mathcal{T}} - Q_{\mathcal{R}}$, there exists $u\in B_T(w',2)\cap V_0$ such that $u\not \in V(\mathcal{R})$. Moreover, since the leaves of $\mathcal{T}$ belongs to $V_0$ and since the distance between two vertices of $V_0$ is even, notice that $d_T(u,\mathcal{R})=2$. On the other hand, there exists a unique vertex $x\in V(\mathcal{R})-Q_\mathcal{R}$ such that $u\in B_T(x,2)$. Since $\mathcal{T}= \mathcal{R}\cup B_T(w,2)$ this proves that $w=x=w'$, that $Q_\mathcal{T}=Q_{\mathcal{R}}\sqcup \{w\}$ and therefore that $\modu{Q_\mathcal{T}}=l$. 
	\end{proof}
	\begin{lemma}\label{alternative form of S}
		Let $\mathcal{T}= \bigcup_{v\in Q}B_T(v,2)$ for some finite set $Q\subseteq V_0$ of order $l\geq 1$. Then $\mathcal{T}\in \mathfrak{T}_{V_1}\lb l\rb$ if and only if $\mbox{Con}(Q)\cap V_0=Q$ where $\mbox{Con}(Q)$ denotes the convex hull of $Q$ in $T$.
	\end{lemma}
	\begin{proof}
		Suppose first that $\mathcal{T}\in \mathfrak{T}_{V_1}\lb l\rb$. The definition of $Q_{\mathcal{T}}$ implies that $Q\subseteq Q_\mathcal{T}$ and Lemma \ref{the form of trees in the family} ensures that $\modu{Q_\mathcal{T}}=l$ which proves that $Q=Q_\mathcal{T}$. We prove that $\mbox{Con}(Q_\mathcal{T})\cap V_0=Q_{\mathcal{T}}$ by induction on $l\geq 1$. If $l=1$ the result is trivial. Suppose that $l\geq 2$. By construction, there exists $\mathcal{R}\in \mathfrak{T}_{V_1}\lb l-1\rb$ and $w\in (V(\mathcal{R})\cap V_0)-Q_{\mathcal{R}}$ such that $\mathcal{T}= \mathcal{R}\cup B_T(w,2)$. Furthermore, since $l-1\geq 1$, Lemma \ref{the form of trees in the family} ensures that  $\mathcal{R}= \bigcup_{v\in Q_{\mathcal{R}}}B_T(v,2)$. Since $T$ is a tree and since $w\in \mathcal{R}$, there exists a unique $u\in Q_{\mathcal{R}}$ such that $d_T(u,w)=2$ and we have that $\mbox{Con}(Q_\mathcal{T})= \mbox{Con}(Q_{\mathcal{R}})\cup \lb u,w\rb$. Finally, notice that $\lb u,w\rb\cap V_0=\{u,w\}$ which proves that $\mbox{Con}(Q_\mathcal{T})\cap V_0= Q_{\mathcal{R}}\cup \{w\}= Q_\mathcal{T}$. 
		
		Now, we show by induction on $l$ that $\mathcal{T}= \bigcup_{v\in Q}B_T(v,2)\in \mathfrak{T}_{V_1}\lb l \rb$ if $Q\subseteq V_0$ is a set of order $l$ such that $\mbox{Con}(Q)\cap V_0=Q$. If $l=1$, the result is trivial. Suppose that $l\geq 2$, choose any vertex $v\in Q$ and let $Q^n=\{w\in Q\lvert d_T(w,v)=2n\}$. Since $Q$ is finite there exists $N\in \N$ such that $Q^N\not = \es$ but $Q^n=\es$ for every $n\gneq N$. In particular, notice that $Q=\bigsqcup_{n\leq N} Q^n$. For every $n\leq N$, we let $S_n=\bigcup_{w\in Q^n, s\leq n}B_T(w,2)$, $l_n=\big\lvert\bigsqcup_{s\leq n}Q^{s-1}\big\rvert$ and we notice by induction on $n$ that $S_n\in \mathfrak{T}_{V_1}\lb l_n\rb$. Notice that the result is trivial for $n=1$, so let $n\geq 2$ and let $Q^n=\{v_1,...,v_{r_n}\}$. Since $\mbox{Con}(Q)\cap V_0=Q$, for all $w\in Q^n$ there exists $v_w\in Q^{n-1}$ such that $d_T(w,v_w)=2$. In particular, starting from our induction hypothesis we obtain iteratively for every $0\leq t\leq r_n$ that 
		$S_n\cup \big( \bigcup_{i\leq t}B_T(v_i,2)\big)\in \mathfrak{T}_{V_1}\lb l_n+ t\rb$. The result follows since $l_N=l$ and $S_N=\mathcal{T}$.
	\end{proof}
	This description of the elements of $\mathfrak{T}_{V_1}$ makes the following easier to prove.
	\begin{lemma}\label{invariance de TTT par rapport a l'action de AutT plus}
		Let $\mathcal{T}\in \mathfrak{T}_{V_1}\lb l \rb$ and $g\in \Aut(T)^+$, then we have $g\mathcal{T}\in \mathfrak{T}_{V_1}\lb l\rb$. 
	\end{lemma} 
	\begin{proof}
		If $l=0$, $\mathcal{T}=B_T(v,1)$ some $v\in V_1$. Furthermore, $gB_T(v,1)=B_T(gv,1)$ and since the element of $\Aut(T)^+$ are type-preserving, $gv\in V_1$ which proves that $g\mathcal{T}\in \mathfrak{T}_{V_1}\lb 0\rb$.  If $l\geq 1$, Lemma \ref{alternative form of S} ensures that $\mbox{Con}(Q_\mathcal{T})\cap V_0=Q_\mathcal{T}$. It is clear from the definition that $Q_{g\mathcal{T}}=gQ_{\mathcal{T}}$ and since $g$ is a type-preserving automorphism of a tree we have $\mbox{Con}(gQ_\mathcal{T})\cap V_0=gQ_\mathcal{T}$. In particular, Lemma \ref{alternative form of S} ensures that $g\mathcal{T}\in \mathfrak{T}_{V_1}\lb l \rb$.
	\end{proof}
	\subsection{Factorization of the generic filtration $\mathcal{S}_{V_1}$}
	The purpose of this section is to prove Theorem \ref{Theorem pour IPV1}. Just as before, let $T$ be a locally finite tree and let $V(T)=V_0\sqcup V_1$ be a bipartition of $T$ such that every edge of $T$ contains exactly one vertex in each $V_i$, let $\mathfrak{T}_{V_1}$ be the family of subtrees defined page \pageref{page Tv1}, let $G$ be a closed non-discrete unimodular subgroup of $\Aut(T)^+$ and let $\mathcal{S}_{V_1}=\{\Fix_G(\mathcal{T})\lvert \mathcal{T}\in \mathfrak{T}_{V_1}\}$. 
	\begin{lemma}\label{la forme des Sl pour IPV1}
		Let $G\leq \Aut(T)^+$ be a closed non-discrete unimodular subgroup satisfying the hypothesis \ref{Hypothese HV1}. Then, $\mathcal{S}_{V_1}$ is a generic filtration of $G$ and
		$$\mathcal{S}_{V_1}\lb l\rb =\{\Fix_G(\mathcal{T})\lvert \mathcal{T}\in \mathfrak{T}_{V_1}\lb l \rb\}\q \forall l\in \N.$$
	\end{lemma}
	\begin{proof}
		For every $\mathcal{T}\in \mathfrak{T}_{V_1}$ notice that $g\Fix_{G}(\mathcal{T})g^{-1}=\Fix_{G}(g\mathcal{T})$ $\forall g\in G$ and therefore that $\mathcal{C}(\Fix_{G}(\mathcal{T}))=\{\Fix_{G}(g\mathcal{T})\lvert g\in G\}$. In particular, for every $\mathcal{T},\mathcal{T}'\in \mathfrak{T}_{V_1}$ we have that $\mathcal{C}(\Fix_{G}(\mathcal{T}'))\leq \mathcal{C}(\Fix_{G}(\mathcal{T}))$ if and only if there exists $g\in G$ such that $\Fix_{G}(\mathcal{T})\leq \Fix_{G}(g\mathcal{T}')$. Since $G$ satisfies the hypothesis \ref{Hypothese HV1}, this implies that $\mathcal{C}(\Fix_{G}(\mathcal{T}'))\leq \mathcal{C}(\Fix_{G}(\mathcal{T}))$ if and only if there exists some $g\in G$ such that $g\mathcal{T}'\subseteq \mathcal{T}$. On the other hand, Lemma \ref{invariance de TTT par rapport a l'action de AutT plus} ensures that $\mathfrak{T}_{V_1}$ is stable under the action of $G$. In particular, for every strictly increasing chain $C_0\lneq C_1\lneq ...\lneq C_{n-1}\lneq C_{n}$ of elements of $\mathcal{F}_{\mathcal{S}_{V_1}}$ there exists a strictly increasing chain $\mathcal{T}_0\subsetneq\mathcal{T}_1\subsetneq...\subsetneq\mathcal{T}_{n-1}\subsetneq\mathcal{T}_n$ of elements of $\mathfrak{T}_{V_1}$ such that $C_t=\mathcal{C}(\Fix_G(\mathcal{T}_t))$. On the other hand, for every strictly increasing chain $\mathcal{T}_0\subsetneq\mathcal{T}_1\subsetneq...\subsetneq\mathcal{T}_{n}\subseteq\mathcal{T}$ of elements of $\mathfrak{T}_{V_1}$ contained in $\mathcal{T}$ we can build a strictly increasing chain $\mathcal{C}(\Fix_G(\mathcal{T}_0))\lneq\mathcal{C}(\Fix_G(\mathcal{T}_1))\lneq...\lneq\mathcal{C}(\Fix_G(\mathcal{T}_{n}))\lneq\mathcal{C}(\Fix_G(\mathcal{T}))$ of elements of $\mathcal{F}_{\mathcal{S}_{V_1}}$. This proves that the height of $\mathcal{C}(\Fix_G(\mathcal{T}))$ is the maximal length of a strictly increasing chain of elements of $\mathfrak{T}_{V_1}$ contained in $\mathcal{T}$.
		The result therefore follows from the following observation. Lemma \ref{the form of trees in the family} ensures that every maximal strictly increasing chain of elements of $\mathfrak{T}_{V_1}$ contained in $\mathcal{T}$ is of the form $\mathcal{T}_0\subsetneq \mathcal{T}_1\subsetneq...\subsetneq\mathcal{T}_{l-1}\subsetneq\mathcal{T}$ where $\mathcal{T}_t\in \mathfrak{T}_{V_1}\lb t \rb$.
	\end{proof}
	The following lemma shows that independence properties such as ${\rm IP}_1$ or ${\rm IP}_{V_1}$ can be realised as factorization properties for the fixators of larger subtrees than edges and ball of radius $1$ around vertices. 
	\begin{lemma}\label{independence on trees}
		Let $T$ be a locally finite tree, let $G\leq \Aut(T)$, let $\mathcal{A}$ be a family of finite subtrees of $T$ with at least two vertices such that
		\begin{equation*}
		\Fix_G( \mathcal{P})=\prod_{v\in \partial  \mathcal{P}} \Fix_G(T( \mathcal{P},v))\mbox{ } \q\forall \mathcal{P}\in \mathcal{A}
		\end{equation*}  
		where $\partial  \mathcal{P}$ denotes the set of leaves of $ \mathcal{P}$ and $T(\mathcal{P},v)$ denotes the half-tree $T( \mathcal{P},v)=\{w\in V(T)\lvert d_T(w, \mathcal{P})\lneq d_T(w,v)\}$ and let $\mathcal{T}$ be a non-empty complete finite subtree  of $T$ such that for every $v\in \partial \mathcal{T}$, there exists a subtree $\mathcal{T}_v\in \mathcal{A}$ with $v\in \mathcal{T}_v\subseteq  \mathcal{T}$. Then, we have that
		\begin{equation*}
		\Fix_G(\mathcal{T})=\prod_{v\in \partial \mathcal{T}} \Fix_G(T(\mathcal{T},v)).
		\end{equation*}
	\end{lemma}
	\begin{proof}
		 For every subset $V\subseteq V(T)$, we denote by $V^c$ the complement of $V$ in $V(T)$. First, notice that for every two distinct leaves $v,v'\in \partial \mathcal{T}$, the support of the elements of $\Fix_G(T(\mathcal{T},v))$ and $\Fix_G(T(\mathcal{T},v'))$ are disjoints. In particular, the elements of $\Fix_G(T(\mathcal{T},v))$ and of $\Fix_G(T(\mathcal{T},v'))$ commute with one another and $\prod_{v\in \partial \mathcal{T}} \Fix_G(T(\mathcal{T},v))$ is a well defined subgroup of $G$. On the other hand, $\forall v\in \partial \mathcal{T}$ we have $\Fix_G(\mathcal{T})\supseteq  \Fix_G(T(\mathcal{T},v))$ and therefore that $ \Fix_G(\mathcal{T})\supseteq  \prod_{v\in \partial \mathcal{T}}\Fix_G(T(\mathcal{T},v))$. 
		In order to prove the other inclusion let $g\in \Fix_G(\mathcal{T})$ and let $\partial\mathcal{T}=\{v_1,...,v_n\}$. 
		The hypotheses on $G$ imply the existence of a subtree $\mathcal{T}_{1}\in \mathcal{A}$ such that $v_1\in \mathcal{T}_{1}\subseteq\mathcal{T}$. Furthermore, since $v_1\in \partial \mathcal{T}$ and $\mathcal{T}_{1}\subseteq \mathcal{T}$, we observe that $v_1\in \partial \mathcal{T}_{1}$ and $T(\mathcal{T}_1,v_1)= T(\mathcal{T},v_1)$. Furthermore, since $\mathcal{T}_{1}\subseteq \mathcal{T}$ and $g\in \Fix_G(\mathcal{T})$, notice that $g\in \Fix_G(\mathcal{T}_1)$. Our hypotheses on $G$ ensures the existence of some $h_{1}\in \Fix_G(T(\mathcal{T}_1,v_1))$ and $$g_1\in \prod_{v\in \partial \mathcal{T}_1-\{v_1\}}\Fix_G(T(\mathcal{T}_1,v))\subseteq \Fix_G(T(\mathcal{T}_1,v_1)^c)$$
		such that $g=h_{1}g_1$. Since $g\in \Fix_G(\mathcal{T})$ and $\Fix_G(T(\mathcal{T}_1,v_1))\subseteq \Fix_G(\mathcal{T})$, this decomposition implies that $g_1\in \Fix_G(\mathcal{T})\cap \Fix_G(T(\mathcal{T},v_1)^c)$. Proceeding iteratively, we prove the existence of some $h_i\in \Fix_G(T(\mathcal{T},{v_i}))$ and some $g_i\in \Fix_G(\mathcal{T})\cap\big( \bigcap_{r\leq i}\Fix_G(T(\mathcal{T},v_r)^c)\big)$ such that $g_{i-1}=h_{i} g_i$. To see that $g_i \in \bigcap_{r\leq i}\Fix_G(T(\mathcal{T},v_r)^c)$, notice by induction that $$g_{i-1}\in \bigcap_{r\leq i-1}\Fix_G(T(\mathcal{T},v_r)^c),$$ that $h_{i} \in \Fix_G(T(\mathcal{T},{v_i}))$ and that $\Fix_G(T(\mathcal{T},v_j))\subseteq \Fix_G(T(\mathcal{T},v_i)^c)$  $\forall i\not=j$. This implies that $h_i\in \bigcap_{r\leq i-1}\Fix_G(T(\mathcal{T},v_r)^c)$ and therefore that $g_i \in \bigcap_{r\leq i}\Fix_G(T(\mathcal{T},v_r)^c)$. The result follows since we have by construction that $g= h_1h_2...h_ng_n$, that $h_i\in \Fix_G(T(\mathcal{T},v_i))$ and that $$g_n\in \Fix_G(\mathcal{T})\cap \bigg(\bigcap_{v\in \partial \mathcal{T}}\Fix_G(T(\mathcal{T},v)^c)\bigg)=\Fix_G(T)=\{1_{\Aut(T)}\}.$$
	\end{proof}
	Among other things, this result provides an alternative proof of the group equality \eqref{equation pour la prop d'indep de tits} or equivalently of Proposition \ref{alternative definition iof property IPk} if $k=1$. An other consequence is the following result which is key to the proof of Theorem \ref{Theorem pour IPV1}
	\begin{proposition}
		Let $G\leq \Aut(T)^+$ be a subgroup satisfying the property \ref{IPV1}. Then, for every $\mathcal{T}\in \mathfrak{T}_{V_1}$, we have
		\begin{equation*}
		\Fix_G(\mathcal{T})=\prod_{v\in \partial \mathcal{T}} \Fix_G(T(\mathcal{T},v)).
		\end{equation*}  
	\end{proposition}
	\begin{proof}
		If $\mathcal{T}\in \mathfrak{T}_{V_1}\lb 0\rb$, there exists $w\in V_1$ such that $\mathcal{T}=B_T(w,1)$. Notice that $\partial \mathcal{T}=\{v\in V(T)\lvert d_T(v,w)=1\}$ and that $T(\mathcal{T},v)=T(w,v)$  $\forall v\in \partial \mathcal{T}$. Since $G$ satisfies the property \ref{IPV1}, we obtain, as desired, that
		\begin{equation*}
		\begin{split}
		\Fix_G(\mathcal{T})&= \prod_{v\in \partial \mathcal{T}}\Fix_G(T(\mathcal{T},v)).
		\end{split}
		\end{equation*}
		Now, let $\mathcal{A}=\mathfrak{T}_{V_1}\lb 0\rb$ and notice from the point (2) of Lemma \ref{the form of trees in the family} that the hypotheses of Lemma \ref{independence on trees} are satisfied for every $\mathcal{T}\in\mathfrak{T}_{V_1}$. The result follows.
	\end{proof}
	\begin{lemma}\label{the key lemma for fertility of groups with IP V1 (build)}
		Let $G\leq \Aut(T)^+$ and suppose that 
		\begin{equation*}
		\Fix_G(\mathcal{T})=\prod_{v\in \partial \mathcal{T}} \Fix_G(T(\mathcal{T},v))\q \forall \mathcal{T}\in \mathfrak{T}_{V_1}.
		\end{equation*}
		Then, for every integers $ l,l'\geq 1$ with $l'\geq l$, $\forall \mathcal{T}\in \mathfrak{T}_{V_1} \lb l \rb$ and $\forall \mathcal{T}'\in \mathfrak{T}_{V_1}\lb l'\rb$ such that $\mathcal{T}\not \subseteq \mathcal{T}'$, there exists a subtree $\mathcal{R}\subseteq\mathcal{T}$ such that $\mathcal{R}\in \mathfrak{T}_{V_1}\lb l-1\rb$ and
		\begin{equation*}
		\Fix_G(\mathcal{R})\subseteq \Fix_G(\mathcal{T}')\Fix_G(\mathcal{T}).
		\end{equation*}	 
	\end{lemma}
	\begin{proof}
		Let $l,l'\geq 1$ be such that $l'\geq l$, let $\mathcal{T}\in \mathfrak{T}_{V_1} \lb l \rb$ and let $\mathcal{T}'\in \mathfrak{T}_{V_1}\lb l'\rb$ be such that $\mathcal{T}\not \subseteq  \mathcal{T}'$. If $l=1$, $Q_\mathcal{T}=\{v_\mathcal{T}\}$ for some $v_\mathcal{T}\in V_0$. Since $\mathcal{T}\not \subseteq  \mathcal{T}'$, we have that $v_\mathcal{T}\not\in Q_{\mathcal{T}'}$. Hence $d_T(v_\mathcal{T},Q_{\mathcal{T}'})\geq 2$. In particular, there exists a unique vertex $w\in V_1\cap B_T(v_\mathcal{T},1)$ such that $\mathcal{T}'\subseteq T(w,v_\mathcal{T})\cup \{v_\mathcal{T}\}$. Let $\mathcal{R}= B_T(w,1)$ and notice that $\mathcal{R}\in \mathfrak{T}_{V_1}\lb 0\rb$. Our hypotheses on $G$ ensure that
		$$\Fix_G(\mathcal{R})= \prod_{v\in \partial \mathcal{R}} \Fix_G(T(\mathcal{R},v)).$$ 
		On the other hand, $\partial \mathcal{R}\subseteq \partial \mathcal{T} \cup \{v_\mathcal{T}\}$ and $T(\mathcal{R},v)=T(w,v)$ $\forall v\in \partial \mathcal{R}$. This proves that $\Fix_G(T(\mathcal{R},v))\subseteq \Fix_G(\mathcal{T})$ for every leaf $v\in \partial \mathcal{R} - \{v_\mathcal{T}\}$ and since $ \mathcal{T}'\subseteq T(w,v_\mathcal{T})\cup\{v_\mathcal{T}\} =T(\mathcal{R},v_\mathcal{T})\cup\{v_\mathcal{T}\}$ we also have that $\Fix_G(T(\mathcal{R},v_\mathcal{T}))\subseteq \Fix_G(\mathcal{T}')$. This proves, as desired, that $$\Fix_G(\mathcal{R})=\prod_{v\in \partial \mathcal{R}} \Fix_G(T(\mathcal{R},v))\subseteq \Fix_G(\mathcal{T}')\Fix_G(\mathcal{T}).$$ 
		
		If $l\geq 2$, we have that $\modu{Q_\mathcal{T}}\geq 2$ and since $\mathcal{T}\not \subseteq \mathcal{T}'$, $Q_\mathcal{T}\not \subseteq Q_{\mathcal{T}'}$. In particular, there exists $v_\mathcal{T}\in Q_\mathcal{T}$ such that $d_T(v_\mathcal{T},Q_{\mathcal{T}'})=\max\{d_T(v,Q_{\mathcal{T}'})\lvert v\in Q_\mathcal{T}\}$. On the other hand, since $Q_\mathcal{T},Q_{\mathcal{T}'}\subseteq V_0$ the distance $d_T(v_\mathcal{T},Q_{\mathcal{T}'})$ must be even, hence $d_T(v_\mathcal{T},Q_{\mathcal{T}'})\geq 2$. Now, let $Q_{\mathcal{R}}=Q_\mathcal{T}- \{ v_\mathcal{T}\}$. Notice that $\mbox{Con}(Q_{\mathcal{R}})\cap V_0=Q_{\mathcal{R}}$. Indeed, suppose for a contradiction that there exists $w\in \mbox{Con}(Q_{\mathcal{R}})\cap V_0$ such that $w\not\in Q_{\mathcal{R}}$. Since $\mathcal{T}\in \mathfrak{T}_{V_1}\lb \modu{Q_\mathcal{T}}\rb$, Lemma \ref{alternative form of S} guarantees that $\mbox{Con}(Q_\mathcal{T})\cap V_0=Q_\mathcal{T}$. In particular, we observe that $w\in Q_\mathcal{T}-Q_{\mathcal{R}}$ and since $Q_{\mathcal{R}}= Q_\mathcal{T} -\{v_\mathcal{T}\}$ this implies that $w=v_\mathcal{T}$.  In particular, $v_\mathcal{T}\in\mbox{Con}(Q_{\mathcal{R}})\cap V_0$, there exists $w_1,w_2\in Q_\mathcal{T}- \{v_\mathcal{T}\}$ such that $(\lb w_1, w_2\rb\cap V_0)-\{w_1,w_2\} =\{v_\mathcal{T}\}$. If $l=2$, this leads to a clear contradiction since $Q_\mathcal{T}$ contains only two elements. On the other hand, if $l\geq 3$, we obtain a contradiction with our choice of $v_\mathcal{T}$. Indeed, for $i=1,2$ we have that $d_T(w_i, Q_{\mathcal{T}'})\leq d_T(v_\mathcal{T},Q_{\mathcal{T}'})$. Since $v_\mathcal{T}\not \in \mathcal{T}'$, this implies the existence of $\tilde{w}_i\in Q_{\mathcal{T}'}\cap T(w_i,v_\mathcal{T})$ and since there exists a unique simple path between $\tilde{w}_1$ and $\tilde{w}_2$, we obtain that $v_\mathcal{T}\in\lb \tilde{w}_1, \tilde{w}_2 \rb \cap V_0 \subseteq \mbox{Con}(Q_{\mathcal{T}'}) \cap V_0=Q_{\mathcal{T}'}\cap V_0$. This is a contradiction since $v_\mathcal{T}\not \in Q_{\mathcal{T}'}$ which proves proves that $\mbox{Con}(Q_{\mathcal{R}})\cap V_0=Q_{\mathcal{R}}$. In particular, Lemma \ref{alternative form of S} ensures that $\mathcal{R}=\bigcup_{w\in Q_{\mathcal{R}}}B_T(w,2)\in \mathfrak{T}_{V_1}\lb l-1\rb$. On the other hand, by choice of $v_\mathcal{T}$, we have $d_T(v_\mathcal{T}, w)\geq d_T(v,w)$ $\forall v\in Q_\mathcal{T}$, $\forall w\in Q_{\mathcal{T}'}$ which implies that  $\Fix_G(T(\mathcal{R},v_\mathcal{T}))\subseteq  \Fix_G(\mathcal{T}')$. On the other hand, $\partial \mathcal{R} \subseteq \partial \mathcal{T} \sqcup \{v_\mathcal{T}\}$ and $T(\mathcal{R},v)=T(\mathcal{T},v)$  $\forall v\in \partial \mathcal{R} \cap \partial \mathcal{T}$. In particular, since  $\Fix_G(\mathcal{T})=\prod_{v\in \partial \mathcal{T}} \Fix_G(T(\mathcal{T},v))$, we obtain that
		\begin{equation*}
		\begin{split}
		\Fix_G(\mathcal{R})&= \prod_{v\in \partial {\mathcal{R}}} \Fix_G(T(\mathcal{R},v))\subseteq \Fix_G(T(\mathcal{R},v_\mathcal{T}))\Fix_G(\mathcal{T})\subseteq \Fix_G(\mathcal{T}')\Fix_G(\mathcal{T}). 
		\end{split}
		\end{equation*}   
	\end{proof}
	\noindent The following result plays a similar role as Proposition \ref{G(W) for Sk-1} in Chapter \ref{application IPk}.
	\begin{proposition}\label{le lemme ultime de fertilité pour IPV1}
		Let $G\leq \Aut(T)^+$ be a closed subgroup satisfying the hypothesis \ref{Hypothese HV1} and the property \ref{IPV1}. Then, for ever integers $l,l'\geq 1$ such that $l'\geq l$, for every $U$ in the conjugacy class of an element of $\mathcal{S}_{V_1}\lb l \rb$ and every $V$ in the conjugacy class of an element of $\mathcal{S}_{V_1}\lb l'\rb$ such that $V\not \subseteq U$, there exists $W \in \mathcal{S}_{V_1}\lb l-1\rb$ such that $U\subseteq W\subseteq VU$.
	\end{proposition}
	\begin{proof}
		Lemma \ref{invariance de TTT par rapport a l'action de AutT plus} ensures that $\mathfrak{T}_{V_1}$ is stable under the action of $G$. In particular, by Lemma \ref{la forme des Sl pour IPV1} there exists $\mathcal{T}\in \mathfrak{T}_{V_1}\lb l\rb$ and $\mathcal{T}'\in \mathfrak{T}_{V_1}\lb l'\rb$ such that $U=\Fix_G(\mathcal{T})$ and $V=\Fix_G(\mathcal{T}')$. Furthermore, since $G$ satisfies the hypothesis \ref{Hypothese HV1} and since $V\not \subseteq U$ we have $\mathcal{T}\not \subseteq\mathcal{T}'$. In particular, since $G$ satisfies the property ${\rm IP}_{V_1}$, Lemma \ref{the key lemma for fertility of groups with IP V1 (build)} ensures the existence of $\mathcal{R}\in \mathcal{T}\lb l-1\rb$ such that $\mathcal{R}\subseteq \mathcal{T}$ and $\Fix_G(\mathcal{R})\subseteq \Fix_G(\mathcal{T}')\Fix_G(\mathcal{T})$.The result follows since, by  Lemma \ref{la forme des Sl pour IPV1},
		 $W=\Fix_G(\mathcal{R})\in \mathcal{S}_{V_1}\lb l-1\rb$ .
	\end{proof}
	We are finally ready to prove the main result of this Chapter.
	\begin{proof}[Proof of Theorem \ref{Theorem pour IPV1}]
		To prove that $\mathcal{S}_{V_1}$ factorizes$^+$ at all depth $l\geq1$ we shall successively verify the three conditions of Definition \ref{definition olsh facto}. 
		
		First, we need to prove that for every $U$ in the conjugacy class of an element of $\mathcal{S}_{V_1}\lb l \rb$ and every $V$ in the conjugacy class of an element of $\mathcal{S}_{V_1}$ such that $V \not\subseteq U$, there exists a $W$ in the conjugacy class of an element of $\mathcal{S}_{V_1}\lb l -1\rb$ such that $U\subseteq W \subseteq V U$. Let $U$, $V$ be as above. If $V$ is conjugate to an element of $\mathcal{S}_{V_1}\lb l'\rb$ for some $l'\geq l$ the results follows directly from Proposition \ref{le lemme ultime de fertilité pour IPV1}. Therefore, we suppose that $l'\lneq l$. Since $\mathfrak{T}_{V_1}$ is stable under the action of $G$ (Lemma \ref{invariance de TTT par rapport a l'action de AutT plus}), Lemma \ref{la forme des Sl pour IPV1} ensures the existence of $\mathcal{T}\in \mathfrak{T}_{V_1}\lb l \rb$ and $\mathcal{T}'\in \mathfrak{T}_{V_1}\lb l'\rb$ such that $U=\Fix_G(\mathcal{T})$ and $V=\Fix_G(\mathcal{T}')$. We have two cases.
		
		Either $\mathcal{T}'\subseteq \mathcal{T}$. In that case, we prove the existence of a finite subtree $\mathcal{R}\in \mathfrak{T}_{V_1}\lb l-1\rb$ such that $\mathcal{T}'\subseteq \mathcal{R}\subsetneq \mathcal{T}$ by induction on $l-l'$. If $l'-l=1$ we can take $\mathcal{R}=\mathcal{T}'$ and the result is trivial. On the other hand, if $l'-l\geq 2$, notice that $Q_\mathcal{T'}\subseteq Q_{\mathcal{T}'}$ and Lemma \ref{the form of trees in the family} ensures that $Q_{\mathcal{T}}- Q_\mathcal{T'}$ contains $l'-l$ vertices. If $Q_{\mathcal{T}'}$ is empty, there exists $v\in Q_\mathcal{T}$ such that $\mathcal{T}'\subseteq B_T(v,2)$. In particular, we let $\mathcal{P}=B_T(v,2)$ and notice that $\mathcal{P}\in\mathfrak{T}_{V_1}\lb l'+1\rb$ and that $\mathcal{T}'\subseteq \mathcal{P}\subsetneq \mathcal{T}$. If $Q_{\mathcal{T}'}$ is not empty, Lemma \ref{alternative form of S} ensures that $\mbox{Con}(Q_{\mathcal{T}})\cap V_0=Q_\mathcal{T}$ and at least one vertex $v\in Q_{\mathcal{T}}- Q_{\mathcal{T}'}$ satisfies that $d_T(v,Q_{\mathcal{T}'})=2$. We let $Q= Q_{\mathcal{T}'}\cup \{v\}$ and $\mathcal{P}=\bigcup_{w\in Q}B_T(w,2)$. Notice that $\mbox{Con}(Q)\cap V_0=Q$. In particular, Lemma \ref{alternative form of S} ensures that $\mathcal{P}\in \mathfrak{T}_{V_1}\lb l'+1\rb$ and we have, by construction, that $\mathcal{T}'\subseteq \mathcal{P}\subsetneq \mathcal{T}$. In both cases ($Q_\mathcal{T}$ is empty or not) our induction hypothesis ensures the existence of a finite subtree $\mathcal{R}\in \mathfrak{T}_{V_1}\lb l-1\rb$ such that $\mathcal{T}'\subseteq\mathcal{P}\subseteq \mathcal{R}\subsetneq \mathcal{T}$. In particular, we have $\Fix_G(\mathcal{R})\subseteq \Fix_G(\mathcal{T}')$ which implies, as desired, that
		$$\Fix_G(\mathcal{R})\subseteq \Fix_G(\mathcal{T}')\subseteq \Fix_G(\mathcal{T}')\Fix_G(\mathcal{T}).$$
		
		Or else, $\mathcal{T}'\not\subseteq \mathcal{T}$. If $Q_{\mathcal{T}'}=\es$, there exists a vertex $v\in V_0-Q_{\mathcal{T}}$ such that $\mathcal{T}'\subseteq B_T(v,2)$. In particular, we choose a set $Q\subseteq V_0$ of order $l$ containing $v$, such that $\mbox{Con}(Q)\cap V_0=Q$ and we set $\mathcal{P}=\bigcup_{w\in Q}B_T(w,2)$. Similarly, if $Q_{\mathcal{T}'}\not=\es$, since $\mbox{Con}(Q_{\mathcal{T}'})\cap V_0=Q_{\mathcal{T}'}$ by Lemma \ref{the form of trees in the family}, there exists a finite set $Q\subseteq V_0$ of order $l$ containing $Q_{\mathcal{T}'}$ and such that $\mbox{Con}(Q)\cap V_0=Q$. We set $\mathcal{P}=\bigcup_{w\in Q}B_T(w,2)$. In both cases ($Q_\mathcal{T}$ is empty or not), Lemma \ref{alternative form of S} ensures that $\mathcal{P}\in \mathfrak{T}_{V_1}\lb l \rb$ and we have by construction that  $\mathcal{T}'\subseteq \mathcal{P}$. In particular, Proposition \ref{le lemme ultime de fertilité pour IPV1} applied to $\Fix_G(\mathcal{T})$ and $\Fix_G(\mathcal{P})$ ensures the existence of a $W \in \mathcal{S}_{V_1}\lb l-1\rb$ such that $\Fix_G(\mathcal{T})\subseteq W\subseteq \Fix_G(\mathcal{P})\Fix_G(\mathcal{T})$. On the other hand, $\mathcal{T}'\subseteq \mathcal{P}$ which implies that $\Fix_G(\mathcal{P})\leq \Fix_G(\mathcal{T}')$. This proves the first condition.
			
		Next, we need to prove that $N_{G}(U, V)= \{g\in {G} \lvert g^{-1}Vg\subseteq U\}$ is compact for every $V$ in the conjugacy class of an element of $\mathcal{S}_{V_1}$. Just as before, notice that $V=\Fix_G(\mathcal{T}')$ for some $\mathcal{T}'\in \mathfrak{T}_{V_1}\lb l'\rb$. Since $G$ satisfies the hypothesis \ref{Hypothese HV1} notice that
		\begin{equation*}
		\begin{split}
			N_{G}(U, V)&= \{g\in G \lvert g^{-1}Vg\subseteq U\}= \{g\in G \lvert g^{-1}\Fix_{G}(\mathcal{T}')g\subseteq \Fix_{G}(\mathcal{T})\}\\
			&= \{g\in G \lvert \Fix_G(g^{-1}\mathcal{T}')\subseteq \Fix_G(\mathcal{T})\}= \{g\in G \lvert g\mathcal{T}\subseteq \mathcal{T}'\}.
		\end{split}
		\end{equation*}	
		Since both $\mathcal{T}$ and $\mathcal{T}'$ are finite subtrees of $T$, this implies that $N_G(U,V)$ is a compact subset of $G$ which proves the second condition.

		Finally, we need to prove that for every $W$ in the conjugacy class of an element of $\mathcal{S}_{V_1}\lb l-1\rb$ such that $U\subseteq W$ we have
		\begin{equation*}
			W\subseteq N_G(U,U) =\{g\in G \mid g^{-1}Ug\subseteq U\}.
		\end{equation*} 
		The same reasoning as before ensures the existence of some $\mathcal{R}\in \mathfrak{T}_{V_1}$ such that $W = \Fix_G(\mathcal{R})$. On the other hand, since $U\subseteq W$ and since $G$ satisfies the hypothesis \ref{Hypothese HV1}, notice that $\mathcal{R}\subseteq \mathcal{T}$. We have multiple cases. If $l=1$, there exists vertices $v\in V_0$ and $w\in V_1$ such that $\mathcal{T}=B_T(v,2)$ and $\mathcal{R}=B_T(w,1)$. In particular, since $\mathcal{R}\subseteq \mathcal{T}$, this implies that $v\in \mathcal{R}$ and therefore that
		\begin{equation*}
			\begin{split}
			\Fix_G(\mathcal{R})&\subseteq \Fix_G(v)=\{h\in G \lvert hB_T(v,2)\subseteq B_T(v,2)\}=\{h\in G\lvert h\mathcal{T}\subseteq \mathcal{T}\}.
			\end{split}
		\end{equation*} 
		Similarly, if $l\geq 2$, notice that $Q_{\mathcal{T}}$ and $Q_{\mathcal{R}}$ are non-empty sets. Furthermore, since $\mathcal{R}\subseteq \mathcal{T}$, we have that $Q_{\mathcal{R}}\subseteq Q_{\mathcal{T}}$ and there exists a unique $v\in Q_{\mathcal{T}}-Q_{\mathcal{R}}$. On the other hand, Lemma \ref{alternative form of S}, implies that $d_T(v,Q_\mathcal{R})=2$ and since $\mathcal{R}=\bigcup_{w\in Q_{\mathcal{R}}}B_T(w,2)$ we observe that $\Fix_G(\mathcal{R})\subseteq \Fix_G(Q_{\mathcal{T}})$. Since $\mathcal{T}=\bigcup_{w\in Q_{\mathcal{T}}}B_T(w,2)$, this implies that
		\begin{equation*}
			\begin{split}
			\Fix_G(\mathcal{R})\subseteq \{h\in G\lvert h\mathcal{T}\subseteq \mathcal{T}\}.
			\end{split}
		\end{equation*} 
		In both cases, since $G$ satisfies the hypothesis \ref{Hypothese HV1} we obtain that 
		\begin{equation*}
			\begin{split}
			\Fix_G(\mathcal{R})&\subseteq  \{h\in G \lvert h\mathcal{T}\subseteq \mathcal{T}\}= \{h\in G \lvert \Fix_G(\mathcal{T})\subseteq \Fix_G(h\mathcal{T})\}\\
			&=\{h\in G \lvert h^{-1}\Fix_G(\mathcal{T})h\subseteq \Fix_G(\mathcal{T})\}= N_G(U,U)
			\end{split}
		\end{equation*} which proves the third condition.
	\end{proof}
	In particular, if $G\leq\Aut(T)^+$ is a closed non-discrete unimodular subgroup satisfying the hypothesis \ref{Hypothese HV1} and the property \ref{IPV1} Theorem \ref{la version paki du theorem de classification} provides a bijective correspondence between the equivalence classes of irreducible representations of $G$ at depth $l\geq1$ with seed $C\in \mathcal{F}_{\mathcal{S}_{V_1}}$ and the $\mathcal{S}_{V_1}$-standard representations of $\Aut_G(C)$. As a concrete example, the group $\Aut(T)^+$ of type preserving automorphisms of a $(d_0,d_1)$-semi-regular tree $T$ with $d_0,d_1\geq 3$ satisfies the hypotheses of Theorem \ref{Theorem pour IPV1}. Other examples will be constructed in Chapter \ref{Application to universal groups of right-angled buildings}. 
	\subsection{Existence of $\mathcal{S}_{V_1}$-standard representations}\label{existence for IPV1}
	Let $T$ be a locally finite tree and let $V(T)=V_0\sqcup V_1$ be a bipartition of $T$ such that every edge of $T$ contains exactly one vertex in each $V_i$. Let $\mathfrak{T}_{V_1}$ be the family of subtrees defined page \pageref{page Tv1}, let $G$ be a closed non-discrete unimodular subgroup of $\Aut(T)^+$ and let $\mathcal{S}_{V_1}=\{\Fix_G(\mathcal{T})\lvert \mathcal{T}\in \mathfrak{T}_{V_1}\}$. If $G\leq \Aut(T)^+$ is a closed unimodular subgroup  satisfying the hypothesis \ref{Hypothese HV1} and the property \ref{IPV1} we have shown that $\mathcal{S}_{V_1}$ is a generic filtration of $G$ that factorizes$^+$ at all depth $l\geq 1$. In particular, Theorem \ref{la version paki du theorem de classification} ensures the existence of a bijective correspondence between the equivalence classes of irreducible representations of $G$ at depth $l\geq 1$ with seed $C\in \mathcal{F}_{\mathcal{S}_{V_1}}$ and the $\mathcal{S}_{V_1}$-standard representations of $\Aut_{G}(C)$. The following result treats the existence of $\mathcal{S}_{V_1}$-standard representations of $\Aut_{\mathcal{S}_{V_1}}(C)$ for all $C\in \mathcal{F}_{\mathcal{S}_{V_1}}$ at height $l\geq 1$ provided that $G$ satisfy some geometrical property.
	\begin{proposition}
		Let $G\leq \Aut(T)^+$ be a closed non-discrete unimodular subgroup satisfying the hypothesis \ref{Hypothese HV1} and the property \ref{IPV1}, let $l\geq 1$ and let $\mathcal{T}\in \mathfrak{T}_{V_1}\lb l\rb$ be such that for every $\mathcal{R}\in \mathfrak{T}_{V_1} \lb l-1\rb$ with $\mathcal{R}\subseteq \mathcal{T}$ we have $$\Fix_G(\mathcal{R})\subsetneq\Stab_G(\mathcal{T})=\{g\in G\lvert g\mathcal{T}\subseteq \mathcal{T}\}.$$
		Then, there exists an $\mathcal{S}_{V_1}$-standard representation of $\Aut_{\mathcal{S}_{V_1}}(\mathcal{C}(\Fix_G(\mathcal{T})))$.
	\end{proposition}
	\begin{proof}
	Let $C=\mathcal{C}(\Fix_G(\mathcal{T}))$. Lemma \ref{la forme des Sl pour IPV1} ensures that $C$ has height $l$ in $\mathcal{F}_{\mathcal{S}_{V_1}}$. Since $G$ satisfies the hypothesis \ref{Hypothese HV1} and as a consequence of Lemma \ref{la forme des Sl pour IPV1}, notice that $N_G(\Fix_G(\mathcal{T}),\Fix_G(\mathcal{T}))=\{g\in G \mid g\mathcal{T}\subseteq\mathcal{T}\}=\Stab_G(\mathcal{T})$, that $\Aut_{G}(C)\cong \Stab_G(\mathcal{T})/\Fix_G(\mathcal{T})$ and that
	\begin{equation*}
	\begin{split}
	\tilde{\mathfrak{H}}_{\mathcal{S}_{V_1}}(\Fix_G(\mathcal{T}))&= \{W\mid \exists g\in G \mbox{ s.t. }gWg^{-1}\in \mathcal{S}_{V_1}\lb l-1 \rb \mbox{ and }\Fix_G(\mathcal{T}) \subseteq W \}\\
	&= \{\Fix_G(\mathcal{R})\lvert \mathcal{R}\in \mathcal{S}_{V_1}\lb l-1\rb \mbox{ s.t. }\mathcal{R}\subseteq\mathcal{T}\}.
	\end{split}
	\end{equation*}
	Furthermore, the hypotheses on $G$ imply that $\Fix_G(\mathcal{T})\subsetneq\Fix_G(\mathcal{R})\subsetneq \Stab_G(\mathcal{T})$ for every $\mathcal{R}\in \mathfrak{T}_{V_1}\lb l-1\rb$ with $\mathcal{R}\subseteq \mathcal{T}$.
	
	We have two cases. If $\mathcal{T}\in \mathfrak{T}_{V_1}\lb 1\rb$, there exists $v\in V_0$ such that $\mathcal{T}=B_T(v,2)$ and every subtree $\mathcal{R}$ of $\mathcal{T}$ that belongs to $\mathcal{S}_{V_1}\lb 0\rb$ is of the form $B_T(w,1)$ for some $w\in \partial B_T(v,1)$. Let $\{\mathcal{T}_1,..., \mathcal{T}_d\}$ be the set of subtrees of $T$ of the form $\bigcup_{w\in \partial B_T(v,1)-\{u\}}B_T(w,1)$ for some $u\in \partial B_T(v,1)$. Notice that each element of $\Stab_G(\mathcal{T})=\Fix_G(v)$ permutes the vertices of $\partial B_T(v,1)$ and therefore permutes the elements of $\{\mathcal{T}_1,..., \mathcal{T}_d\}$. On the other hand, for every $i,j\in \{1,...,d\}$ with $i\not= j$, we have that $\mathcal{T}_i\cup  \mathcal{T}_j=\mathcal{T}$ and thanks to Proposition \ref{independence on trees}, $\Fix_G(\mathcal{T})\subsetneq \Fix_G(\mathcal{T}_i)\subsetneq \Stab_{G}(\mathcal{T})$. In particular, Proposition \ref{existence criterion} ensures the existence of an irreducible representation $\sigma$ of $\Aut_{G}(C)\cong \Stab_G(\mathcal{T})/\Fix_G(\mathcal{T})$ without non-zero $p_{\Fix_G(\mathcal{T})}(\Fix_G(\mathcal{T}_i))$-invariant vector and therefore without non-zero $p_{\Fix_G(\mathcal{T})}(\Fix_G(\mathcal{R}))$-invariant vector for every subtree $\mathcal{R}\in \mathcal{S}_{V_1}\lb 0\rb$ such that $\mathcal{R}\subseteq \mathcal{T}$.

	If $\mathcal{T}\in \mathfrak{T}_{V_1}\lb l\rb$ for some $l\geq 2$, every subtree $\mathcal{R}\in \mathcal{S}_{V_1}\lb l-1\rb$ is such that $Q_{\mathcal{R}}\not=\es$. Let $\{\mathcal{T}_1,..., \mathcal{T}_d\}$ be the set of subtrees $\mathcal{R}$ of $\mathcal{S}_{V_1}\lb l-1\rb$ such that $\mathcal{R}\subseteq \mathcal{T}$ and notice that $Q_{\mathcal{T}_i}\subsetneq Q_{\mathcal{T}}$ $\forall i$. Furthermore, notice the elements of  $\Stab_G(\mathcal{T})$ permutes the elements of $Q_{\mathcal{T}}$ and therefore the elements of $\{\mathcal{T}_1,..., \mathcal{T}_d\}$. On the other hand, for every $i,j\in \{1,...,d\}$ with $i\not= j$, we have $\mathcal{T}_i\cup \mathcal{T}_j=\mathcal{T}$. In particular, Proposition \ref{existence criterion} ensures the existence of an irreducible representation $\sigma$ of $\Aut_{G}(C)\cong \Stab_G(\mathcal{T})/\Fix_G(\mathcal{T})$ without non-zero $p_{\mathcal{T}}(\Fix_G(\mathcal{T}_i))$-invariant vectors and the result follows.
	\end{proof}
	\newpage
	
	\section{Universal groups of certain right-angled buildings}\label{Application to universal groups of right-angled buildings}
	The purpose of this chapter is to prove that the universal groups of certain semi-regular right-angled buildings embeds as closed subgroups of the group $\Aut(T)^+$ of type-preserving automorphisms of a locally finite tree and that those subgroups satisfy the hypotheses of Theorem \ref{Theorem pour IPV1} if the prescribed local action is $2$-transitive on panels. In particular, for every of such group we obtain a generic filtration that factorizes$^+$ at all depth and the machinery developed in Chapters \ref{chaptitre intro} and \ref{section generalisation of Ol'shanskii machinery} applies. 
	\subsection{Preliminaries}\label{section application right angle preliminaire}
	In this document, buildings are realised as $W$-metric space for a given Coxeter system $(W,I)$. We refer to \cite{AbramenkoBrown2008} more details on that point of view. We recall that a \textbf{right-angled building} is a building associated with a right-angled Coxeter system that is a Coxeter system with set of generators $I$ such that the order $m_{i,j}$ of $ij$ is either $2$ or $\infty$ for every $i,j\in I$ with $i\not=j$. We denote by $\Aut(\Delta)$ the group of type-preserving automorphisms of the building $\Delta$, by $\delta$ its $W$-metric and by $\Ch(\Delta)$ its set of chambers. We equip $\Ch(\Delta)$ with the natural metric $d_\Delta$ that is $d_\Delta(c,d)$ correspond to the length of a minimal gallery containing both $c$ and $d$. For every subset $J\subseteq I$ we denote by $W_J$ the subgroup of $W$ generated by $J$ and we define the \textbf{residue} of type $J$ containing a chamber $c\in \Ch(\Delta)$ by $$\mathcal{R}_J(c)=\{d\in \Ch(\Delta)\lvert \delta(c,d)\in W_J\}.$$ A residue of type $\{i\}$ is called an $i$-\textbf{panel}. One of the fundamental features of buildings is the existence of combinatorial projections between residues. We refer to \cite{Tits1974} for more details about this notion that we now recall. Given a chamber $c\in \Ch(\Delta)$ and a residue $\mathcal{R}$, the projection $\proj_\mathcal{R}(c)$ of $c$ on the residue $\mathcal{R}$ be the unique chamber $d\in \Ch(\mathcal{R})$ that is closest to $c$ for the chamber metric $d_\Delta$. We also recall that for any two residues $\mathcal{R}$ and $\mathcal{R}'$ of $\Delta$, the set
	$$\{\proj_\mathcal{R}(c)\lvert c\in \Ch(\mathcal{R}')\}$$
	is the chamber-set of a residue contained in $\mathcal{R}$. Furthermore, if $\Delta$ is a right-angled building and if $\mathcal{R}$ is a residue of $\Delta$, this notion of projections on $\mathcal{R}$ allows one to partition the buildings into convex chamber sets called wings. To be more precise, for any subset $J\subseteq I$ and any chamber $c\in \Ch(\Delta)$ of a right-angled building $\Delta$, the \textbf{$J$-wing} containing the chamber $c$ is the set $$X_J(c)=\{d\in \Ch(\Delta)\lvert \proj_{\mathcal{R}_J(c)}(d)=c\}.$$
 	If $J = \{i\}$ is a singleton, we write $X_i(c)$ and call it the \textbf{$i$-wing} of $c$. For right-angled buildings, it is shown in \cite{Caprace2014}, that  wings are convex chamber sets and that for every $J$-residue $\mathcal{R}$, $\Ch(\Delta)$ is partitioned by the $J$-wings $X_J(c)$ with $c\in \mathcal{R}$. We also recall the following few results about wings that will be used later in this chapter to prove that the group of automorphisms of certain semi-regular right-angled buildings $\Delta$ can be `nicely' embedded as a closed subgroup of $\Aut(T)^+$ for some locally finite tree $T$. 
 	\begin{lemma}[{\cite[Lemma $3.1$]{Caprace2014}}]\label{la J wing est l'intersection des j wing wesh wesh wesh}
 		Let $\Delta$ be a right-angled building of type $(W,I)$, let $J\subseteq I$ be a non-empty set and let $c\in \Ch(\Delta)$. Then the $J$-wing $X_J(c)$ containing $c$ satisfies the following properties:
 		\begin{enumerate}
 			\item $X_J(c)=\bigcap_{i\in J}X_i(c)$. 
 			\item $X_J(c) = X_J(c')$ for all $c'\in X_J (c) \cap \mathcal{R}_{J\cup J^\perp}(c)$.
 		\end{enumerate}
 	\end{lemma}
	\begin{lemma}\label{residues of different types intersects in a single chamber}
		Let $\Delta$ be a right-angled building of type $(W,I)$, let $J,J'\subseteq I$ be two disjoint subsets and let $c\in \Ch(\Delta)$. Then $\mathcal{R}_{J}(c)\subseteq X_{J'}(c)$.  
	\end{lemma}
	\begin{proof}
		The result follows directly from the fact the two residues contain $c$ and that the intersection of a $J$-residue and a $J'$-residue is a $J\cap J'$-residue. In particular, this proves that $\proj_{\mathcal{R}_{J'}(c)}(\mathcal{R}_{J}(c))=\{c\}$ and therefore that $\mathcal{R}_{J}(c)\subseteq X_{J'}(c)$. 
	\end{proof}
	\begin{lemma}[{\cite[Lemma $3.4$]{Caprace2014}}]\label{the inclusion of wings}
		Let $\Delta$ be a right-angled building of type $(W,I)$, let $i,i'\in I$ be such that $m_{i,j}=\infty$ and let $c,c'\in \Ch(\Delta)$ be such that $c'\in X_i(c)$ but $c\not\in X_{i'}(c')$. Then, we have $X_{i'}(c')\subseteq X_i(c)$. 
	\end{lemma}
	
	An other feature of combinatorial projections is given by the relation of parallelism. Two residues $\mathcal{R}$ and $\mathcal{R}'$ in a building $\Delta$ are said to be \tg{parallel} if $\proj_{\mathcal{R}}(\mathcal{R}')=\mathcal{R}$ and $\proj_{\mathcal{R}'}(\mathcal{R})=\mathcal{R}'$. Notice that the chamber sets of parallel residues are in bijection under the respective projection maps and two parallel residues have the same rank. Caprace showed that the relation of parallelism has a particular flavour in right-angled buildings.
	\begin{lemma}[{\cite[Corollary $2.9$]{Caprace2014}}]
		Let $\Delta$ be a right-angled building. Then the relation of parallelism of residues is an equivalence relation.
	\end{lemma}
	\noindent An equivalence class of $i$-panel in a right-angled building $\Delta$ is called an $i$\tg{-wall-residue} of $\Delta$. 
	\begin{lemma}[{\cite[Proposition $2.8$]{Caprace2014}}]
		Let $\Delta$ be a right-angled building of type $(W,I)$ and let $J\subseteq I$. Then two $J$-residues $\mathcal{R}$ and $\mathcal{R}'$ are parallel if and only if they are both contained in a common $J\cup J^\perp$-residue where $J^\perp=\{i\in I\lvert ij=ji\qq \forall j\in J\}$.
	\end{lemma}
	
	We now recall the notion of universal groups of semi-regular right-angled buildings introduced in \cite{Universal2018}. Let $(W,I)$ be a right-angled Coxeter system and let $(q_i)_{i\in I}$ be a tuple of cardinals with bigger than $2$. Haglund and Paulin have proved the existence and uniqueness up to isomorphism of a right-angled building $\Delta$ of type $(W,I)$ such that each $i$-panel of $\Delta$ contains $q_i$ chambers. We refer to \cite[Theorem 5.1]{Davis1994} for the existence and to \cite[Proposition 1.2]{Haglund2003} for the uniqueness. Such a building is called a \tg{semi-regular building of prescribed thickness} $(q_i)_{i\in I}$ and it is locally finite if each cardinal $q_i$ with $i\in I$ is finite. In \cite{Universal2018}, Tom De Medts, Ana C. Silva and Koen Struyve have generalised the concept of universal groups of regular trees introduced by Burger and Mozes to the more general setting of locally finite semi-regular right-angled buildings. Their definition requires the notion of a legal coloring that we now recall. Let $\Delta$ be a semi-regular right-angled building of type $(W, I)$ and prescribed thickness $(q_i)_{i\in I}$. 
	\begin{definition}
		For each $i\in I$, let $Y_i$ be a set of cardinality $q_i$. We will refer to $Y_i$ as the set of $i$-colors. A \tg{set of legal coloring} of $\Delta$ is a tuple $(h_i)_{i\in I}$ of maps $$h_i:\Ch(\Delta) \rightarrow Y_i$$ such that:
		\begin{enumerate}
			\item for every $i$-panel $\tau$, $\restr{h_i}{\Ch(\tau)}:\Ch(\tau) \rightarrow Y_i$ is a bijection.
			\item for every $(I-\{i\})$-residue $\mathcal{R}$ and for every $c, c' \in \Ch(\mathcal{R})$, $h_i (c) = h_i (c' )$.
		\end{enumerate}
	\end{definition}
	\begin{definition}
		For each $i\in I$, let $G_i\leq \Sym(Y_i)$ be a transitive permutation group and let $(h_i)_{i\in I}$ be a set of legal coloring of $\Delta$. Then the universal group $\mathcal{U}((h_i,G_i)_{i\in I})$ of $\Delta$ with respect to the set of legal coloring $(h_i)_{i\in I}$ and the groups $(G_i)_{i\in I}$ is the subgroup of $\Aut(\Delta)$ defined by
		\begin{equation*}
		\{g\in \Aut(\Delta)\lvert (\restriction{h_i}{g\tau_i})\circ g\circ (\restriction{h_i}{\tau_i})^{-1}\in G_i \qq \forall i\in I \mbox{ and every }i-\mbox{panel }\tau_i \mbox{ of }\Delta\}
		\end{equation*}
	\end{definition}
	\noindent Those groups satisfy a particular property that we now introduce.
	\begin{definition}
		Let $\Delta$ be a right-angled buildings of type $(W,I)$ and let $J\subseteq I$. A subgroup $G\leq \Aut(\Delta)$ is said to satisfy the property ${\rm IP}_{J}$ if for all $J\cup J^\perp$-residue $\mathcal{R}$ of $\Delta$ we have 
		\begin{equation*}\tag{${\rm IP}_{J}$}\label{IPJ}
		\Fix_G(\mathcal{R})= \prod_{c\in \mathcal{R}} \Fix_G(V_{J}(c))
		\end{equation*} 
		where $V_{J}(c)=\{ d\in \Ch(\Delta)\lvert \proj_{\mathcal{R}}(d)\not=c\}$ is the complement of the $J$-wing containing $c$ and where $\Fix_G(\mathcal{R})=\{g\in G \lvert gc=c\qq \forall c\in \Ch(\mathcal{R})\}$.
	\end{definition}
	The following proposition ensures that every universal group of right-angled buildings satisfies the property ${\rm IP}_{\{i\}}$ for every $i\in I$. Furthermore, Proposition \ref{independence IPJ for universal} below ensures that they also satisfy the property ${\rm IP}_{J}$ for every finite set $J\subseteq I$ such that $\{i\}\cup \{i\}^\perp=J$ $\forall i\in J$. 
	\begin{proposition}[{\cite[Proposition $3.16$]{Universal2018}}]\label{universal satify IP i de larticle universal}
		Let $G$ be a universal group of a locally finite semi-regular right-angled building $\Delta$. Then $G$ satisfies the property ${\rm IP}_{\{i\}}$ for every $i\in I$. Furthermore,  $\forall i\in I$, $\forall g\in \Fix_G(\mathcal{R}_{\{i\}\cup \{i\}^\perp}(c))$, and $\forall c\in \Ch(\Delta)$ the type preserving automorphism 
		\begin{equation*}
		g^c : \Ch(\Delta)\rightarrow \Ch(\Delta): x \mapsto \begin{cases}
		\qq g x\qq&\mbox{ if }x \in X_{i}(c) \\
		\qq x \qq &\mbox{ if }x\in V_{i}(c)
		\end{cases} 
		\end{equation*}
		is an element of $G$.  
	\end{proposition}
	\begin{proposition}\label{independence IPJ for universal}
		Let $G\leq \Aut(\Delta)$ be a universal group of a locally finite semi-regular right-angled building $\Delta$. Then, $G$ satisfies the property ${\rm IP}_{J}$ for every finite set $J\subseteq I$ such that $\{i\}\cup \{i\}^\perp=J$  $\forall i\in J$. 
	\end{proposition}
	\begin{proof}
		If $\modu{J}=1$, the results follows directly from Proposition \ref{universal satify IP i de larticle universal}. Suppose therefore that $\modu{J}\geq 2$ and let us show that $$\Fix_G(\mathcal{R})= \prod_{c\in \mathcal{R}} \Fix_G(V_{J}(c))$$ for every $J$-residue $\mathcal{R}$. First, notice that $\Fix_G(V_{J}(c))$ is a subgroup of $\Fix_G(\mathcal{R})$ for every $c\in \mathcal{R}$. Lemma \ref{la J wing est l'intersection des j wing wesh wesh wesh} ensures that $X_J(c)=\bigcap_{i\in J}X_i(c)$ for every $c\in \Ch(\mathcal{R})$. In particular, taking the complement, we obtain that $V_J(c)=\bigcup_{i\in J}V_i(c)$ and therefore that $\Fix_G(V_J(c))=\bigcap_{i\in J}\Fix_G(V_i(c))$. Let $i\in J$ and remember from Proposition \ref{universal satify IP i de larticle universal} that $G$ satisfies ${\rm IP}_{\{i\}}$ which implies that $\Fix_G(V_i(c))\subseteq \Fix_G(\mathcal{R}_{\{i\}\cup \{i\}^{\perp}}(c))$. Furthermore, since by the hypothesis $J=\{i\}\cup \{i\}^\perp$, we obtain that $\mathcal{R}=\mathcal{R}_{\{i\}\cup \{i\}^{\perp}}(c)$ and therefore that $\Fix_G(V_J(c))\subseteq \Fix_G(\mathcal{R})$. Notice that for every two distinct $c,d\in \Ch(\mathcal{R})$, the supports of the elements of $\Fix_G(V_i(c))$ and $\Fix_G(V_i(d))$ are disjoint from one another which proves that $\prod_{c\in \mathcal{R}} \Fix_G(V_{J}(c))$ is a well defined subgroup of $G$. The above discussion proves that $\prod_{c\in \mathcal{R}} \Fix_G(V_{J}(c))\subseteq\Fix_G(\mathcal{R})$. To prove the other inclusion, let $g\in \Fix_G(\mathcal{R})$ and let $J=\{i_1,...,i_n\}$. For any $i\in J$ and $c\in \Ch(\mathcal{R})$, let $\tau_i(c)$ be the $i$-panel containing $c$ in $\mathcal{R}$. Let us fix some chamber $c\in \Ch(\mathcal{R})$ and let 
		\begin{equation*}
		g_1^d : \Ch(\Delta)\rightarrow \Ch(\Delta): x \mapsto \begin{cases}
		\qq g x \qq&\mbox{ if }x \in X_{i_1}(d) \\
		\qq x \qq &\mbox{ if }x\in V_{i_1}(d)
		\end{cases} 
		\end{equation*}
		for every $d\in \tau_{i_1}(c)$.  Proposition \ref{universal satify IP i de larticle universal} ensures that $g^d_1\in G$ $\forall d\in \tau_{i_1}(c)$ and that $g=\prod_{d\in \tau_{i_1}(c)}g^d_1$. On the other hand, for every $d\in \tau_{i_1}(c)$, there exists a unique $i_2$-panel $\tau_{i_2}(d)$ such that $d\in \tau_{i_2}(d)$. Since $g^d_1\in \Fix_G(V_{i_1}(d))\subseteq \Fix_G(\mathcal{R})$, we can repeat the above argument and we obtain that $g^d_1 = \prod_{d'\in \tau_{i_2}(d)}g^{d'}_2$ where 
		\begin{equation*}
		g^{d'}_2 : \Ch(\Delta)\rightarrow \Ch(\Delta): x \mapsto \begin{cases}
		\qq g^d_1 x \qq&\mbox{ if }x \in X_{i_2}(d') \\
		\qq x \qq &\mbox{ if }x\in V_{i_2}(d')
		\end{cases} 
		\end{equation*}
		for every $d'\in \tau_{i_2}(d)$. Just as before, Proposition \ref{universal satify IP i de larticle universal} ensures that $g^{d'}_2\in G$ for every $d'\in \tau_{i_2}(d)$. On the other hand, $\tau_{i_1}(c)\cap \tau_{i_2}(d)=\tau_{i_1}(d)\cap \tau_{i_2}(d)=\{d\}$. In particular, this implies that $\proj_{\tau_{i_1}(c)}(\tau_{i_2}(d))=\{d\}$ and $d'\in X_{i_1}(d)$. Since $\mathcal{R}$ is an $\{i_1\}\cup \{i_1\}^\perp$-residue Lemma \ref{la J wing est l'intersection des j wing wesh wesh wesh} ensures that $X_{i_1}(d)=X_{i_1}(d')$. This proves that $g_2^{d'}$ has support in $X_{i_1}(d')\cap X_{i_2}(d')=X_{\{i_1, i_2\}}(d')$ and therefore that $g_{d'}^2\in \Fix_G(V_{\{i_1,i_2\}}(d'))$. Proceeding iteratively, for any of the constructed $g^d_k$ with $k\in \{2,...,n-1\}$, we set
		\begin{equation*}
		g^{d'}_{k+1} : \Ch(\Delta)\rightarrow \Ch(\Delta): x \mapsto \begin{cases}
		\qq g^d_k x \qq&\mbox{ if }x \in X_{i_{k+1}}(d') \\
		\qq x \qq &\mbox{ if }x\in V_{i_{k+1}}(d')
		\end{cases} 
		\end{equation*}
		for every $d'\in \tau_{i_{k+1}}(d)$. Once more, Proposition \ref{universal satify IP i de larticle universal} ensures that $g^{d'}_{k+1}\in G$  $\forall d'\in \tau_{i_{k+1}}(d)$ and that $g^d_k=\prod_{d'\in \tau_{i_{k+1}}(d)}g^{d'}_{k+1}$. On the other hand, $\tau_{i_l}(d)\cap \tau_{i_{k+1}}(d)=\{d\}$ for every $l=1,...,k$ which implies that $\proj_{\tau_{i_l}(d)}(\tau_{i_{k+1}}(d))=\{d\}$ and therefore that $d'\in X_{i_{l}}(d)$. Since $\mathcal{R}$ is an $\{i_l\}\cup \{i_l\}^\perp$-residue, Lemma \ref{la J wing est l'intersection des j wing wesh wesh wesh} ensures that $X_{i_l}(d)=X_{i_l}(d')$. Finally, since $g^d_k$ has support in $X_{i_1}(d)\cap X_{i_2}(d)\cap ...\cap X_{i_k}(d)$ this proves that $g_{k+1}^{d'}$ has support in $X_{i_1}(d') \cap ... X_{i_k}(d')\cap X_{i_{k+1}}(d')=X_{\{i_1,...,i_{k+1}\}}(d')$ and therefore that $g_{d'}\in \Fix_G(V_{\{i_1,i_2,..., i_{k+1}\}}(d'))$. 
	\end{proof}

	\subsection{Groups of automorphisms of certain right-angled buildings as groups of automorphisms of trees.}
	
	The purpose of this section is to show that the group of type preserving automorphisms $\Aut(\Delta)$ of certain semi-regular right-angled buildings $\Delta$ can be realised as closed subgroups of the group $\Aut(T)^+$ of type-preserving automorphisms of a locally finite tree $T$ in such a way that the universal groups of those buildings embedded as closed subgroups $G\leq \Aut(T)^+$ satisfying the property \ref{IPV1}. This applies only to certain Coxeter types and motivates the following definition.
	\begin{definition}
		A right-angled Coxeter system $(W,I)$ is said to satisfy the hypothesis \ref{Hypothese cox star} if it is finitely generated and there exists $r\geq 2$ such that
		\begin{equation}\tag{$\star$}\label{Hypothese cox star}
		I=\bigsqcup_{k=1}^rI_k
		\end{equation}
		for some $I_k=\{i\}\cup \{i\}^\perp$  $\forall i\in I_k$ and $\forall k=1,...,r$.
	\end{definition} 
	\begin{remark}
	A right-angled Coxeter system satisfying the hypothesis \ref{Hypothese cox star} is isomorphic to a free product $W_1*W_2*...*W_r$ where each of the $W_k$ is a direct product of finitely many copies of the group $C_2$ of order $2$. In particular, $W$ is virtually free.
	\end{remark} 
	\noindent Let $(W,I)$ be a right-angled Coxeter system satisfying the hypothesis \ref{Hypothese cox star}, let $(q_i)_{i\in I}$ be a family of finite cardinals bigger than $2$ and let $\Delta$ be a semi-regular building of type $(W,I)$ and prescribed thickness $(q_i)_{i\in I}$. We associate a locally finite bipartite graph to $\Delta$ as follows. We let $V_0=\Ch(\Delta)$, we let $$V_1= \{ \mathcal{R}\lvert \mathcal{R} \mbox{ is an }I_k-\mbox{residue of }\Delta \mbox{ for some }k\in \{1,...,r\}\}$$ and we define $T$ as the bipartite graph with set of vertices $V(T)= V_0 \sqcup V_1$ and where a chamber $c\in V_0$ is adjacent to a residue $\mathcal{R}\in V_1$ if $c\in \mathcal{R}$. 
	\begin{lemma}\label{the tree gamma associated with Delta}
		The graph $T$ is a locally finite tree. 
	\end{lemma}
	\begin{proof}
		The graph $T$ is locally finite since each chamber is contained in finitely many residues and since each $I_k$-residue is finite (locally finite and spherical). The graph $T$ is path connected since every two chambers of $\Delta$ are connected by a gallery and since each such gallery corresponds naturally to a path in $T$. We now  show that $T$ does not contain any cycle. Suppose for a contradiction that there is a simple cycle in $T$, say
		\begin{equation*}
		c_1 - \mathcal{R}_1 - c_2 - ... - \mathcal{R}_n - c_1.
		\end{equation*}
		Since each chamber $c\in \Ch(\Delta)$ is contained in a unique residue $\mathcal{R}$ of type $I_t$ and since the cycle is simple, notice that $\mathcal{R}_1$ and $\mathcal{R}_n$ have different types. In particular, Lemma \ref{residues of different types intersects in a single chamber} ensures that $\mathcal{R}_n\subseteq X_{J_1}(c_1)$ where $J_1$ is the type of $\mathcal{R}_1$. On the other hand, as we show below, $\mathcal{R}_n\subseteq X_{J_1}(c_2)$. For now, we assume this inclusion that is $\mathcal{R}_n\subseteq X_{J_1}(c_1)\cap X_{J_1}(c_2)$ and we show this leads to a contradiction. Since the cycle is simple, we have that $c_1\not=c_2$. Furthermore, since $c_1,c_2\in\mathcal{R}_1$, there exists some $i\in J_1$ such that $c_1\not\in X_i(c_2)$. Hence, we have that $X_i(c_1)\cap X_i(c_2)=\es$ and therefore that $X_{J_1}(c_1)\cap X_{J_1}(c_2)=\es$. The desired contradiction follows from our inclusion.
		
		Now, let us prove that $\mathcal{R}_n\subseteq X_{J_1}(c_2)$. Following this purpose, we show by induction on $t$ that $X_{J_{t+1}}(c_{t+2})\subseteq X_{J_t}(c_{t+1})$ for every $t=1,..., n-2$ where $J_t$ is the type of $\mathcal{R}_t$. Since the cycle is simple notice that $\mathcal{R}_t$ and $\mathcal{R}_{t+1}$ have different types and that $\mathcal{R}_{t+1}\subseteq X_{J_t}(c_{t+1})$ for every $t=1,...,n-2$. On the other hand, since $c_{t+1}\not=c_{t+2}$, there exists some $i'\in J_{t+1}$ such that $c_{t+1}\not\in X_{i'}(c_{t+2})$. Notice for every $i\in J_t$ that $m_{i,i'}=\infty$ and that $c_{t+2}\in \mathcal{R}_{t+1}\subseteq X_{i}(c_{t+1})$. In particular, Lemma \ref{the inclusion of wings} implies that $X_{J_{t+1}}(c_{t+2})\subseteq X_{i'}(c_{t+2})\subseteq \bigcap_{i\in J_t}X_i(c_{t+1})= X_{J_t}(c_{t+1})$ which completes the induction. This proves as desired that $\mathcal{R}_n\subseteq X_{J_{n-1}}(c_{n})\subseteq ...\subseteq X_{J_1}(c_2)$.  
	\end{proof}

	Our next goal is to explicit an injective map $\alpha :\Aut(\Delta)\rightarrow\Aut(T)^+$ defining an homeomorphism on its image. To this end notice that any type-preserving automorphism $g\in \Aut(\Delta)$ is bijective on the set of chambers $\Ch(\Delta)$ but also on the $I_k$-residues of $\Delta$ for any fixed $k$. For every $g\in \Aut(\Delta)$ we define the map $\alpha(g): V(T)\rightarrow V(T)$ as follows: 
	\begin{itemize}
		\item If $v\in V_0$ then $v$ is a chamber $c\in \Ch(\Delta)$ and we define $\alpha(g)v=gc$.
		\item If $v\in V_1$ then $v$ is an $I_k$-residue $\mathcal{R}$ of $\Delta$ for some $k\in \{1,...,r\}$ and we define $\alpha(g)v= g\mathcal{R}$.
	\end{itemize} 
	The map $\alpha(g)$ is clearly defines a type preserving bijection on $V(T)$. In fact, $\alpha(g)$ is a tree automorphism of $T$ and $\alpha: \Aut(\Delta)\rightarrow \Aut(T)^+$ is a well defined group homomorphism  since  for every $g\in \Aut(\Delta)$, every residue $\mathcal{R}$ of $\Delta$ and every $ c\in \Ch(\Delta)$, we have that $c\in \mathcal{R}$ if and only if $gc\in g\mathcal{R}$.
	\begin{proposition}\label{the automrophism of buildings as automorphims of locally finite graphs}
		The map $\alpha :\Aut(\Delta)\rightarrow \Aut(T)^+$ is an injective group homomorphism; $\alpha(\Aut(\Delta))$ is a closed subgroup of $\Aut(T)^+$ and $\alpha$ defines an homeomorphism between $\Aut(\Delta)$ and $\alpha(\Aut(\Delta))$. 
	\end{proposition}
	\begin{proof}
		The homomorphism $\alpha$ is injective, since 
		\begin{equation*}
		\begin{split}
		\ker(\alpha)&=\{g\in \Aut(\Delta)\lvert \alpha(g)=1_{\Aut(T)}\}\\
		&\subseteq\{g\in \Aut(\Delta)\lvert gc=c\qq \forall c\in \Ch(\Delta)\}=\{1_{\Aut(\Delta)}\}.
		\end{split}
		\end{equation*}
		Remember that the sets $$U_{T}(F_{T})=\{g \in \Aut(T)^+\lvert g v=v \qq\forall v\in F_T\}$$ 
		where $F_T\subsetneq V_0$ is finite form a basis of open neighbourhoods of the identity in $\Aut(T)^+$. On the other hand, an element $h\in \Aut(T)^+$ belongs to $\Aut(T)^+-\alpha(\Aut(\Delta))$ if and only if there exists $i\in I$ and two $i$-adjacent chambers $c,d\in V_0$ such that $hc$ and $hd$ are not $i$-adjacent. In particular, for every such automorphism $h$, the set $hU_T(\{c,d\})$ is an open neighbourhood of $h$ in $\Aut(T)^+-\alpha(\Aut(\Delta))$. This proves that the complement of $\alpha(\Aut(\Delta))$ is an open set and therefore that $\alpha(\Aut(\Delta))$ is a closed subgroup of $\Aut(T)^+$.
		 
		Let $\Phi: \Ch(\Delta)\rightarrow V_0$ be the map sending a chamber of $\Delta$ to the corresponding vertex of $V_0\subseteq V(T)$ and remember that the sets 
		$$U_\Delta(F_\Delta)=\{g\in \Aut(\Delta)\lvert g c=c \qq \forall c\in F_\Delta\}$$ 
		where $F_\Delta\subseteq \Ch(\Delta)$ is finite form a basis of open neighbourhoods of the identity in $\Aut(\Delta)$.
		In particular, notice that $\alpha : \Aut(\Delta)\rightarrow \alpha (\Aut(\Delta))$ is continuous since for every finite subset $F_T\subsetneq V_0$ we have that $\alpha^{-1}(U_{T}(F_T)\cap \alpha(\Aut(\Delta))=U_\Delta(\Phi^{-1}(F_T))$. Finally, notice that  $\alpha : \Aut(\Delta)\rightarrow \alpha (\Aut(\Delta))$ is an open map since, for every finite set $F_\Delta\subsetneq\Ch(\Delta)$ we have that $\alpha(U_{\Delta}(F_\Delta))=U_\Delta(\Phi(F_\Delta))\cap \alpha(\Aut(\Delta))$.
	\end{proof}
	The following proposition shows that under this correspondence, the property \ref{IPV1} of groups of type-preserving automorphisms of trees is tightly related to the property \ref{IPJ} of groups of type-preserving automorphisms of right-angled buildings.  
	\begin{proposition}\label{lemme de on a IP IK alors IP V1}
		Let $G\leq \Aut(\Delta)$ be a closed subgroup satisfying the property ${\rm IP}_{I_k}$ for every $k=1,...,r$. Then, $\alpha(G)$ is a closed subgroup of $\Aut(T)^+$ satisfying the property ${\rm IP}_{V_1}$. 
	\end{proposition}
	\begin{proof}
		Proposition \ref{the automrophism of buildings as automorphims of locally finite graphs} ensures that $\alpha(G)$ is a closed subgroup of $\alpha(\Aut(\Delta))$ and therefore of $\Aut(T)^+$. Let $\Phi: \Ch(\Delta)\rightarrow V_0$ be the map sending a chamber of $\Delta$ to the corresponding vertex of $V_0\subseteq V(T)$ and notice that:
		\begin{itemize}
			\item For every residue $\mathcal{R}\in V_1$, $\Phi(\mathcal{R})=\{v\in V_0\lvert v\in \mathcal{R}\}=B_T(\mathcal{R},1)\cap V_0$. In particular, we have that
			\begin{equation*}
			\begin{split}
			\alpha(\Fix_{G}(\mathcal{R}))&=\Fix_{\alpha(G)}(\Phi(\mathcal{R}))=\Fix_{\alpha(G)}(B_T(\mathcal{R},1)\cap V_0)\\
			&=\Fix_{\alpha(G)}(B_T(\mathcal{R},1)).
			\end{split}
			\end{equation*} 
			\item For every $k\in \{1,...,r\}$ and every chamber $c\in \Ch(\Delta)$ we have
			\begin{equation*}
			\begin{split}
			\Phi(\Ch(\Delta)-X_{I_k}(c))&=\Phi(\{d\in \Ch(\Delta)\lvert \proj_{\mathcal{R}_{I_k}(c)}(d)\not=c\})\\
			&=\{v\in V_0\lvert d_T(v,\mathcal{R}_{I_k}(c))\lneq d_T(v,c)\}\\
			&=T( \mathcal{R}_{I_k}(c),c)\cap V_0.
			\end{split}
			\end{equation*}
			In particular, we obtain that
			\begin{equation*}
			\begin{split}
			\alpha (\Fix_G(\Ch(\Delta)-X_{I_k}(c)))&=\Fix_{\alpha(G)}(\Phi(\Ch(\Delta)-X_{I_k}(c)))\\
			&=\Fix_{\alpha(G)}(T(\mathcal{R}_{I_k}(c), c)).
			\end{split}
			\end{equation*}
		\end{itemize}
	The results follows from the definitions of the properties ${\rm IP}_{I_k}$ and ${\rm IP}_{V_1}$. 
	\end{proof}
	
	\noindent 
	The proof of Theorem \ref{Theorem E} requires one last preliminary. Let $(W,I)$ be a right-angled Coxeter system. Let $(q_i)_{i\in I}$ be a set of finite cardinals bigger than $2$ and let $\Delta$ be the semi-regular building of type $(W,I)$ and prescribed thickness $(q_i)_{i\in I}$. Let $Y_i$ be a set of cardinal $q_i$ for every $i\in I$, let $(h_i)_{i\in I}$ be a set of legal colorings of $\Delta$ and let $G_i\leq \Sym(Y_i)$ be transitive on $Y_i$ for every $i\in I$. 
	\begin{proposition}\label{les groupes universelles sont unimodular}
		For every $J\subseteq I$, every $(g_{j})_{j\in J}\in \prod_{j\in J}^{}G_j$ and every $c\in \Ch(\Delta)$, there exits $g\in \mathcal{U}((h_i,G_i)_{i\in I})$ such that $g\mathcal{R}_J(c)=\mathcal{R}_{J}(c)$, $h_j\circ g=g_j\circ h_j$ for all $j\in J$ and $h_i\circ g=h_i$ for all $i\in I-J$.
	\end{proposition}
	\begin{proof}
		Let $(h_i')_{i\in I} $ be the colorings obtained from $(h_i)_{i\in I}$ by replacing $h_j$ by $g_j\circ h_j$ for all $j\in J$ and leaving the other colorings unchanged. Notice that $h_j'$ is still a legal coloring of $G$ for every $j\in J$ since for all $I-\{j\}$-residue $\mathcal{R}$ and for all $d,d'\in\mathcal{R}$ we have $h_j'(c)=g_j\circ h_j(d)=g_j\circ h_j(d')=h_j'(d')$. Now, let $c'$ be the chamber of $\mathcal{R}_{J}(c)$ with colors $h_j(c')=g_j\circ h_j(c)$ for every $j\in J$. Since $h_i'(c')=h_i(c)$ for every $i\in I$, \cite[Proposition $2.44$]{Universal2018} ensures the existence of an automorphism $g\in \Aut(\Delta)$ mapping $c$ to $c'$ and such that $h_i\circ g=h_i'$ for all $i\in I$. Since $c'\in \mathcal{R}_J(c)$ notice that $g$ stabilizes $\mathcal{R}_J(c)$. Finally, notice that $g$ acts locally as the identity on $i$-panels for all $i\in I-J$ and as $g_j$ on $j$-panels for all $j\in J$. Hence, $g$ is as desired.
	\end{proof}
	The following result proves Theorem \ref{Theorem E}.
	\begin{theorem}\label{theorem pour les right angled buildings}
		 Suppose that $(W,I)$ satisfies the hypothesis \ref{Hypothese cox star}. Then, the group $\alpha(\mathcal{U}((h_i,G_i)_{i\in I}))$ is a closed subgroup of $\Aut(T)^+$ satisfying the property \ref{IPV1}. Furthermore, if $G_i$ is \textit{2}-transitive on $Y_i$ for every $i\in I$, the group automorphism of tree $\alpha(\mathcal{U}((h_i,G_i)_{i\in I}))$ satisfies the hypothesis \ref{Hypothese HV1} and is unimodular.
	\end{theorem}  
	\begin{proof}
		Let $G = \mathcal{U}((h_i,G_i)_{i\in I})$. The first part of the theorem follows directly from Proposition \ref{independence IPJ for universal} and Proposition \ref{lemme de on a IP IK alors IP V1}. Now, suppose that $G_i$ is \textit{2}-transitive on $Y_i$ for every $i\in I$, let $\mathfrak{T}_{V_1}$ be the family of subtrees defined on page \pageref{page Tv1} and let $\mathcal{T}, \mathcal{T}'\in \mathfrak{T}_{V_1}$. If $\mathcal{T}\subseteq\mathcal{T}'$ we clearly have that $\Fix_{\alpha(G)}(\mathcal{T}')\subseteq \Fix_{\alpha(G)}(\mathcal{T})$. On the other hand, if $\mathcal{T}\not \subseteq \mathcal{T}'$, let $v_\mathcal{T}$ a vertex of $V_1\cap \mathcal{T}$ that is at maximal distance from $\mathcal{T}'$ and let $w_\mathcal{T}\in B_T(v_\mathcal{T},1)-\{v_\mathcal{T}\}$ be such that  $\mathcal{T}'\subseteq T(w_\mathcal{T}, v_\mathcal{T})$.  In particular, we have that $\Fix_{\alpha(G)}(T(w_\mathcal{T},v_\mathcal{T}))\leq \Fix_{\alpha(G)}(\mathcal{T}')$. Let $I_k$ denote the type of $v_\mathcal{T}$ seen as a residue of $\Delta$. and let $j\in I_k$. Consider an element $g_j\in G_i$ that is not trivial and such that $g_j\circ h_j(w_{\mathcal{T}})=h_j(w_\mathcal{T})$ and let $g_i={\rm id}_{Y_i}$ for every $i\in I-\{j\}$. Proposition \ref{les groupes universelles sont unimodular} ensures the existence of an element $g\in \mathcal{U}((h_i,G_i)_{i\in I})$ such that $g\mathcal{R}_J(c)=\mathcal{R}_{J}(c)$ and $h_i\circ g=g_i\circ h_i$ for every $i\in I$. Now, notice that there exists a unique vertex $v'_{\mathcal{
		T}}$ that is adjacent to $w_{\mathcal{T}}$ and such that $\mathcal{T}'\subseteq T(v'_\mathcal{T},w_\mathcal{T}) \cup \{w_\mathcal{T}\}$. Now, notice that $v'_\mathcal{T}$ is a residue of type $I_{k'}$ with $I_{k'}\not = I_k$ and we realise from the definition that $g$ fixes every chamber of $v'_{\mathcal{T}}$ or equivalently that $\alpha(g)$ fixes $B_T(v'_{\mathcal{T}},1)$ pointwise. Proposition \ref{lemme de on a IP IK alors IP V1} implies the existence of an element $h\in \Fix_{\alpha(G)}(B_T(v'_\mathcal{T}),1)\cap \Fix_G( T(v'_{\mathcal{T}},w_{\mathcal{T}}))$ such that $hv=\alpha(g)v$ for every $v\in T(w_{\mathcal{T}},v'_{\mathcal{T}})$. In particular, we have that $h\in \Fix_{\alpha(G)}(\mathcal{T}')$ but $h\not \in \Fix_{\alpha(G)}(\mathcal{T})$ since by the definition $g$ does not fix every chamber of $v_{\mathcal{T}}$. This proves as desired that $\Fix_{\alpha(G)}(\mathcal{T}')\subsetneq \Fix_{\alpha(G)}(\mathcal{T})$.
		
		To prove that $\alpha(G)$ is unimodular, we apply \cite[Corollary 5]{BaumgartnerBertrandWillis2007}. This result ensures that a group $G$ which acts $\delta$-$2$-transitively on the set of chambers of a locally finite building is unimodular. Choose a chamber $c\in \Ch(\Delta)$. Since $G$ is transitive on the chambers of $\Delta$, we need to show for any two chambers $d_1,d_2\in \Ch(\Delta)$ that are $W$-equidistant from $c$ that there exists an element $g\in \Fix_G(c)$ such that $gd_1=d_2$. First of all, notice from the hypothesis \ref{Hypothese cox star} and the solution of the word problem in Coxeter groups that every $w\in W$ admits a unique decomposition $w=w_1...w_n$ with $w_t\in W_{I_{k_t}}-\{1_W\}$ such that $I_{k_t}\not= I_{k_{t+1}}$ for every $t$. Suppose that $d_1,d_2\in \Ch(\Delta)$ have $W$-distance $w_1...w_n$ from $c$ with $w_t\in W_{I_{k_t}}$ and let us show the existence of $g$ by induction on $n$. If $n=1$, the result follows from Proposition \ref{theorem pour les right angled buildings} since for every $j\in I_{k_1}$, there exists an element $g_j\in G_j$ such that $g_j\circ h_j(c)=h_j(c)$ and $g_j\circ h_j(d_1)=h_j(d_2)$. Hence, there exists an element $g\in G$ such that $h_j\circ g_j=g\circ h_j$ for every $j\in J$. In particular, $gc=c$, $gd_1=d_2$ and the result follows. If $n\geq2$, we let $d_s'=\proj_{\mathcal{R}_{I_{k_n}}(d_s)}(c)$ and notice that $\delta(c,d_s')=w_1...w_{n-1}$. Our induction hypothesis therefore ensures the existence of a $g'\in  G$ such that $g'c=c$ and $g'd_1'=d_2'$. Now, notice that $\delta(d_2',d_2)=w_n$ and $\delta(d_2',g'd_1)=\delta (g'd_1',g'd_1)=\delta (d_1',d_1)=w_n$. In particular, Proposition \ref{theorem pour les right angled buildings} ensures the existences of an element $h\in G$ such that $hd=d$ for every $d\in \mathcal{R}_{I_{k_{(n-1)}}}(d_2')$ and $hg'd_1=d_2$. Since $G$ satisfies the property ${\rm IP}_{I_{k_{(n-1)}}}$ by Proposition \ref{independence IPJ for universal}, this implies the existence of an element $h'\in \Fix_{G}(V_{I_{k_{(n-1)}}}(d_2'))$ such that $h'b=hb$ for every $b\in X_{I_{k_{(n-1)}}}(d_2')$. Since $c\in V_{I_{k_{(n-1)}}}(d_2')$, the automorphism $h'g'\in G$ satisfies that $h'g'c=c$ and that $h'g'd_1=d_2$. The result follows.
	\end{proof}
	In particular, if $\mathcal{U}((h_i,G_i)_{i\in I})$ is non-discrete, Theorem \ref{theorem Ipv1 letter} applies to $\alpha(G)$ and Theorem \ref{la version paki du theorem de classification} provides a bijective correspondence between the equivalence classes of irreducible representations of $\alpha(G)$ at depth $l\geq 1$ with seed $C\in \mathcal{F}_{\mathcal{S}_{V_1}}$ and the $\mathcal{S}_{V_1}$-standard representations of $\Aut_{\alpha(G)}(C)$. We recall further that an existence criterion for those representations was given in Section \ref{existence for IPV1}. 
	
	Since $\alpha$ is a homeomorphism on its image, the same holds for the representations of $G$. Notice that under the correspondence given by $\alpha^{-1}$, the generic filtration $\mathcal{S}_{V_1}$ describes a generic filtration $\mathcal{S}_\Delta$ of $G$ that factorizes$^+$ at all depths $l\geq 1$ and which can be interpreted as follows. We explicit this correspondence bellow. Let $\delta$ denote the $W$-distance of $\Delta$ and let consider the set $$\mathcal{R}'(c)=\{d\in \Ch(\Delta)\lvert \delta(c,d)=w\mbox{ s.t. } \exists k\in \{1,...,r\} \mbox{ for which }w\in W_{I_{k}}\}$$ 
	for every chamber $c\in \Ch(\Delta)$. By use of the correspondence $\Phi: \Ch(\Delta) \rightarrow V_0$ between chambers of $\Delta$ and vertices of $V_0$, notice, for every $c\in V_0$, that  $\Phi^{-1}(B_T(\Phi(c),2 )\cap V_0))=\mathcal{R}'(c)$. We define a family $\mathcal{T}_\Delta$ of subsets of $\Ch(\Delta)$ as follows:
	\begin{enumerate}
		\item $\mathcal{T}_\Delta\lb 0\rb =\{\mathcal{R}_{I_k}(c)\lvert k\in \{1,...,r\}, c\in \Ch(\Delta)\}$.
		\item For every $l$ such that $l\geq 0$, we define iteratively:
		\begin{equation*}
		\begin{split}
		\mathcal{T}_\Delta\lb l+1\rb =\{\mathcal{R}\subseteq \Ch(\Delta)\lvert \exists \mathcal{Q}\in \mathcal{T}_\Delta\lb l\rb, \qq \exists c\in &\mathcal{Q}\mbox{ s.t. }\mathcal{R}'(c)\not\subseteq \mathcal{Q}\\
		&\mbox{ and }\mathcal{R}=\mathcal{Q}\cup\mathcal{R}'(c)) \}.
		\end{split}
		\end{equation*}
		\item We set $\mathcal{T}_\Delta= \bigsqcup_{l\in \N}\mathcal{T}_\Delta\lb l\rb$.
	\end{enumerate}
	It is quite easy to realise that $\mathcal{S}_\Delta=\{\Fix_G(\mathcal{R})\lvert\mathcal{R}\in \mathfrak{T}_\Delta \}$ is the generic filtration of $G$ corresponding to $\mathcal{S}_{V_1}$ under the correspondence given by $\alpha^{-1}$ and that 
	$$\mathcal{S}_\Delta\lb l\rb=\alpha^{-1}\big(\mathcal{S}_{V_1}\lb l\rb \big)=\{\Fix_G(\mathcal{R})\lvert \mathcal{R}\in \mathcal{T}_\Delta\lb l\rb\}.$$
	  
	\clearpage
	\newpage
\addcontentsline{toc}{section}{Bibliography}
\bibliographystyle{alpha}
\bibliography{bibliography}	
\end{document}